\documentclass[10pt]{report}
\usepackage[english]{babel}
\usepackage[T1]{fontenc}
\usepackage[latin1]{inputenc}
\usepackage{hyperref}
\usepackage{graphicx}
\usepackage{times}
\usepackage[all]{xy}
\usepackage{amsthm,amsfonts,amssymb,amsmath}
\newtheorem{satz}{Satz}[section]
\newtheorem{teo}[satz]{Theorem}
\newtheorem{cor}[satz]{Corollary}
\newtheorem{conj}[satz]{Conjecture}
\newtheorem{prop}[satz]{Proposition}
\newtheorem{lemma}[satz]{Lemma}
\theoremstyle{definition}
\newtheorem{de}[satz]{Definition}
\newtheorem{rem}[satz]{Remark}
\newtheorem{ex}[satz]{Example}
\newtheorem{notation}[satz]{Notation}
\newtheorem{de-teo}[satz]{Definition-Theorem}
\newtheorem{say}[satz]{}
\newcommand\N{\mathbb N}
\newcommand\Q{\mathbb Q}
\newcommand\Z{\mathbb Z}
\newcommand\R{\mathbb R}
\newcommand\C{\mathbb C}
\DeclareMathOperator{\Star}{Star}
\DeclareMathOperator{\Spec}{Spec}
\DeclareMathOperator{\Hom}{Hom}
\DeclareMathOperator{\Pic}{Pic}

\DeclareMathOperator{\Conv}{Conv}
\DeclareMathOperator{\Cone}{Cone}
\DeclareMathOperator{\codim}{codim}
\DeclareMathOperator{\Proj}{Proj}
\DeclareMathOperator{\Sym}{Sym}
\DeclareMathOperator{\mult}{mult}
\DeclareMathOperator{\codeg}{codeg}
\DeclareMathOperator{\interior}{int}
\DeclareMathOperator{\Vol}{Vol}
\DeclareMathOperator{\Cayley}{Cayley}
\DeclareMathOperator{\Div}{Div}
\DeclareMathOperator{\Span}{Span}

\begin{document}

\begin{center}

\textbf{TORIC BIRATIONAL GEOMETRY AND ITS \\ APPLICATIONS TO LATTICE POLYTOPES}
\end{center}

\vspace{0,1 cm}

\begin{center}
Douglas Mons\^ores de Melo Santos

monsores@impa.br
\end{center}

\vspace{1 cm}
\begin{center}
\textbf{ Abstract}
\end{center}
\vspace{0,3 cm}

Toric geometry provides a bridge between the theory of polytopes and algebraic geometry: one can associate to each lattice polytope  a polarized toric variety. 
In this thesis we explore this correspondence to classify smooth lattice polytopes having small degree,
extending a classification provided by Dickenstein, Di Rocco and Piene.
Our approach consists in interpreting the degree of a polytope as 
a geometric invariant of the corresponding polarized variety,
and then applying techniques from Adjunction Theory and Mori Theory. 

In the opposite direction, we use the combinatorics of  fans to describe the extremal rays of several cones of cycles of a projective toric variety $X$. One of these cones is the \emph{cone of moving curves}, that is the closure of the cone generated by classes of curves moving in a family that sweeps out $X$. This cone plays an important role in the problem of birational classification of projective varieties.    

\vspace{1,5 cm}

\textbf{Acknowledgements}: This work is my Ph.D. thesis. I thank my advisor Carolina Araujo, with whom I had the honor of working, for sharing her wisdom, patience and attention in our discussions. I also thank Victor Batyrev, Ana-Maria Castravet and Edilaine Nobili for their precious suggestions and discussions.

\vspace{1,5 cm}


\thispagestyle{empty}

\pagenumbering{arabic}
\tableofcontents

%%%%%%%%%%%%%%%%%%%%%%%%%%%%%%%%%%%%%%
%%%%%%%%%%%%%%%%%%%%%%%%%%%%%%%%%%%%%%
%%%%%%%%%%%%%%%%%%%%%%%%%%%%%%%%%%%%%%
%
%
%            CAPÍTULO 1
%
%
%%%%%%%%%%%%%%%%%%%%%%%%%%%%%%%%%%%%%%
%%%%%%%%%%%%%%%%%%%%%%%%%%%%%%%%%%%%%%
%%%%%%%%%%%%%%%%%%%%%%%%%%%%%%%%%%%%%%

\chapter{Introduction}

A \emph{toric variety} is a normal algebraic variety over $\C$ that contains an algebraic torus $T \simeq (\C^*)^n$ as anopen dense subset so that the natural action of $T$ on itself extends to an action $T \times X \to X$. Every $n$-dimensional toric variety gives rise to a certain collection $\Sigma_X$ of cones in $\R^n$ called a \emph{fan}. Many algebro-geometric properties of the variety $X$ can be translated to combinatorial properties of $\Sigma_X$ that are, in general, simpler to be studied. \\ 

Thanks to this combinatorial structure, toric varieties provide a ``testing ground'' for general conjectures. For instance, one of the greatest achievements in the problem of classification of complex projective manifolds is the Minimal Model Program (MMP). Roughly speaking, the goal of the MMP is perform successive birational transformations in a given complex algebraic variety to obtain a birational model which is ``as simple as possible''. In 1988, Mori established the MMP for 3-folds. With this result Mori earned the Fields medal in 1990. In the toric context, the MMP was established by Reid in 1983 for varieties of arbitrary dimension using the combinatorics of the associated fans. The MMP for toric varieties will play a key role in this work.

In this thesis we follow this philosophy: by exploring the combinatorics of the fan of a toric variety $X$ we prove some algebro-geometric properties of $X$. Going in the opposite direction, we solve a combinatorial problem by interpreting it as a problem in Algebraic Geometry. \\

The first problem we solve in this thesis concerns combinatorics of lattice polytopes. A \emph{polytope} $P \subset \R^n$ is the convex hull of a finite set of points. The \emph{dimension} of $P$ is the dimension of the smallest linear subspace containing $P$. We say that $P$ is a \emph{full dimensional} polytope when $\Span(P) = \R^n$. Points in $\Z^n$ are said to be \emph{lattice points}. When the vertices of $P$ are lattice points we say that $P$ is a \emph{lattice polytope}.  

For each full dimensional lattice polytope $P$, we can associate an integer number, $\deg(P)$, called the \emph{degree} of $P$ that encodes structural properties of the polytope. A purely combinatorial problem is to classify lattice polytopes with a given degree $d$. For instance, polytopes with degree zero are precisely unimodular simplices. A non-conventional way to approach this subject is by using Toric Geometry. As with fans, one can obtain toric varieties from polytopes. From a given full dimensional lattice polytope $P \subset \R^n$, one can construct a projective toric variety $X_P$ together with an ample divisor $L_P$ on it. The fan of $X_P$ is called the \emph{normal fan} of $P$. We provide a classification theorem for polytopes with low degree, by reinterpreting $\deg(P)$ as a birational invariant of the polarized toric variety $(X_P, L_P)$. We use  techniques of birational geometry, such as Adjunction Theory and the MMP, to classify the variety $X_P$. Then we recover a description of the combinatorics of the polytope $P$.

Another problem that we treat in this thesis concerns finding generators for special cones of cycles in toric varieties. As an application, we describe the cone of pseudo-effective divisors of Losev-Manin moduli spaces

In the following sections we will give more details about the problems treated in this thesis.  

\section{Polytopes with small degree}

The \emph{degree} of a full dimensional lattice polytope
$P \subset \mathbb{R}^n$ is the smallest non-negative integer $d$ such that $kP$ contains no 
interior lattice points for $1\leq k\leq n-d$. The \emph{codegree} of $P$ is defined as $\codeg(P)=n+1-d$.

Lattice polytopes with small degree are very special.
It is not difficult to see that lattice polytopes with degree $d= 0$ are precisely
unimodular  simplices (\cite[Proposition 1.4]{baty}).
%every $n$-dimensional lattice polytope with 
%degree $d= 0$ is affinely isomorphic to
%the standard  $n$-dimensional unimodular  simplex $\Delta_n$ (\cite[Proposition 1.4]{baty}).
In \cite[Theorem 2.5]{baty}, Batyrev and Nill  classified lattice polytopes
with degree $d= 1$.
They all belong to a special class of lattice polytopes, called \emph{Cayley polytopes}.
A Cayley polytope is a lattice polytope affinely isomorphic to
$$
P_0 *...* P_k \ := \ \Conv \big(P_0 \times \{\bar{0}\}, P_1 \times \{e_1\},...,P_k \times \{e_k\}\big) \subset \mathbb{R}^{m}\times \mathbb{R}^{k},
$$
where the $P_i$'s are lattice polytopes in $ \mathbb{R}^m$, and 
$\{e_1,...,e_k\}$ is a basis for $\mathbb{Z}^k$.
Batyrev and Nill also posed the following problem: to find a function $N(d)$ such that every lattice polytope of degree $d$ and dimension $n > N(d)$ is a Cayley polytope.
In \cite[Theorem 1.2]{hnp}, Hasse, Nill and Payne solved this problem with the quadratic polynomial 
$N(d) = (d^2+19d-4)/2$. 
It was conjectured in \cite[Conjecture 1.2]{alicia2} that one can take $N(d)=2d$. 
This would be a sharp bound.
Indeed, let $ \Delta_n$ denote the standard  $n$-dimensional unimodular  simplex.
If $n$ is even, then %the polytope
$2 \Delta_n$ has degree $d=\frac{n}{2}$, but is not a Cayley polytope.

While the methods of \cite{baty} and \cite{hnp} are purely combinatorial,  
Hasse, Nill and Payne pointed out that these results can be interpreted 
in terms of Adjunction Theory on toric varieties.
This point of view was then explored by Dickenstein, Di Rocco and Piene in \cite{alicia} 
to study \emph{smooth} lattice polytopes with small degree. 
Recall that an $n$-dimensional lattice polytope $P$ is smooth if
there are exactly $n$ facets incident to each vertex of $P$,
and the primitive inner normal vectors of these facets form a basis 
for  $\mathbb{Z}^n$.
This condition is equivalent to saying that the toric variety associated to $P$ is smooth.
One has the following classification of smooth $n$-dimensional  
lattice polytopes $P$ with degree $d<\frac{n}{2}$ (or, equivalently, $\codeg(P) \geq \frac{n+3}{2}$). \\

\noindent \textbf{Theorem} (\cite[Theorem 1.12]{alicia}  and \cite[Theorem 1.6]{alicia2}) Let $P \subset \mathbb{R}^n$ be a smooth $n$-dimensional lattice polytope. Then $\codeg(P) \geq \frac{n+3}{2}$ if and only if $P$ is 
affinely isomorphic to a Cayley polytope $P_0 *...* P_k$, where all the $P_i$'s 
have the same normal fan, and $k = \codeg(P)-1 > \frac{n}{2}$. \\

This theorem was first proved in \cite{alicia} under the additional assumption that 
$P$ is a \emph{$\mathbb{Q}$-normal} polytope. (See Definition~\ref{defn_qnormal} for the notion of
$\mathbb{Q}$-normality.) 
Then, using combinatorial methods, Dickenstein and Nill showed in  \cite{alicia2}
that  the inequality $\codeg(P) \geq \frac{n+3}{2}$  implies that $P$ is $\mathbb{Q}$-normal.

In this thesis we go one step further. 
We address smooth $n$-dimensional lattice polytopes $P$  of degree 
$d< \frac{n}{2}+1$ (or, equivalently, $\codeg(P) \geq \frac{n+1}{2}$).
Not all such polytopes are Cayley polytopes.
In order to state our classification result, we need the following generalization
of Cayley polytope concept, introduced in \cite{alicia}.\\

\noindent \textbf{Definition.}
Let $P_0, \dots, P_k \subset \mathbb{R}^m$  be $m$-dimensional lattice polytopes, 
and $s$ a positive integer.   Set
$$
[P_0 *...* P_k]^s \ := \ \Conv \big(P_0 \times \{\bar{0}\}, P_1 \times \{se_1\},...,P_k \times \{se_k\}\big) \subset \mathbb{R}^{m}\times \mathbb{R}^{k},
$$
where $\{e_1,...,e_k\}$ is a basis for $\mathbb{Z}^k$.
We say that a lattice polytope $P$ is an $s^{th}$  \emph{order generalized Cayley polytope} 
if $P$ is affinely isomorphic to a polytope $[P_0 *...* P_k]^s$ as above.
If all the $P_i$'s  have the same normal fan, we write 
$P = \Cayley^s(P_0, \ldots ,P_k)$, and say that $P$ is \emph{strict}.\\

The following is our main result:\\

\noindent \textbf{Theorem \ref{mainthm}.} Let $P \subset \mathbb{R}^n$ be a smooth $n$-dimensional $\mathbb{Q}$-normal lattice polytope. Then $\codeg(P) \geq \frac{n+1}{2}$ if and only if $P$ is affinely isomorphic to one of the following polytopes:
\begin{enumerate}
\item[{\rm (i)}] $2\Delta_n$;
\item[{\rm (ii)}] $3\Delta_3 \ (n=3)$;
\item[{\rm (iii)}] $s\Delta_1$, $s\geq 1 \ (n=1)$;
\item[{\rm (iv)}] $\Cayley^1(P_0, \ldots, P_k)$, where $k=\codeg(P)-1 \geq \frac{n-1}{2}$;
\item[{\rm (v)}] $\Cayley^2(d_0\Delta_1, d_1\Delta_1, \ldots, d_{n-1}\Delta_1)$, where $n \geq 3$ is odd and the $d_i$'s are congruent modulo 2.
\end{enumerate}

\noindent \textbf{Corollary \ref{bound_tobe_ cayley}.} Let $P \subset \mathbb{R}^n$ be a smooth $n$-dimensional $\mathbb{Q}$-normal lattice polytope. If $\codeg(P) \geq \frac{n+1}{2}$, then $P$ is a strict generalized Cayley polytope.
\\

In Example \ref{contra}, we describe a smooth $n$-dimensional lattice polytope 
$P \subset \mathbb{R}^n$ with 
$\codeg(P) = \frac{n+1}{2}$ which is not a generalized Cayley polytope. 
So we cannot drop the assumption of $\mathbb{Q}$-normality in the corollary above.

Our proof of Theorem \ref{mainthm} follows the strategy of \cite{alicia}: we
interpret the degree of $P$ as a geometric invariant of the corresponding polarized variety $(X_P,L_P)$,
and then apply techniques of Birational Geometry, such as Adjunction Theory and the MMP. 
This approach naturally leads to introducing more refined invariants of polytopes,
which are the polytope counterparts of  important invariants of polarized varieties.
In particular, we consider the \emph{$\mathbb{Q}$-codegree}  $\codeg_{\mathbb{Q}}(P)$
of  $P$ (see Definition~\ref{defn_qnormal}).
This is a rational number that carries information about the birational geometry of $(X,L)$,
and satisfies $ \lceil \codeg_{\mathbb{Q}}(P) \rceil \leq \codeg(P)$. Equality holds if $P$ is $\Q$-normal.

The following is the polytope version of a conjecture by Beltrametti and Sommese 
(see \cite[7.18]{beltrametti}).
\\

\noindent \textbf{Conjecture.} Let $P\subset \mathbb{R}^n$ be a smooth $n$-dimensional lattice polytope. Then $P$ is $\mathbb{Q}$-normal if $\codeg_{\mathbb{Q}}(P) > \frac{n+1}{2}$.
\\

If this conjecture holds, then Theorem \ref{mainthm} implies that smooth polytopes $P$ with 
$\codeg_{\mathbb{Q}}(P) > \frac{n+1}{2}$ are those in (iv) with $k \geq \frac{n}{2}$. 
By Proposition \ref{prop1}, these polytopes have $\mathbb{Q}$-codegree $\geq \frac{n+2}{2}$. 
Hence, if the conjecture holds, then the $\mathbb{Q}$-codegree of smooth lattice polytopes
does not assume values in the interval $\left( \frac{n+1}{2}, \frac{n+2}{2} \right)$.

\section{Cones of Cycles on Toric Varieties}

Let $X$ be a smooth complex projective variety of dimension $n$. A \emph{k-cycle} on $X$ is a finite formal sum $\sum a_V V$, where each $V$ is a $k$-dimensional proper subvariety of $X$ and $a_V \in \Z$. We denote the free group of $k$-cycles by $A_k(X)$. For every $k=0, 1, ..., n$, there exists a well defined intersection product between cycles on $X$:
$$
\cdot \ : A_k(X) \times A_{n-k}(X)  \to \Z,
$$                         
with the property that if $V \in A_k(X)$ and $W \in A_{n-k}(X)$ are proper subvarieties intersecting transversally, then $V \cdot W = \#(V \cap W)$.
Two $k$-cycles $\alpha$ e $\alpha'$ on $X$ are \emph{numerically equivalent} if $\alpha \cdot \beta = \alpha' \cdot \beta$ for every $\beta \in A_{n-k}(X)$. Numerical equivalence is an equivalence relation in $A_k(X)$. We denote by $[\alpha]$ the numerical class of a cycle $\alpha$. Let $N_k(X)$ be the vector space of $\R$-linear combinations of classes of $k$-cycles modulo numerical equivalence. It is well known that these vector spaces have finite dimension. For every $k =0, 1, ..., n$ we define the cones:
$$
\overline{NE}_k(X) := \overline{\Cone \Big([V] \ \big| \  V \subset X \ \mbox{is a proper subvariety of dimension} \ k \Big)} \subset N_k(X),
$$ 
$$ Nef^k(X) := \overline{NE}_k(X)^{\vee} \subset N_{n-k}(X).     
$$
The cone $\overline{NE}_k(X)$ is called the \emph{cone of pseudo-effective k-cycles} of $X$.
Cones of cycles play an important role in Algebraic Geometry. For instance, Kleiman proved in \cite{kleiman} that a Cartier divisor on $X$ is ample if and only if its numerical class $[D]$ belongs to the interior of the cone $Nef^1(X)$. Moreover, the steps of the Minimal Model Program are related to certain extremal rays of the cone $\overline{NE}_1(X)$.

In \cite{bdpp}, Boucksom, Demailly, Paun and Peternell established the following interpretation for the cone $Nef^{n-1}(X)$: this is equal to the closure of the cone in $N_1(X)$ generated by the classes of movable curves on $X$ (i.e., curves moving in an irreducible family that sweeps out $X$). In \cite{araujo}, Araujo observed that under certain assumptions, $Nef^{n-1}(X)$ is polyhedral and that each extremal ray of this cone is related to certain outcome of an MMP for $X$. Such outcomes are varieties $X'$ with a special fibration structure $f: X' \to Y$ called \emph{Mori fiber space} (see Section \ref{log} for details).

When $X$ is a toric variety, much of this theory can be related to the combinatorics of its fan $\Sigma_X$. Let $\Sigma_X(1)$ be the set of primitive lattice vectors generating the $1$-dimensional cones of $\Sigma_X$. The vector space $N_1(X)$ can be identified with the vector space of linear relations among the vectors in $\Sigma_X(1)$. A relation:
$$
a_0v_0 + ... +a_kv_k = 0, \ \ \ \ v_i \in \Sigma_X(1) \ \ \forall i,
$$
is called \emph{minimal} if each $a_i$ is a positive rational number, and the vector space generated by the $v_i's$ has dimension $k$. Victor Batyrev conjectured that the minimal relations are exactly the generators of the extremal rays of the cone $Nef^{n-1}(X)$. We prove this conjecture with the following theorem:\\

\noindent \textbf{Theorem \ref{main_theorem}.} A linear relation among vectors in $\Sigma_X(1)$ generates an extremal ray of $Nef^{n-1}(X)$ if and only if it is minimal. Moreover, each minimal relation of primitive lattice vectors in $\Sigma(1)$ determines the fan of a general fiber of a Mori fiber space obtained by an MMP for $X$. \\

For toric varieties, the cone of pseudo-effective divisors $\overline{NE}_{n-1}(X)$ is polyhedral and is generated by numerical classes of $T$-invariant divisors on $X$. Each such divisor $D_{\tau}$ is uniquely associated to a $1$-dimensional cone of $\Sigma_X$ generated by a primitive lattice vector $v_{\tau}$. The theorem above allows us to establish a simple combinatorial criterion to decide when $D_{\tau}$ determines an extremal ray of $\overline{NE}_{n-1}(X)$.\\

\noindent \textbf{Corollary \ref{criterion_eff}.} Write $\rho(X) = \dim (N_1(X))$. A class $[D_{\tau}]$ generates an extremal ray of $\overline{NE}_{n-1}(X)$ if and only if there exist $\rho(X)-1$ linearly independent minimal relations, such that $v_{\tau}$ does not appear in any of these relations.\\

The corollary above will be useful to describe the cone of pseudo-effective divisors of Losev-Manin moduli spaces $\bar{L}_{n+1}$. These spaces were introduced in \cite{lm00} and provide a compactification of the moduli spaces $\mathcal{M}_{0,n+3}$ of smooth rational curves with $n+3$ marked points. They parametrize chains of lines with $n+1$ marked points satisfying certain properties (see Section \ref{losev_manin} for a precise description). One can show that $\bar{L}_{n+1}$ is an $n$-dimensional smooth projective toric variety. We prove the following result:\\

\noindent \textbf{Proposition \ref{losev_cone}.} The cone $\overline{NE}_{n-1}(\bar{L}_{n+1})$ is polyhedral and has $2^{n+1} -2$ extremal rays. This number is exactly the quantity of $T$-invariant prime divisors on $\bar{L}_{n+1}$. \\

Classically the study of cones of cycles on a projective variety was concerned only with cycles of dimension and codimension one. Little is known about cones with arbitrary dimension. Recently, these cones have started to come into focus. For instance, in \cite{voisin} it was established some results about the cones $\overline{NE}_k(X)$ when $X$ is an Abelian variety. In \cite{eide}, the cone $\overline{NE}_2(X)$ was used to classify toric 2-Fano manifolds of dimension 4.

In the toric case, Reid proved in \cite{reid} that the cone of curves $\overline{NE}_1(X)$ is a polyhedral cone. In the paper \cite[Problem 6.9]{voisin} the authors say that ``presumably'' the other cones $\overline{NE}_k(X)$ are also polyhedral. In fact, we prove this assertion showing that every $k$-dimensional subvariety $V \subset X$ is numerically equivalent to a linear combination, with nonnegative coefficients, of $T$-invariant subvarieties of $X$. It would be interesting, as in the case $k=n-1$, to have a criterion to decide when a $T$-invariant proper subvariety of $X$ determines an extremal ray of $\overline{NE}_k(X)$.

\section{Organization of the thesis}

This thesis is structured as follows: In Chapter 2 we give an overview of the general theory of toric varieties, Adjunction theory and the Minimal Model Program. 

In Chapter 3 we discuss certain invariants of lattice polytopes, relating them to algebro-geometric invariants of polarized varieties. We also describe properties of toric $\mathbb{P}^k$-bundles and their relationship with Cayley Polytopes.

In Chapter 4 we discuss about Generalized Cayley Polytopes and we prove a classification theorem for smooth lattice polytopes that extends the main result of \cite{alicia}.

In Chapter 5 we prove that the cone of pseudo-effective $k$-cycles on a smooth complete toric variety $X$ is polyhedral. We also prove a conjecture by Victor Batyrev, characterizing extremal rays of the cone of moving curves of a $\Q$-factorial projective toric variety. As a consequence, we give a criterion to decide which numerical classes of $T$-invariant prime divisors generate extremal rays of the cone of pseudo-effective divisors. In this chapter we also describe the cone of pseudo-effective divisors on Losev-Manin moduli spaces (this last part comes from a joint work with Edilaine Nobili).

%%%%%%%%%%%%%%%%%%%%%%%%%%%%%%%%%%%%%%
%%%%%%%%%%%%%%%%%%%%%%%%%%%%%%%%%%%%%%
%%%%%%%%%%%%%%%%%%%%%%%%%%%%%%%%%%%%%%
%
%
%           CAPÍTULO 2
%
%
%%%%%%%%%%%%%%%%%%%%%%%%%%%%%%%%%%%%%%
%%%%%%%%%%%%%%%%%%%%%%%%%%%%%%%%%%%%%%
%%%%%%%%%%%%%%%%%%%%%%%%%%%%%%%%%%%%%%

\chapter{Preliminaries}

\section{Basic properties of toric Varieties}

The object of our work consists in a special class of algebraic varieties, known as toric varieties. These varieties are equipped with a very special combinatorics, based on properties of suitable collections of cones contained in the $n$-dimensional Euclidean space $\R^n$. Exploring this combinatorics, one can determine many geometric properties of the variety (e.g. smoothness, projectivity , completeness) as well as compute with reasonable ease its Picard group and intersection products of cycles of any dimension. In this sense, toric varieties are excellent ``prototypes'' to check the validity of conjectures stated in a general scope.

\vspace{0,3 cm}

\noindent \textbf{Notation:} Throughout this section, $N \simeq \Z^n$ will be a lattice of rank $n$ and $M := N^{\vee} = \Hom_{\Z}(N, \Z)$ the dual lattice of $N$. We denote by $N_{\R}$ the real vector space $N \otimes_{\Z} \R$ associated to $N$ and by $M_{\R}$ the dual of $N_{\R}$. 

\subsection{Construction of toric varieties}

The elementary results of the theory of toric varieties will be merely stated. Proofs can be found in \cite{cox} or \cite{fulton}. Throughout these notes, an algebraic variety will be an irreducible, separated scheme of finite type over $\C$.

\begin{de} An $n$-dimensional normal algebraic variety $X$ is called a \emph{toric variety} if $X$ contains an $n$-dimensional algebraic torus $T \simeq (\C^*)^n$ as a Zariski open subset such that the natural action of $T$ on itself extends to an action $\varphi: T \times X \to X$.
\end{de}

\begin{ex} The $n$-dimensional projective space $\mathbb{P}^n$ is the simplest example of projective toric variety. The torus $T \subset \mathbb{P}^n$ consist in the points $(1: t_1 : ... : t_n)$ with $t_i \neq 0, \ \forall i$. The action $\varphi: T \times \mathbb{P}^n \to \mathbb{P}^n$, $\varphi \big ( (1: t_1 : ... : t_n), (x_0 : ... : x_n) \big ) = (x_0 : t_1x_1 : ... : t_nx_n)$ extends the natural action of $T$ on itself.
\end{ex}

The coordinate ring of an $n$-dimensional algebraic torus $T$ is isomorphic to the $\C$-algebra $\C[x_1^{\pm 1}, ..., x_n^{\pm 1}]$, hence the function field of any $n$-dimensional toric variety is the field $\C(x_1, ..., x_n)$.\\

Now, we describe a combinatorial way of constructing toric varieties. There is an algebraic torus $T = T_N$ naturally associated to the lattice $N$, namely, the torus $N \otimes_{\Z} \C^*$. For each $u = (u_1, ..., u_n) \in M$, the \emph{character} of $T_N$ associated to $u$ is the function $\chi^u: T_N \to \C^*$, $\chi(t_1, ..., t_n) = t_1^{u_1}, ..., t_n^{u_n}$. Given $u, u' \in M$, we have $\chi^u \cdot \chi^{u'} = \chi^{u+u'}$. Note that the coordinate ring of $T_N$ is the $\C$-algebra generated by all the characters of $T_N$.\\

A {\it cone} in the vector space $N_{\R}$ is a set 
$$ 
\sigma = \Cone(S) := \Big\{ \displaystyle \sum_{v_i \in S}a_iv_i \ \Big| \ a_i \in \R_+ \Big\}.
$$
We say that $S$ is a set of \emph{generators} of the cone $\sigma$.  The {\it dimension} $\dim(\sigma)$ of a cone $\sigma$ is the dimension of the smallest vector subspace of $N_{\R}$ containing $\sigma$. A subcone $\tau$ of $\sigma$ is called a \emph{face} of $\sigma$, and we write $\tau \prec \sigma$, if given any $u,v \in \sigma$ with $u+v \in \tau$ we have $u,v \in \tau$. When $\dim(\tau) =1$, $\tau$ is called an {\it extremal ray} of $\sigma$.  
 
Given a cone $\sigma$ in $N_{\R}$, its {\it dual cone} is defined by:
$$\sigma^{\vee} := \{ u \in M_{\R} \ \Big| \ \langle u, v \rangle \geq 0, \ \forall v \in \sigma \}.$$
The set $\sigma^{\vee}$ is a cone in $M_{\R}$ such that $(\sigma^{\vee})^{\vee} = \bar{\sigma}$, where $\bar{\sigma}$ is the closure of $\sigma$ in $M_{\R}=\R^n$. If $\tau \subset \sigma$ then $\sigma^{\vee} \subset \tau^{\vee}$. 

When $S$ is a finite set we say that $\sigma$ is a {\it polyhedral cone}. In addition if $S \subset N$, $\sigma$ is called \emph{rational}. The faces of a polyhedral cone $\sigma$ are polyhedral cones again and can be described as  intersection of $\sigma$ with some hyperplanes. In fact, every dual vector $u \in M_{\R}$ defines in $N_{\R}$ a hyperplane $H_u = \{v \in N_{\R} \ : \ \langle u, v \rangle = 0\}$ and a positive half-space $H_u^+ = \{v \in N_{\R} \ / \ \langle u, v \rangle  \geq 0\}$. Every face $\tau \prec \sigma$ is obtained as a intersection $\tau = \sigma \cap H_u$ for some $u \in M_{\R}$ such that $\sigma \subset H_u^+$.

When $\sigma \subset N_{\R}$ is a polyhedral cone, there exists an one-to-one inclusion-reversing correspondence: 
\begin{align*}
\big\{\mbox{faces \ of} \ \sigma & \big\}  \longleftrightarrow \big\{\mbox{faces \ of} \ \sigma^{\vee} \big\}  \notag \\
& \tau  \longmapsto  \tau^* := \sigma^{\vee} \cap \tau^{\perp}, 
\end{align*}
where $\tau^{\perp} = \{u \in M_{\R} \ | \ \langle u, v \rangle = 0, \ \forall v \in \tau \}$. We have $\dim(\tau^*) = \dim(\sigma^{\vee}) - \dim(\tau)$.

\begin{rem} Let $\sigma$ be a cone of dimension $n$ in $N_{\R}$. It follows from the correspondence between faces of $\sigma$ and $\sigma^{\vee}$ that if $v$ is a nonzero vector in $\sigma$ and if there exists a linearly independent set $\{u_1,...,u_{n-1}\} \subset \sigma^{\vee}$ such that $\langle u_i, v \rangle = 0 \ \forall i$, then $\R_+\cdot v$ is an extremal ray of $\sigma$. In fact, every non-zero vector $v \in \sigma$ defines a proper face $\tau'$ of $\sigma^{\vee}$ given by the intersection:
$$
\tau' = \sigma^{\vee} \cap \{u \in M_{\R} \ | \ \langle u, v \rangle =0 \}.
$$ 
If $\tau'$ contains a linearly independent set with $n-1$ vectors then $\dim(\tau') = n-1$. Using the one-to-one correspondence, we conclude that $\tau'$ determines a $1$-dimensional face $\tau \prec \sigma$ containing $v$, therefore $\tau = \R_+v$, as required. 

\label{lin_equiv}
\end{rem}       

\begin{de} A {\it fan} $\Sigma$ in $N_{\R}$ is a finite collection of rational polyhedral cones in $N_{\R}$ satisfying:
\begin{enumerate}
\item[(1)] $\sigma \in \Sigma$, $\tau \prec \sigma \Rightarrow \tau \in \Sigma$;
\item[(2)] $\sigma_1, \sigma_2 \in \sigma \Rightarrow \sigma_1 \cap \sigma_2 \in \Sigma$;
\item[(3)] $\sigma \cap (-\sigma) = {0}, \ \forall \sigma \in \Sigma.$
\end{enumerate}

\label{fan}
\end{de}

Note that given a polyhedral cone $\sigma$, the intersection $U = \sigma \cap (-\sigma)$ is a vector subspace of $N_{\R}$. Therefore, the condition (3) above says that the cones of the fan do not contain lines. Cones satisfying (3) are called \emph{strongly convex}.

Given $m \in \{0, 1, ..., n\}$, we denote by $\Sigma(m)$ the collection of $m$-dimensional cones of $\Sigma$. Cones of $\Sigma$ with dimension $n$ are called \emph{maximal}. We say that a cone $\tau \in \Sigma(n-1)$ is a \emph{wall} when $\tau$ is face of two maximal cones of $\Sigma$. The \emph{support} of a fan $\Sigma$, denoted by $|\Sigma|$, is the subset of $N_{\R}$ obtained by the union of all the cones $\sigma \in \Sigma$. When $|\Sigma| = N_{\R}$ we say that $\Sigma$ is \emph{complete}.\\

A finite collection of rational polyhedral cones $\Sigma^*$ in $N_\R$ is called a \emph{degenerated fan} when it satisfies conditions (1) and (2) of a fan and the following condition: 
\begin{enumerate}
\item[(3')] There is a vector subspace $U \neq  \{0\}$ of $N_{\R}$ such that $U = \sigma \cap (-\sigma), \ \forall \sigma \in \Sigma^*$.
\end{enumerate}

By (1), the vector subspace $U$ belongs to $\Sigma^*$ and is a face of every cone $\sigma \in \Sigma^*$. We say that $U$ is the \emph{vertex} of $\Sigma^*$. If $\Sigma^*$ is a degenerated fan in $N_{\R}$ and $g: N_{\R} \to N_{\R}/U$ is the natural projection, we can construct a fan $\Sigma$ in $N_{\R}/U$, setting $\Sigma = \{g(\sigma) \ / \ \sigma \in \Sigma^* \}$. Here, we set $N/(N \cap U)$ to be the lattice corresponding to $\Sigma$ in $N_{\R}/U$. \\

From now on, by simplicity, a \emph{cone} will be a strongly convex rational polyhedral cone.\\

Given a semigroup $S$ of the lattice $M$, we denote by $\C[S]$ the $\C$-algebra generated by the characters $\chi^u$, $u \in S$. If $\sigma$ is a rational polyhedral cone in $N_{\R}$, let $S_{\sigma} := \sigma^{\vee} \cap M$. 

\begin{prop}[Gordan's Lemma] $S_{\sigma}$ is a finitely generated semigroup of $M$. In particular, the $\C$-algebra $\C[S_{\sigma}]$ is finitely generated and hence defines an affine toric variety $\mathcal{U}_{\sigma} = \Spec \C[S_{\sigma}]$ with torus $T_{N} = \mathcal{U}_{\{0\}}$.

\label{gordan}
\end{prop}

It is not difficult to see that the $\dim(\mathcal{U}_{\sigma}) = n$ if and only if $\sigma$ is strongly convex (see for instance \cite[1.2.18]{cox}).\\

Let $\Sigma$ be a fan in $N_{\R}$. There exists, up to isomorphism, a unique toric variety corresponding to $\Sigma$ that we will denote by $X_{\Sigma, N}$ or simply by $X_{\Sigma}$. This variety is obtained by gluing together the affine varieties $\mathcal{U}_{\sigma}$, $ \sigma \in \Sigma$, along the subvarieties $\mathcal{U}_{\tau} \subset \mathcal{U}_{\sigma}$ for $\tau \prec \sigma$ (see \cite[3.1.5]{cox} for details of the construction). It is also possible to construct a toric variety from a degenerated fan $\Sigma^*$ with vertex $U$, setting it to be the toric variety corresponding to the fan obtained projecting $\Sigma^*$ in $N_{\R}/U$. It is a well-known result, attributed to Sumihiro, that for every $n$-dimensional toric variety $X$, there exists a fan $\Sigma$ in $N_{\R}$ such that $X \simeq X_{\Sigma}$ (see for instance \cite[3.1.8]{cox}).

\begin{ex} In this example we describe the fan $\Sigma$ of $\mathbb{P}^n$. Let $\{e_1, ..., e_n\}$ be the canonical basis of $N$ and $e_0:= -e_1 - ... - e_n$. For each $j = 0, 1, ..., n$, we define:
$$
\sigma_j := \Cone \big(e_0, ..., \hat{e}_j, ..., e_n \big)
$$
The cones of $\Sigma$ are the cones $\sigma_j$, $j=0, ..., n$ and their faces. The open subsets $\mathcal{U}_{\sigma_j}$ are precisely the basic open subsets $\{(x_0: ...: x_n) \ | \ x_j \neq 0\}$ of $\mathbb{P}^n$. 
\begin{figure}[h]
\centering
\includegraphics[scale=0.3]{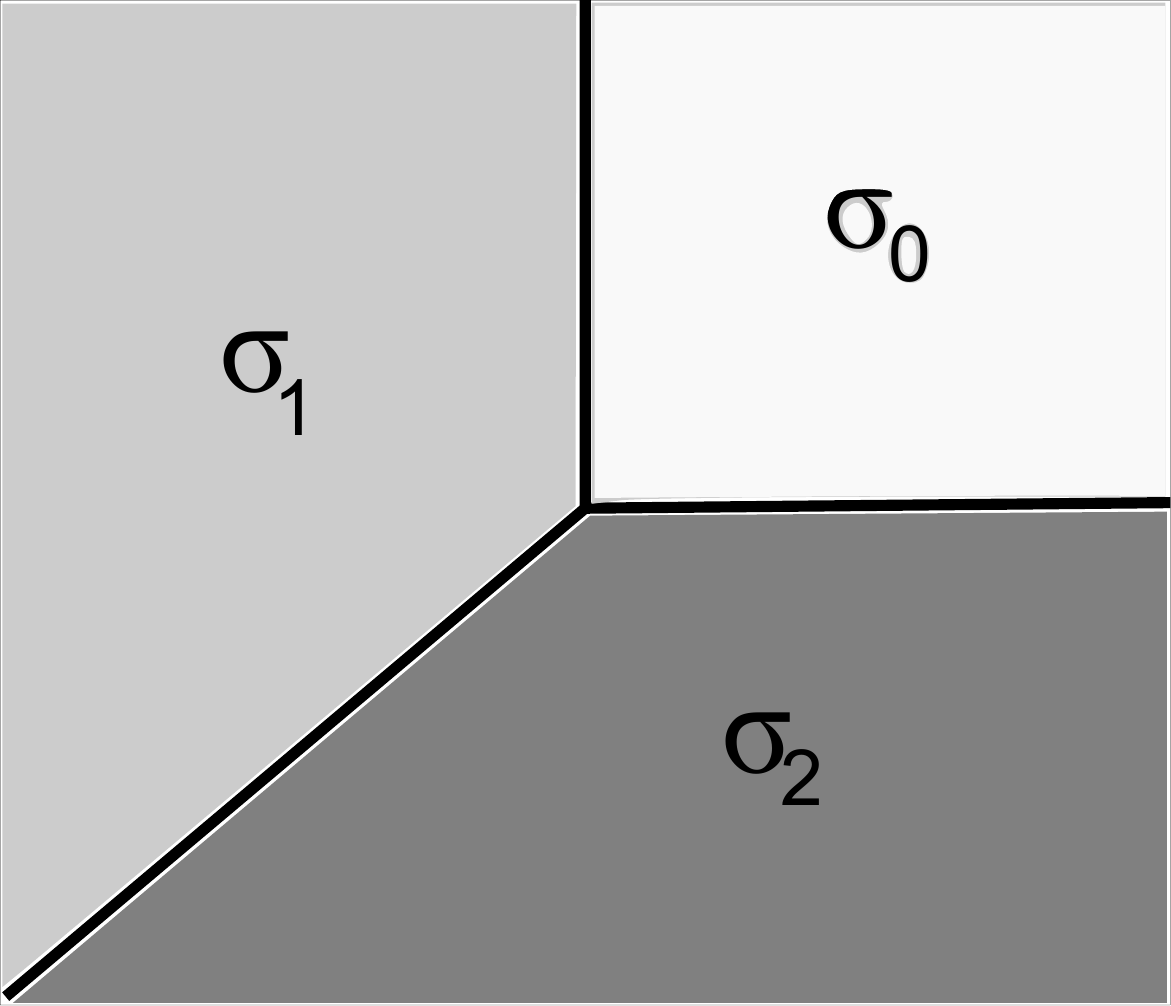}
\caption{The fan for $\mathbb{P}^2$.}
\end{figure}
\end{ex} 

It is a classical result of Algebraic Geometry that there exists an one-to-one correspondence between the closed points of an affine algebraic variety and the maximal ideals of its coordinate ring. In the toric context there is another special one-to-one correspondence between points of a toric variety and certain homomorphisms of semigroups. In fact, if $\sigma \subset N_{\R}$ is a cone we have a bijection:

\begin{center}
     $\Big\lbrace$ points of $\mathcal{U}_{\sigma} \Big\rbrace \longleftrightarrow \Big\lbrace $ homomorphisms of semigroups $S_{\sigma} \to \C \Big\rbrace.$
     
\end{center}

Given a point $p \in \mathcal{U}_{\sigma}$, we define a homomorphism of semigroups $f_p: S_{\sigma} \to \C$ setting $f_p(u) = \chi^u(p)$. Conversely, each homomorphism of semigroups $f: S_{\sigma} \to \C$ induces a natural epimorphism of rings $\tilde{f}: \C[S_{\sigma}] \to \C$, with $\tilde{f}(\chi^u) = f(u)$ for all $u \in S_{\sigma}$. Hence, this epimorphism determines a maximal ideal of $\C[S_{\sigma}]$ and therefore a point $p \in \mathcal{U}_{\sigma}$. It is not difficult to see that these constructions are inverse to each other. Thus, for each cone $\sigma \in \Sigma$, the following homomorphism of semigroups determines a point $x_{\sigma}$ in the toric variety $\mathcal{U}_{\sigma}$:
$$
u\in S_{\sigma} \longmapsto
\left\{
\begin{array}{ll}
1 & \mbox{if } u\in S_{\sigma}\cap \sigma^{\perp} \\
0 & \textrm{otherwise}, \\
\end{array}
\right.
$$
where $\sigma^{\perp} := \{u \in M_{\R} \ / \ \langle u, v \rangle = 0 \ \ \forall v \in \sigma \}$.  The point $x_{\sigma}$ is called the \emph{distinguished point} of $\mathcal{U}_{\sigma}$. 

In the following, we state the first notable result about the structure of toric varieties (see for instance \cite[3.2.6]{cox}).

\begin{prop}[Cone-Orbit Correspondence] Let $X_{\Sigma}$ be a toric variety corresponding to a fan $\Sigma$ in $N_{\R}$.

\begin{enumerate}
\item There is a one-one correspondence:
\begin{align}
          \Big\{ cones \ \ \sigma \in \Sigma \Big\} \longleftrightarrow
 \Big\{ T_N \textrm{-orbits in} \ X_{\Sigma} \Big\}
 \notag
 \\
 \sigma \longleftrightarrow O(\sigma):= T_N \cdot x_{\sigma} \notag
  \end{align}
Furthermore, $\dim(O(\sigma)) = n - \dim(\sigma)$.

\item $\mathcal{U}_{\sigma} = \displaystyle \bigcup_{\tau \prec \sigma} O(\tau)$

\item $\tau \prec \sigma$ if and only if $O(\sigma) \subset V(\tau) : = \overline{O(\tau)}$ and
$$
V(\tau) = \displaystyle \bigcup_{\tau \prec \sigma} O(\sigma),
$$
where $\overline{O(\tau)}$ denotes the closure in both the classical and Zariski topologies.

\end{enumerate}

\label{cone_orbit}

\end{prop}

We say that an irreducible subvariety $V$ of $X_{\Sigma}$ is \emph{T-invariant} if $t \cdot p \in V, \ \forall t \in T$ and $\forall p \in V$. The unique invariant irreducible subvarieties of $X$ are the closures of the orbits $V(\tau) = \overline{O(\tau)}$, $\tau \in \Sigma$.
Moreover, each $V(\tau)$ is also a toric variety and its fan is given as follows: let $N_{\tau} = \Span(\tau) \cap N$. The lattice $N(\tau) = N/N_{\tau}$ is torsion free, then $N \simeq N_{\tau} \oplus N(\tau)$. For each $\sigma \in \Sigma$, we denote by $\bar{\sigma}$ the image of $\sigma$ in $N(\tau)$ under the projection $N_{\R} \to N(\tau)_{\R}$. The \emph{star} of $\tau$ is the collection:
$$
\Star(\tau) = \Big\{\bar{\sigma} \ \big| \ \tau \prec \sigma \in \Sigma \Big\}.
$$ 
This defines a fan in $N(\tau)_{\R}$ whose associated variety is isomorphic to $V(\tau)$. The torus of $V(\tau)$ is $O(\tau)$ and this orbit is isomorphic to the torus $T_{N(\tau)}$ (see \cite[3.2.5 and 3.2.7]{cox}.

\vspace{0,5 cm}

Next we show how certain combinatorial properties of a fan $\Sigma$ in $N_{\R}$ can determine geometric properties of the toric variety $X_{\Sigma}$ such as completeness and smoothness. A cone $\sigma \in N_{\R}$ is called \emph{simplicial} (respectively \emph{smooth}) if it can be generated by a subset of a basis of $N_{\R}$ (respectively by a basis of $N$). A fan $\Sigma$ is simplicial (respectively smooth) when every cone $\sigma \in \Sigma$ is simplicial (respectively smooth).

We recall that an algebraic variety is called \emph{$\Q$-factorial} when every Weil divisor on $X$ has a positive multiple that is Cartier. We have the following result:

\begin{prop}[{\cite[1.3.12, 3.4.8 and 4.2.7]{cox}} ] Let $\Sigma$ be a fan in $N_{\R}$. The toric variety $X_{\Sigma}$ is complete (respectively smooth, respectively $\Q$-factorial) if and only if $\Sigma$ is complete (respectively smooth, respectively simplicial).

\label{fan_properties}
\end{prop} 

\begin{de} A non-zero lattice vector $v \in N$ is called \emph{primitive} if for every lattice vector $u$ such that $v = t \cdot u$, $t \in \Z_{\geq 0}$, we have $t=1$ (and hence $v=u$).
\label{primitive_vector}
\end{de}

\begin{de} Let $\sigma= \Cone(v_1, ..., v_m)$ be an $m$-dimensional cone in $N_{\R}$, where the $v_i$ are primitive lattice vectors. The \emph{multiplicity of} $\sigma$ is the index:
$$
\mult(\sigma) := (N_{\sigma}: \Z v_1 + ... + \Z v_n),
$$
where $N_{\sigma} = \Span(\sigma) \cap N$. Equivalently, let $\mathcal{B} = \{e_1, ..., e_m\}$ be any basis of $N_{\sigma}$ and consider the matrix $[v_1, ..., v_m]$ of order $m$, whose $i^{th}$ column is composed by the coefficients of $v_i$ in the basis $\mathcal{B}$. Then:
$$ 
\mult(\sigma) := \big|\det[v_1, ..., v_n]\big|.
$$

\label{multiplicity}
\end{de}

Observe that an $m$-dimensional cone $\sigma \subset N_{\R}$ is smooth if and only if $\mult(\sigma) = 1$.

\subsection{Toric Morphisms and Refinements}

Let $N$ and $N'$ be lattices, $\Sigma$ a fan in $N_{\R}$ and $\Sigma'$ a fan in $N'_{\R}$. Among the morphisms $\phi: X_{\Sigma} \to X_{\Sigma'}$, we give a special attention to those that preserve the toric structure, i.e.:

\begin{itemize}
\item $\phi$ maps $T_N$ into $T_{N'}$ so that $\phi_{|T_N}$ is a group homomorphism;

\item $\phi$ is \emph{equivariant}. This means that $\phi(t \cdot p) = \phi(t) \cdot \phi(p)$, for every $t \in T_N$ and $p \in X_{\Sigma}$.

\end{itemize}

A morphism satisfying the two conditions above is called a \emph{toric morphism}. In what follows we detail how to construct toric morphisms from suitable $\Z$ -linear maps between the lattices $N$ and $N'$.\\

Let $\bar{\phi}: N \to N'$ be a $\Z$-linear map between two lattices $N$ and $N'$. We denote by $\bar{\phi}_{\R}: N_{\R} \to N'_{\R}$ the $\R$-linear transformation obtained tensorizing $\bar{\phi}$ with $\R$. Given fans $\Sigma$ and $\Sigma'$ in $N_{\R}$ and $N'_{\R}$ respectively, we say that $\bar{\phi}$ is \emph{compatible} with $\Sigma$ and $\Sigma'$ if for every cone $\sigma \in \Sigma$, there exists a cone $\sigma' \in \Sigma'$ such that $\bar{\phi}_{\R}(\sigma) \subset \sigma'$. In this case, the map $\bar{\phi}$ naturally induces a toric morphism $\phi: X_{\Sigma} \to X_{\Sigma'}$. Moreover, every toric morphism between two toric varieties is obtained in this way (see for instance \cite[3.3.4]{cox}). 

We have the following useful result about the structure of toric morphisms:

\begin{lemma}[{\cite[3.3.21]{cox}} ] Let $\phi: X_{\Sigma} \to X_{\Sigma'}$ be a toric morphism coming from a $\Z$- linear map $\bar{\phi}: N \to N'$ that is compatible with $\Sigma$ and $\Sigma'$. Given $\sigma \in \Sigma$, let $\sigma' \in \Sigma$ be the minimal cone of $\Sigma'$ that contains $\bar{\phi}_{\R}(\sigma)$. Then $\phi(x_{\sigma}) = x_{\sigma'}$. In particular, $\phi(O(\sigma)) \subset O(\sigma')$ and $\phi(V(\sigma)) \subset V(\sigma')$. Furthermore, the induced morphism $\phi_{|V(\sigma)}:V(\sigma) \to V(\sigma')$ is a toric morphism.
\label{toric_morphism}
\end{lemma}

\begin{say}[Refinements of fans]

Given a fan $\Sigma$ in $N_{\R}$ we say that a fan $\Sigma'$ \emph{refines} $\Sigma$ if $|\Sigma'| = |\Sigma|$ and every cone of $\Sigma'$ is contained in a cone of $\Sigma$. A refinement $\Sigma'$ of $\Sigma$ induces a natural surjective toric morphism $\phi : X_{\Sigma'} \to X_{\Sigma}$ because the identity map of $N$ is evidently compatible with $\Sigma'$ and $\Sigma$. Note that $\phi$ is a birational morphism because it is the identity on the torus $T_N$.

An important type of refinement is the so called \emph{star subdivision}: let $\Sigma$ be a fan in $N_{\R} \simeq \R^n$ and $\tau \in \Sigma$ a cone of $\Sigma$ with the property that all cones of $\Sigma$ containing $\tau$ are smooth. Suppose that the extremal rays of $\tau$ are generated by primitive vectors $v_1, ..., v_m$. Let $v:= v_1+ ... + v_n$. For every cone $\sigma \in \Sigma$, we define $\sigma(1)$ to be the set of cones in $\Sigma(1)$ that are faces of $\sigma$. For every $\sigma \in \Sigma$ with $\tau \prec \sigma$, set:
$$
\Sigma_{\sigma}^*(\tau) = \Big\{ \Cone(A) \ \Big| \ A \subset \{v\} \cup \sigma(1), \ \tau (1) \nsubseteq  A \Big\}.
$$
The \emph{star subdivision of $\Sigma$ centered in $\tau$} is the fan:
$$
\Sigma(\tau) = \Big\{\sigma \in \Sigma \ \Big| \ \tau \nsubseteq \sigma \Big\} \cup \displaystyle \bigcup_{\tau \subset \sigma}\Sigma_{\sigma}^*(\tau).
$$
The fan $\Sigma(\tau)$ is a refinement of $\Sigma$ and the toric morphism $\phi: X_{\Sigma(\tau)} \to X_{\Sigma}$ is the blowup of $X_{\Sigma}$ along the subariety $V(\tau)$ (see \cite[3.3.17]{cox}). 

\begin{figure}[h]
\centering
\includegraphics[scale=0.24]{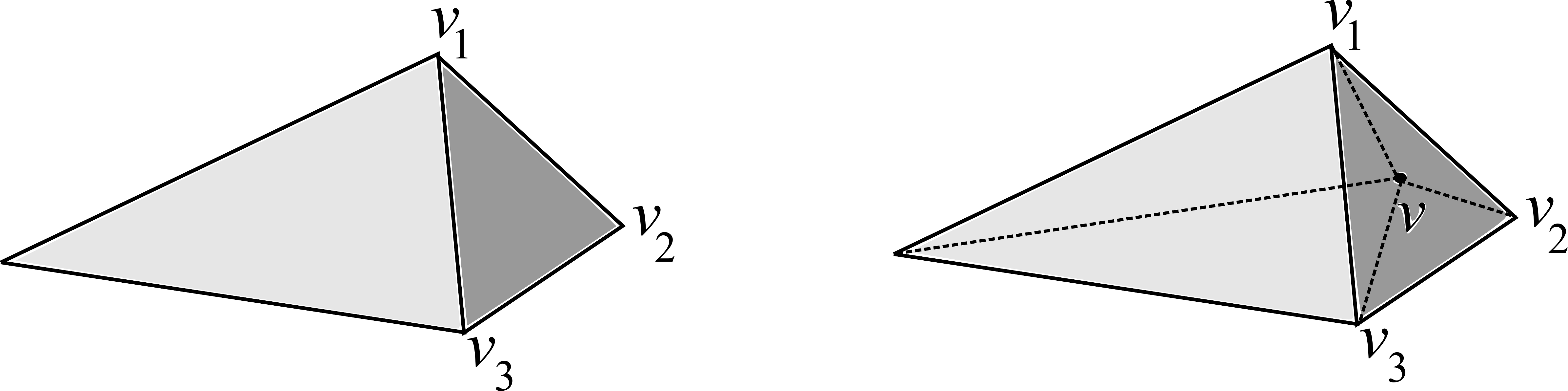}
\caption{The star subdivision of a cone.}
\end{figure}

\label{blowup}
\end{say}

The following proposition is the toric version of Chow's Lemma (see \cite[6.1.18 and 11.1.9]{cox})

\begin{prop}[Toric Chow's Lemma] Let $X_{\Sigma}$ be a complete toric variety. There exists a refinement $\Sigma'$ of $\Sigma$ such that $X_{\Sigma'}$ is a smooth projective toric variety.
\label{chow_lemma}
\end{prop}

\begin{say}[Fiber Bundles] We recall how a toric morphism associated to a surjective map of lattices $N \to N'$ has a structure of fiber bundle over the torus $T_{N'}$. First we recall a general fact about the Cartesian product of toric varieties.

Let $\Sigma$ and $\Sigma'$ be two fans respectively in $N_{\R}$ and $N'_{\R}$. The Cartesian product of $\Sigma$ and $\Sigma'$ is the collection:
$$
\Sigma \times \Sigma' := \Big\{ \sigma \times \sigma' \ \Big| \ \sigma \in \Sigma \ \textrm{and} \ \sigma' \in \Sigma' \Big\}.
$$
This is a fan in $(N \times N')_{\R}$. For every $\sigma \in \Sigma$ and $\sigma' \in \Sigma'$, we have $(\sigma \times \sigma')^{\vee} = \sigma^{\vee} \times {\sigma'}^{\vee}$ and hence $S_{\sigma \times \sigma'} = S_{\sigma} \times S_{\sigma'}$. It follows that $\C[S_{\sigma \times \sigma'}] = \C[S_{\sigma}] \otimes \C[S_{\sigma'}]$ and therefore $\mathcal{U}_{\sigma \times \sigma'} \simeq \mathcal{U}_{\sigma} \times \mathcal{U}_{\sigma'}$. These open subsets glue together, resulting in an isomorphism $X_{\Sigma \times \Sigma'} \simeq X_{\Sigma} \times X_{\Sigma'}$.\\

Consider now a surjective $\Z$-linear map $\bar{\phi}: N \to N'$ compatible with the fans $\Sigma$ and $\Sigma'$. Let $N_0$ be the kernel of $\bar{\phi}$. The collection $\Sigma_0 = \{\sigma \in \Sigma \ / \ \bar{\phi}_{\R}(\sigma) = \{0\} \}$ is a fan in ${N_0}_{\R}$ and $N_{\R}$. There is an exact sequence:
$$
0 \longrightarrow N_0 \longrightarrow N \longrightarrow N' \longrightarrow 0.
$$

Thus, $N \simeq N_0 \times N'$. Considering the trivial fan $\{0\}$ in $N'$, by the considerations about cartesian product discussed above we have:
$$
X_{\Sigma_0, N} \simeq X_{\Sigma_0, N_0} \times T_{N'}.
$$

On the other hand, if $\phi: X_{\Sigma} \to X_{\Sigma'}$ is the toric morphism corresponding to $\bar{\phi}$, Lemma \ref{toric_morphism} says that $\phi^{-1}(T_{N'}) = X_{\Sigma_{0, N}}$. It follows that $\phi^{-1}(T_{N'}) \simeq X_{\Sigma_0, N_0} \times T_{N'}$ and therefore $\phi$ is a fiber bundle over $T_{N'}$ whose general fiber is isomorphic to $X_{\Sigma_0, N_0}$. 

\begin{rem} In the situation discussed above, the special fibers of $\phi$ (i.e., those that lie over the invariant subvarieties of $X_{\Sigma'}$) can fail to be isomorphic to $X_{\Sigma_0, N_0}$. In fact, let $\Sigma'$ be the fan of $\mathbb{P}^1$. The fan of $\mathbb{P}^1 \times \mathbb{P}^1$ is $\Sigma = \Sigma' \times \Sigma'$ in $\R^2$. If $\sigma = \Cone(e_1, e_2)$, the star subdivision $\Sigma(\sigma)$ of $\sigma$ is the fan of the blowup of $\mathbb{P}^1 \times \mathbb{P}^1$ in the invariant point $x_{\sigma}$. The projection $\Z^2 \to \Z, \ (x,y) \mapsto x$ is compatible with $\Sigma(\sigma)$ and $\Sigma'$, hence induces a morphism $\phi: X_{\Sigma(\sigma)} \to \mathbb{P}^1$ that is the composition of the blowup map $X_{\Sigma{\sigma}} \to \mathbb{P}^1 \times \mathbb{P}^1$ with the first projection $\mathbb{P}^1 \times \mathbb{P}^1 \to \mathbb{P}^1$. The general fiber of $\phi$ is isomorphic to $\mathbb{P}^1$, but the fiber over the invariant point $x_{\R_+e_1}$ of $\mathbb{P}^1$ has two irreducible components.
   
\end{rem} 

\label{fiber_bundle}
\end{say}

\subsection{Divisors on toric varieties}
\label{toric_divisors}

Throughout this section, $X_{\Sigma}$ will be the $n$-dimensional toric variety associated to a fan $\Sigma$ in $N_{\R}$  The unique invariant irreducible subvarieties of $X_{\Sigma}$ are the closures of the orbits $V(\rho) = \overline{O(\rho)}$, $\rho \in \Sigma$. When $\dim(\rho) =1$, $\overline{O(\rho)}$ is a prime divisor, and sometimes we will denote it by $D_{\rho}$ rather than $V(\rho)$.\\

For each $\rho \in \Sigma(1)$, we denote by $v_{\rho}$ the unique primitive vector in $N$ such that $v_{\rho} \in \rho$. Thus, the set $\Sigma(1)$ can be identified with the set of these primitive vectors. Sometimes we will write $v_{\rho} \in \Sigma(1)$ by abuse of notation. Recall that the function field of $X_{\Sigma}$ is equal to the fraction field of the $\C$-algebra generated by the characters $\chi^u$ with $u \in M = N^{\vee}$. Since $\chi^u$ is regular and non-zero on the torus, the principal divisor $div(\chi^u)$ is supported in $\displaystyle \bigcup_{\rho \in \Sigma(1)} D_{\rho}$. 

\begin{prop}[{\cite[4.1.4]{cox}}] With the notations as above, we have:
$$div(\chi^u) = \displaystyle \sum_{\rho \in \Sigma(1)} \langle u, v_{\rho} \rangle D_{\rho}.
$$
\end{prop}

Recall that the class group, $Cl(X)$, of a normal variety $X$ is the quotient group of Weil divisors on $X$ modulo linear equivalence. The class group of a toric variety has a nice presentation in terms of the $T$-invariant divisors. Every divisor $D$ on $X_{\Sigma}$ is linearly equivalent to a $T$-invariant divisor $\displaystyle \sum_{\rho \in \Sigma(1)} a_{\rho} D_{\rho}$. This is a consequence of the following proposition.

\begin{prop}[{\cite[4.1.3]{cox}}] There exists an exact sequence: 
$$
M \longrightarrow \mbox{T-Div}(X_{\Sigma}) \longrightarrow Cl(X_{\Sigma}) \longrightarrow 0,
$$
where $\mbox{T-Div}(X_{\Sigma})$ is the abelian group of $T$-invariant Weil divisors. The first map is $u \mapsto div(\chi^u)$ and the second one sends an invariant divisor to its class in $Cl(X_{\Sigma})$. If there is a maximal cone in $\Sigma$ then the sequence is exact on the left:
\begin{equation}
0 \longrightarrow M \longrightarrow \mbox{T-Div}(X_{\Sigma}) \longrightarrow Cl(X_{\Sigma}) \longrightarrow 0.
\label{exact_sequence}
\end{equation}  

\end{prop}

\begin{rem} The fact that the class group of a toric variety is generated by $T$-invariant divisors generalizes to cycles of arbitrary dimension: let $A_k(X_{\Sigma})$ be the Chow group of order $k$, i.e., the group of cycles of dimension $k$ modulo rational equivalence. This group is generated by $T$-invariant cycles $V(\sigma)$ with dimension $k$, i.e., $\sigma \in \Sigma(n-k)$ (see \cite[5.1]{fulton}). 

\label{chow_toric}
\end{rem}

When $D$ is a $T$-invariant Weil divisor on $X_{\Sigma}$, the global sections of the coherent sheaf $\mathcal{O}_{X_{\Sigma}}(D)$, defined by:  

$$U \mapsto \mathcal{O}_{X_{\Sigma}}(D)(U) := \Big\{ f \in K(X_{\Sigma}) \ \Big| \ \big(D + div(f)\big)_{|U} \geq 0 \ \Big\},
$$ 
can be described in terms of characters of the torus $T$. In fact, we have (see \cite[4.3.2]{cox}):
\begin{equation}
\Gamma(X_{\Sigma}, \mathcal{O}_{X_{\Sigma}}(D)) = \displaystyle \bigoplus_{div(\chi^u)+D \geq 0} \C \cdot \chi^u.
\label{global_sections}
\end{equation}

In particular, the characters $\chi^u$ with $div(\chi^u)+D \geq 0$ form a basis for the complex vector space $\Gamma(X_{\Sigma}, \mathcal{O}_{X_{\Sigma}}(D))$. When $X_{\Sigma}$ is complete, this vector space has finite dimension (see for instance {\cite[Vol.2, VI3.4]{shafa}}).

\begin{prop} Every effective Weil divisor $D$ on $X_{\Sigma}$ is linearly equivalent to an effective $T$-invariant divisor. 

\label{eff_toric_div}
\end{prop}

\begin{proof}
Since every divisor on $X_{\Sigma}$ is linearly equivalent to a $T$-invariant divisor, there exists $D'= \displaystyle \sum_{\rho \in \Sigma(1)} a_{\rho} D_{\rho}$ with $D \sim D'$. The vector space $\Gamma(X_{\Sigma}, \mathcal{O}_{X_{\Sigma}}(D))$ is isomorphic to $\Gamma(X_{\Sigma}, \mathcal{O}_{X_{\Sigma}}(D'))$ and since $D$ is effective, these vector spaces are non-zero.  Because $D'$ is $T$-invariant, there exists a character $\chi^u$ with $D''=div(\chi^u) + D' \geq 0$. Thus, $D''$ is an effective $T$-invariant divisor linearly equivalent to $D$. 

\end{proof}

We also recall some properties of $T$-invariant Cartier divisors. Such divisors have special coordinate data given by characters defined in the open sets $\mathcal{U}_{\sigma}$. More precisely, let $D$ be a $T$-invariant Cartier divisor on $X_{\Sigma}$ and write $D = \displaystyle \sum_{\rho \in \Sigma(1)} a_{\rho}D_{\rho}$. For every cone $\sigma \in \Sigma$, there exists $u_{\sigma} \in M$ such that $D_{|\mathcal{U}_{\sigma}}$ is principal in $\mathcal{U}_{\sigma}$ and its local equation in this open set is given by the character $\chi^{-u_{\sigma}}$. If $D_\rho$ is an invariant divisor intersecting $\mathcal{U}_{\sigma}$, then $\langle u_{\sigma}, v_{\rho} \rangle = -a_{\rho}$.\\

If the fan $\Sigma$ contains a maximal cone of $N_{\R}$ then the Picard Group $\Pic(X_{\Sigma})$ is a free abelian group. In particular, when $X_{\Sigma}$ is complete and $\Q$-factorial, $\Pic(X_{\Sigma})_{\R} = Cl(X_{\Sigma})_{\R}$. Tensorizing Equation \ref{exact_sequence} by $\R$ and using this last equality, we conclude that the vector space $\Pic(X_{\Sigma})_{\R}$ has dimension $\# \Sigma(1) - n$. Therefore, when $\Sigma$ is complete and simplicial, the rank of $\Pic(X_{\Sigma})$ is equal to $\# \Sigma(1) - n$.\\

The canonical class of a toric variety $X_{\Sigma}$ is represented by the following $T$-invariant divisor (see for instance \cite[4.3]{fulton}):
$$
K_{X_{\Sigma}} = \displaystyle - \sum_{\rho \in \Sigma(1)} D_{\rho}.
$$
In particular, when $X_{\Sigma}$ is complete, the canonical divisor of a toric variety is not effective.\\

Next, we describe intersection products of invariant Cartier divisors with $T$-invariant cycles on a toric variety $X_{\Sigma}$. We refer to \cite[5.1]{fulton} for proofs:

\begin{prop} Let $X_{\Sigma}$ be an arbitrary toric variety. Let $D = \displaystyle \sum_{\rho \in \Sigma(1)} a_{\rho}D_{\rho}$ be a $T$-invariant Cartier divisor and $V(\sigma)$ an invariant subvariety of $X_{\Sigma}$ not contained in the support of $D$. Then:
$$
D \cdot V(\sigma) = \sum b_{\tau}V(\tau),
$$ 
the sum over all cones $\tau \in \Sigma$ containing $\sigma$ with $\dim(\tau) = \dim(\sigma) + 1$. The coefficients $b_{\tau}$ are computed as follows: let $v_{\rho}$ be any primitive lattice generator of $\tau$ not contained in $\sigma$ and let $e \in N_{\tau}/N_{\sigma}$ be such that the class of $v_{\rho}$ in this quotient is $\bar{v}_{\rho} = s_{\rho} \cdot e$, $s_{\rho} > 0$. Then:
$$
b_{\tau} = \frac{a_{\rho}}{s_{\rho}}.
$$
\label{toric_intersections}
\end{prop}

\begin{cor} Let the assumptions be as in Proposition \ref{toric_intersections}. If $X_{\Sigma}$ is smooth and $D=D_{\rho}$ we have:
$$
D_{\rho} \cdot V(\sigma) = 
\left\{ \begin{array}{ll}
V(\tau), \ \mbox{if} \ \sigma \ \mbox{and} \ \rho \ \mbox{span a cone} \ \tau \in \Sigma;\\
0, \ \mbox{otherwise.}\\
\end{array}
\right.
$$
\end{cor}

\section{Polytopes}

Let $N$ be a lattice of rank $n$ and $M$ its dual lattice. A \emph{polytope} in $M_{\R}$ is the convex hull of a finite number of points of $M_{\R}$. The \emph{dimension} of a polytope $P$ is the dimension of the smallest affine subspace of $M_{\R}$ containing $P$. When $\dim(P)=n$, $P$ is said to be \emph{full dimensional}. Given a nonzero vector $v \in N_{\R}$ and $a \in \R$, we define the affine hyperplane $H_{v,a}$ and the closed half-space $H^+_{v,a}$ in $M_{\R}$ by:
$$
H_{v,a} = \{u \in M_{\R} \ / \ \langle u, v \rangle = -a\} \ \ \textrm{and} \ \ H^+_{v,a} = \{u \in M_{\R} \ / \ \langle u, v \rangle \geq -a \}.
$$

A nonempty subset $Q$ of a polytope $P \subset M_{\R}$ is called a \emph{face}, and we write $Q \prec P$, if there exist $v \in N_{\R}-\{0\}$ and $a \in \R$ as above such that $Q = H_{v,a} \cap P$ and $P \subset H^+_{v, a}$. Such a hyperplane $H_{v,a}$ is called a \emph{supporting hyperplane} of the face $Q$. Faces of $P$ are polytopes again. We also consider $P$ as a face of itself. A face $Q$ is called a \emph{vertex} (respectively a \emph{facet}) of $P$ if $\dim(Q) = 0$ (respectively $\dim(Q) = \dim(P) - 1$). We say that $P$ is a \emph{lattice} polytope if its vertices belong to $M$. 

\begin{figure}[h]
\centering
\includegraphics[scale=0.5]{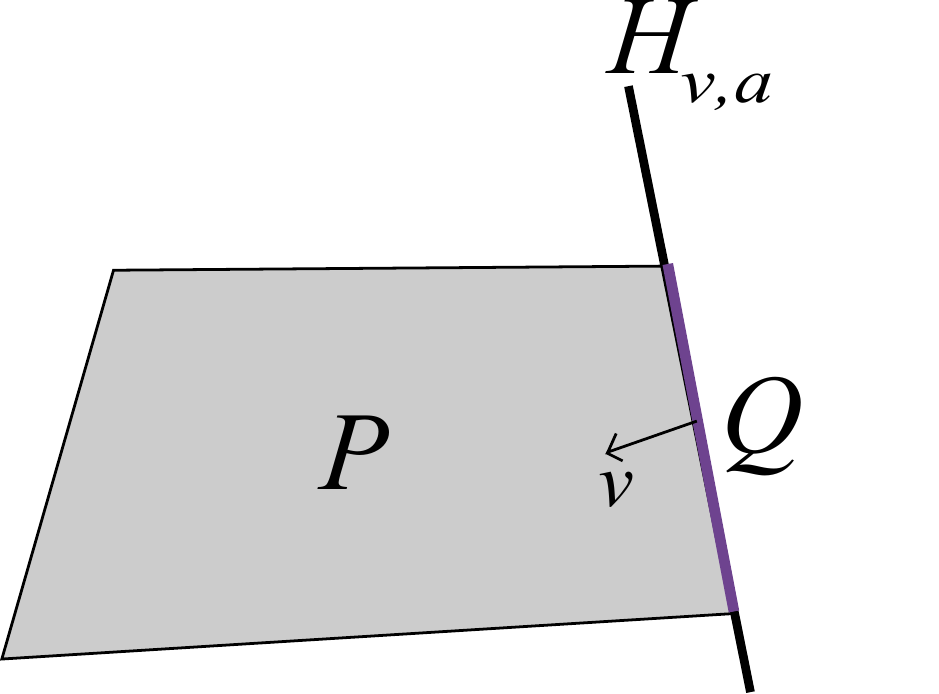}
\caption{}
\end{figure}

We are interested in full dimensional lattice polytopes $P$ in $M_{\R}$. In this case, for each facet $F$ of $P$, there is an unique supporting hyperplane $H_{v_F, a_F}$ of $F$, where $v_F$ is a primitive lattice vector. Thus, we can write $P$ as intersection of closed half-spaces determined by the supporting hyperplanes of the facets of $P$. This is called the \emph{facet presentation of P}:
\begin{equation}
P = \displaystyle \bigcap_{F \ \textrm{facet}} H^+_{v_F, a_F}.
\label{facet_presentation}
\end{equation}

\subsection{Polytopes Versus Toric Varieties}
\label{polytope_toric}

There are rich connections between the theory of polytopes and the theory of toric varieties. We can construct toric varieties from lattice polytopes.
Suppose that $P$ is a lattice polytope with maximal dimension in $M_{\R}$. Consider its facet presentation: 
$$
P = \displaystyle \bigcap_{F \ \textrm{facet}} H^+_{v_F, a_F}.
$$
For each face $Q$ of $P$ we define the cones in $N_{\R}$:
$$
\sigma_Q:= {\Cone}\Big(v_F \ \Big| \ Q \prec F, \ F \ \mbox{facet} \Big).
$$
We have $\dim(\sigma_Q)= n - \dim(Q)$. In particular, the vertices of $Q$ provide cones in $N_{\R}$ with dimension $n$. The collection $\Sigma_P := \{ \sigma_Q \ | \ Q \prec P\}$ is a complete fan in $N_{\R}$, called the \emph{normal fan of P}. 

\begin{de} Let $P$ be an $n$-dimensional lattice polytope in $M_{\R}$. We say that $P$ is \emph{smooth}, if for each vertex $v$ of $P$, the primitive inner normal lattice vectors defining the facets incidents to $v$ form a basis for $N$. Equivalently, $P$ is smooth if and only if the normal fan $\Sigma_P$ is smooth. 
\end{de}

\begin{figure}[h]
\centering
\includegraphics[scale=0.7]{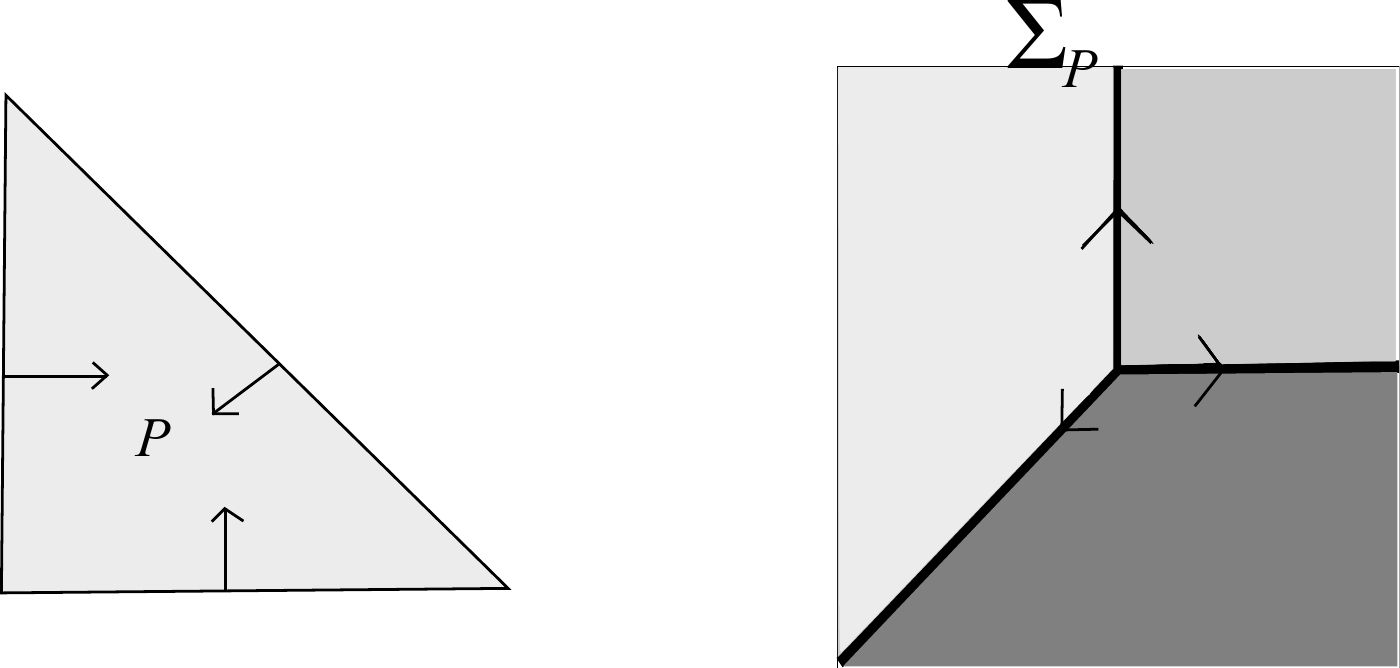}
\caption{The normal fan of the standard simplex is the fan of $\mathbb{P}^2$.}
\end{figure}

Given a full dimensional lattice polytope $P$ in $M_{\R}$, we define $X_P$ to be the toric variety corresponding to $\Sigma_P$. The divisor $D_P := \sum a_FD_{v_F}$ is an ample divisor on $X_P$, hence $X_P$ is a projective variety (see for instance \cite[6.1.15]{cox}).\\ 

It is also possible to construct polytopes from complete toric varieties: let $X_{\Sigma}$ be a complete $n$-dimensional toric variety. Given any Weil divisor $D = \displaystyle \sum_{\rho \in \Sigma(1)} a_{\rho}D_{\rho}$ on $X_{\Sigma}$ we can define a polytope (possibly empty!) $P_D$ as:
$$
P_D = \displaystyle \bigcap_{\rho \in \Sigma(1)} H^+_{v_{\rho},a_{\rho}}.
$$

In fact, to see that $P_D$ is a polytope, it is sufficient to see that $P_D$ is bounded. To prove this, we see that $P_D$ has a finite number of lattice points. We have from (\ref{global_sections}):
\begin{align*}
u \in P_D & \iff \langle u, v_\rho \rangle \geq -a_{\rho} \ \forall \rho \in \Sigma(1) \notag \\ 
 & \iff div(\chi^u) + D \geq 0 \notag \\ 
 & \iff \chi^u \in \Gamma(X_{\Sigma}, \mathcal{O}_{X_{\Sigma}}(D)).
\end{align*}  
Hence:
\begin{equation}
\Gamma(X_{\Sigma}, \mathcal{O}_{X_{\Sigma}}(D)) = \displaystyle \bigoplus_{u \in P_D \cap M} \C \cdot \chi^u.
\label{global_sections_pol}
\end{equation}
Since $X_{\Sigma}$ is complete, $\Gamma(X_{\Sigma}, \mathcal{O}_{X_{\Sigma}}(D))$ is a finite dimensional vector space, therefore, there are finitely many lattice points in $P_D$. Moreover, the dimension of the complex vector space $\Gamma(X_{\Sigma}, \mathcal{O}_{X_{\Sigma}}(D))$ is equal to the number of lattice points lying in $P_D$. 

However, the polytope $P_D$ can fail to be a lattice polytope, or to be full dimensional. The following proposition shows how these and other combinatorial properties of the polytope $P$ are reflected by geometric properties of the divisor $D$ (see for instance \cite[6.1.10 and 6.1.15]{cox}). 

\begin{prop} Let $\Sigma$ be a complete fan in $N_{\R}$ and $D$ a $T$-invariant Cartier divisor on the toric variety $X_{\Sigma}$. Then:
\begin{enumerate} 

\item[(a)] If $D$ is basepoint free then $P_D$ is a lattice polytope;

\item[(b)] If $D$ is ample, then $P_D$ is a full dimensional lattice polytope, $\Sigma = \Sigma_{P_D}$ and $D=D_{P_D}$.

\end{enumerate}

\label{polytope_div_conections}
\end{prop}

It would be interesting to know weaker conditions on a $T$-invariant Cartier divisor $D$ that imply that $P_D$ is a lattice polytope.

\begin{cor} There exists a one-to-one correspondence:
$$
\left\{ 
\begin{array}{cc} 
\mbox{full dimensional lattice} \\
\mbox{polytopes in} \  M_{\R}                   
\end{array} \right\} \longleftrightarrow
\left\{
\begin{array}{cc}
n \mbox{-dimensional polarized} \\
\mbox{toric varieties} 
\end{array} \right\},
$$
that associates to each full dimensional lattice polytope $P$  the polarized toric variety $(X_P, D_P)$.

\end{cor}

\begin{de} Let $M$ and $M'$ be two $n$-dimensional lattices. 
An \emph{affine isomophism} of $M_{\R}$ and $M'_{\R}$ is a linear isomorphism between these two vector spaces obtained by tensorizing by $\R$ a lattice isomorphism $\phi:M \to M'$.
We say that two full dimensional lattice polytopes $P \subset M_{\R}$ and $P' \subset M_{\R}$ are \emph{affinelly isomorphic} if there exists an affine isomorphism $\varphi: M_{\R} \to M'_{\R}$ that applies $P$ bijectively onto $P'$. 

\label{aff_isom}
\end{de}

\begin{rem} Using the notation of Definition \ref{aff_isom}, the polytopes $P$ and $P'$ are affinelly isomorphic if and only if there exists a toric isomorphism $f: X_P \to X_{P'}$ with $f^*(D_{P'}) = D_P$.

\label{aff_isom_toric_pol}
\end{rem}

\section{Basics on Mori Theory and Adjunction Theory}

\subsection{Cones of cycles}
\label{cones_cycles}

In this section we introduce basic notions about cones of cycles, recall some properties and introduce notation. Throughout this section, $X$ will be a normal projective variety. 

 We denote by $NS(X)$ the Neron-Severi group of $X$ and by $N^1(X)$ the tensor product $NS(X) \otimes \R$. It is well known that $N^1(X)$ is a finite dimensional vector space and its dimension, denoted by $\rho(X)$, is called the \emph{Picard number} of $X$. If $D = \sum a_iD_i$ is an $\R$-linear combination of Cartier divisors on $X$ we denote by $[D]$ its corresponding class in $N^1(X)$. 

There are two important cones of divisors to consider: we recall that a Cartier divisor $D$ is \emph{numerically effective} (\emph{nef} for short) when the intersection product $D \cdot C$ is nonnegative for every complete integral curve $C \subset X$. The \emph{nef cone} $\overline{Nef}(X)$ is the closure in $N^1(X)$ of the cone generated by the classes of nef divisors. By Kleiman's Ampleness criterion, a Cartier divisor $D$ is ample if and only if its class $[D]$ belongs to the interior of $Nef(X)$. The \emph{pseudo-effective cone} $\overline{Eff}(X)$ is the the closure of the cone generated by the class of effective divisors on $X$. Since every very ample Cartier divisor is linearly equivalent to an effective divisor, we have $\overline{Nef}(X) \subset \overline{Eff}(X)$.

Let $Z_1(X)$ be the free group generated by the complete integral curves contained in $X$. The elements of $Z_1(X)$ are called \emph{1-cycles}. When two cycles $C$ and $C'$ satisfy $D \cdot C = D \cdot C'$ for every Cartier divisor $D$ on $X$, we say that these cycles are \emph{numerically equivalent} and write $C \equiv C'$. We see that $`` \equiv "$ is an equivalence relation. Let $N_1(X)$ be the vector space $\big(Z_1(X)/\equiv \big) \otimes \R$. Similarly to the divisor case, given a 1-cycle $C$ on $X$ we denote by $[C]$ the corresponding class in $N_1(X)$. The intersection product of Cartier divisors with complete integral curves establishes a perfect pairing between $N^1(X)$ and $N_1(X)$. Thus $N_1(X)$ has finite dimension, equal to $\rho(X)$. 

The \emph{Mori Cone} of $X$, $\overline{NE}(X)$, is the closure of the cone in $N_1(X)$ generated by the classes of complete integral curves on $X$. By definition, $\overline{NE}(X)$ is the dual to the nef cone $\overline{Nef}(X)$. To describe the dual cone of the pseudo-effective cone $\overline{Eff}(X)$ we need a definition:

\begin{de} A complete integral curve $C \subset X$ is called a \emph{movable} (or a moving curve) on $X$ if $C=C_{t_0}$ is a member of an algebraic family $\{C_t\}_{t \in S}$ of curves on $X$ such that $\displaystyle \bigcup_{t \in S} C_t = X$. 
\end{de}

\begin{ex} Let $(x:y:z:w)$ be the coordinates of points in $\mathbb{P}^3$ and let $X$ be the quadric cone with vertex $O=(0:0:0:1)$ given by the polynominal equation $z^2 = x^2+y^2$. Let $f: X \dashrightarrow \mathbb{P}^2$, $f(x:y:z:w)=(x:y:z)$. Blowing up $X$ at $O$, we obtain a variety $\tilde{X}$ and a morphism $g:\tilde{X} \to \mathbb{P}^2$ that resolves the rational map $f$. The image of $g$ is the conic $\Gamma$ given by the equation $z^2 = x^2+y^2$ in $\mathbb{P}^2$ and this map realizes $\tilde{X}$ as a $\mathbb{P}^1$-bundle over $\Gamma$. Moreover, the strict transform $\tilde{C}$ of a ruling $C$ of the cone $X$ is mapped to $g$ in a point of $\Gamma$. It follows that the rulings of the cone form an algebraic family of curves covering $X$ parametrized by $\Gamma$. Consequently, every ruling of the cone is movable on $X$. 

\end{ex}

The \emph{cone of moving curves} of $X$, $\overline{Mov}(X)$, is the closure in $N_1(X)$ of the cone generated by the classes of movable curves on $X$. By definition, $\overline{Mov}(X) \subset \overline{NE}(X)$.

\begin{rem} Another description of $\overline{Mov}(X)$ was given in \cite[1.3]{bdpp}. There it is showed that $\overline{Mov}(X)$ is generated by complete integral curves on $X$ called \emph{strongly movable curves}. These curves are obtained by images in $X$ of curves given by complete intersections of suitable very ample divisors on birational modifications of $X$. In particular, the numerical class of any curve given by a complete intersection $D_1 \cap ... \cap D_{n-1}$ of very ample divisors on $X$ belongs to $\overline{Mov}(X)$.
\label{strong_movable_curve}
\end{rem}

\begin{de} A \emph{$\Q$-divisor} (respectively an \emph{$\R$-divisor}) on $X$ is a $\Q$-linear (respectively an $\R$-linear) combination of prime Weil divisors on $X$. A \emph{$\Q$-divisor} (respectively an $\R$-divisor) $D$ on $X$ is called \emph{$\Q$-Cartier} (respectively \emph{$\R$-Cartier}) if $D$ is a $\Q$-linear (respectively an $\R$-linear) combination of Cartier divisors on $X$. An $\R$-Cartier $\R$-divisor $D$ is \emph{ample} (respectively \emph{nef}, respectively \emph{pseudo-effective}) if $[D] \in int(\overline{Nef}(X))$ (respectively $[D] \in \overline{Nef}(X)$, respectively $[D] \in \overline{Eff}(X)$).  
\end{de}

When $X$ is smooth, it is possible to define intersection products between cycles of arbitrary dimension (see \cite[Cap. 8]{fulton_int}). Thus, we can generalize the notions of Mori Cone and effective cone considering cycles on $X$ of dimension $\geq 1$. For each integer $k$, $2 \leq k \leq n-2$ let $Z_k(X)$ be the group of cycles of dimension $k$. Two cycles $V, V' \in Z_k(X)$ are numerically equivalent, written $V \equiv V'$, when $D \cdot V = D \cdot V'$ for every $D \in Z_{n-k}(X)$. The intersection product induces a perfect pairing between the vector spaces $N_k(X) = (Z_k(X)/\equiv) \otimes \R$ and $N^k(X) = (Z_{n-k}(X)/\equiv) \otimes \R$  and these vector spaces have finite dimension again. We denote by $[V]$ the class in $N_k(X)$ of a $k$-cycle $V$ on $X$.

A $k$-cycle $V$ is called \emph{pseudo-effective} if there exist irreducible $k$-dimensional subvarieties $V_i$ such that $V \equiv \sum_i a_iV_i$, $a_i>0$. We define $\overline{NE}_{k}(X)$ to be the closure of the cone generated by the classes of pseudo-effective $k$-cycles on $X$. Note that $\overline{NE}_1(X)$ is the Mori Cone of $X$ and $\overline{NE}_{n-1}(X)$ is the cone $\overline{Eff}(X)$.\\

\vspace{0,3 cm}
\begin{say} 
\label{mov_toric}
\textbf{Specializing to the toric case.} Let $X = X_{\Sigma}$ be a $\Q$-factorial projective toric variety corresponding to an $n$-dimensional fan $\Sigma $ in $N_{\R}$. It is well-known that linear and numerical equivalence coincide on complete toric varieties (see for instance \cite[6.2.15]{cox}). Thus, $N^1(X) \simeq \Pic(X) \otimes \R$. We view $\Pic(X)$ as a lattice for the vector space $N^1(X)$. As we have seen in Section \ref{toric_divisors}, the Picard group $\Pic(X)$ of $X$ is Abelian, free and $rank(\Pic(X)) = \# \Sigma(1) - n$. Hence the Picard number of $X$ is $\rho(X) = \# \Sigma(1) - n$. Moreover, a Cartier divisor is nef if and only if it is basepoint free (\cite[6.1.12]{cox}).

\vspace{0,2 cm}
The cones of curves and divisors defined above in the general case are polyhedral in the toric context. In fact, by Proposition \ref{eff_toric_div}, effective divisors are linearly equivalent to effective $T$-invariant divisors. Thus:
\begin{center}
$\overline{Eff}(X) = \Cone\Big( \ [D] \ \Big| \ D$ is a prime $T$-invariant divisor on $X \Big).$
\end{center}

Similarly, the Mori Cone $\overline{NE}(X)$ is generated by the classes of $T$-invariant curves $V(\omega), \ \omega \in \Sigma(n-1)$ (see \cite[1.6]{reid} or \cite[6.2.20]{cox}):

\begin{equation}
\overline{NE}(X) = \displaystyle \sum_{\omega \in \Sigma(n-1)} \R_{\geq 0}\cdot [V(\omega)].
\label{toric_cone_theorem}
\end{equation}

Taking the dual of these cones, we conclude that also the cone of moving curves and the nef cone are polyhedral.\\

We give now a combinatorial description of the elements of $N_1(X)$ when $X$ is a $\Q$-factorial projective toric variety. Tensoring the exact sequence (\ref{exact_sequence}) by $\R$ and dualizing the result, we obtain the following exact sequence: 

\begin{equation}
0 \longrightarrow N_1(X) \stackrel{\alpha}{\longrightarrow} {\R}^{\# \Sigma(1)} \stackrel{\beta}{\longrightarrow} N_{\R} \longrightarrow 0. 
\end{equation}

The map $\alpha$ sends each class $[C]$ of an integral complete curve $C$ to $(D_{\rho} \cdot C)_{\rho \in \Sigma(1)}$ and $\beta$ sends each canonical vector $e_{\rho}$ to the unique non-zero primitive lattice vector $v_{\rho}$ contained in $\rho$. Hence the elements $c \in N_1(X)$ are identified with the linear relations $\displaystyle \sum_{\rho} a_{\rho}v_{\rho} = 0$ among the minimal generators of the cones in $\Sigma(1)$, where $a_{\rho} = D_{\rho} \cdot c, \ \forall \rho \in \Sigma(1)$. One such linear relation is called \emph{effective} when each coefficient $a_{\rho}$ is nonnegative. Since $\overline{Mov}(X) = \overline{Eff}(X)^{\vee}$ we conclude that $\overline{Mov}(X)$ is generated by effective linear relations.
\end{say}

\subsection{Log singularities and contractions of extremal rays} \label{log}

Let $X$ and $Y$ be normal projective varieties and $g:X \dashrightarrow Y$ a dominant rational map. Let $U_g$ be the largest open subset of $X$ such that $g_{|U_g}$ is a morphism. The \emph{exceptional locus} of $g$, $Exc(g)$, is the complement in $X$ of the largest open set $U \subset X$ such that $g_{|U}: U \to g(U)$ is an isomorphism. Of course if $\dim(X) \neq \dim(Y)$ then $Exc(g) = X$. Suppose now that $g$ is birational. We say that $g$ is \emph{surjective in codimension} 1 if $\codim_Y(Y-Im(g))\geq 2$. In this case, given a Cartier divisor $D=\sum a_i D_i$ on $X$ we define the \emph{pushforward} of $D$ as the divisor $g_*(D)$ on $Y$ given by $g_*(D) = \sum a_i D_i'$, where $D_i' = 0$ if $\dim \overline{g(D_i \cap U_g)} < \dim D_i$, and $D_i' = \overline{g(D_i \cap U_g)}$ if $\dim \overline{g(D_i \cap U_g)} = \dim D_i$. The \emph{pullback} of a Cartier divisor $D'$ on $Y$ is the divisor on $X$ given by $\overline{g_{|U_g}^*(D')}$, where $g_{|U_g}^*(D')$ is the Cartier divisor on $U_g$ obtained by pulling back to $U_g$ the local equations defining the divisor $D'$ on $Y$. These two maps induce an injective linear map $g^*:N^1(Y) \to N^1(X)$ and a surjective linear map $g_*:N^1(X) \to N^1(Y)$. We say that $g$ is a \emph{small modification} when $g$ and $g^{-1}$ are surjective in codimension 1. Note that when $g$ is a small modification, $g_*$ and $g^*$ are inverse to each other.  

\vspace{0,2 cm}

Let $W$ be a smooth projective variety and $D = \sum_i D_i$ a divisor on $W$. We say that $D$ has \emph{simple normal crossings} if each $D_i$ is smooth and irreducible and for every $p \in W$, the tangent spaces $T_p(D_i) \subset T_p(W)$ meet transversally, i.e.:
$$
\codim( \displaystyle \bigcap_{p \in D_i} T_p(D_i)) = \# \big\{i \ \big| \ p \in D_i \big\}.
$$

Let $X$ be a normal projective variety and $\Delta = \sum d_jD_j$ a $\Q$-divisor on $X$, where the $D_j$'s are prime divisors on $X$. Let $W$ be a smooth projective variety and $f: W \to X$ a birational morphism. If $U$ is the largest open subset in $X$ such that $(f^{-1})_{|U}$ is a morphism, we define:
$$
f_*^{-1}(\Delta) = \sum_j d_j\overline{f^{-1}(D_j \cap U)}.
$$ 
Suppose that $Exc(f)$ is the union of prime divisors $E_i's$ on $W$. We say that $f$ is a \emph{log resolution} of $(X, \Delta)$ if the divisor $\sum_i E_i + f_*^{-1}(\Delta)$ has simple normal crossings.
If $K_X+\Delta$ is $\Q$-Cartier, we can pick a canonical divisor $K_{W}$ such that (see \cite[Sec. 2.3]{km98}):
$$
K_W + f^{-1}_*(\Delta) = f^*(K_X+\Delta) + \sum_i a_iE_i,
$$
where the coefficients $a_i \in \R$ do not depend of the log resolution $f$, they depend only on the valuation determined by the $E_i$.

\begin{de} Let $(X, \Delta)$ be a pair such that $K_X + \Delta$ is $\Q$-Cartier and $\Delta = \sum d_jD_j$ a $\Q$-divisor with $d_j \in [0,1)$ and $D_j$ is a prime divisor $\forall j$. If there is a log resolution $f:W\to X$ of $(X, \Delta)$ as above such that $a_i > -1, \ \forall i$, we say that the pair $(X, \Delta)$ has \emph{Kawamata log terminal} (\emph{klt} for short) \emph{singularities}. In addition, if $X$ is $\Q$-factorial, we say that $(X, \Delta)$ is a \emph{$\Q$-factorial klt pair}. If $X$ is $\Q$-factorial, $\Delta = 0$ and the $a_i$ are positive, we say that $X$ has \emph{terminal singularities}. 

\label{klt}
\end{de}

Pairs $(X, \Delta)$ with klt singularities are very special. For instance, there is a good description of the part of their Mori Cone $\overline{NE}(X)$ contained in the halfspace $K_X + \Delta <0$ (see for instance \cite[Theorem 3.7]{km98}):

\begin{teo}[Cone Theorem] Let $(X, \Delta)$ be a pair with klt singularities. There is a countable set $\Gamma \subset \overline{NE}(X)$ of classes of rational curves $C \subset X$ such that $0<-(K_X + \Delta )\cdot C \leq 2\dim(X)$ and:
$$
\overline{NE}(X) = \overline{NE}(X)_{K_X + \Delta  \geq 0} + \displaystyle \sum_{[C] \in \Gamma} \R_{\geq0}[C].
$$

Moreover, for every ample $\R$-divisor $A$ on $X$, we can pick finitely many classes $[C_1], ..., [C_l] \in \Gamma$ such that:

$$ \overline{NE}(X) = \overline{NE}(X)_{K_X + \Delta + A \geq 0} + \displaystyle \sum_{i=1}^l \R_{\geq0}[C_i].
$$

\label{cone_theorem}
\end{teo}

\begin{figure}[h]
\centering
\includegraphics[scale=0.5]{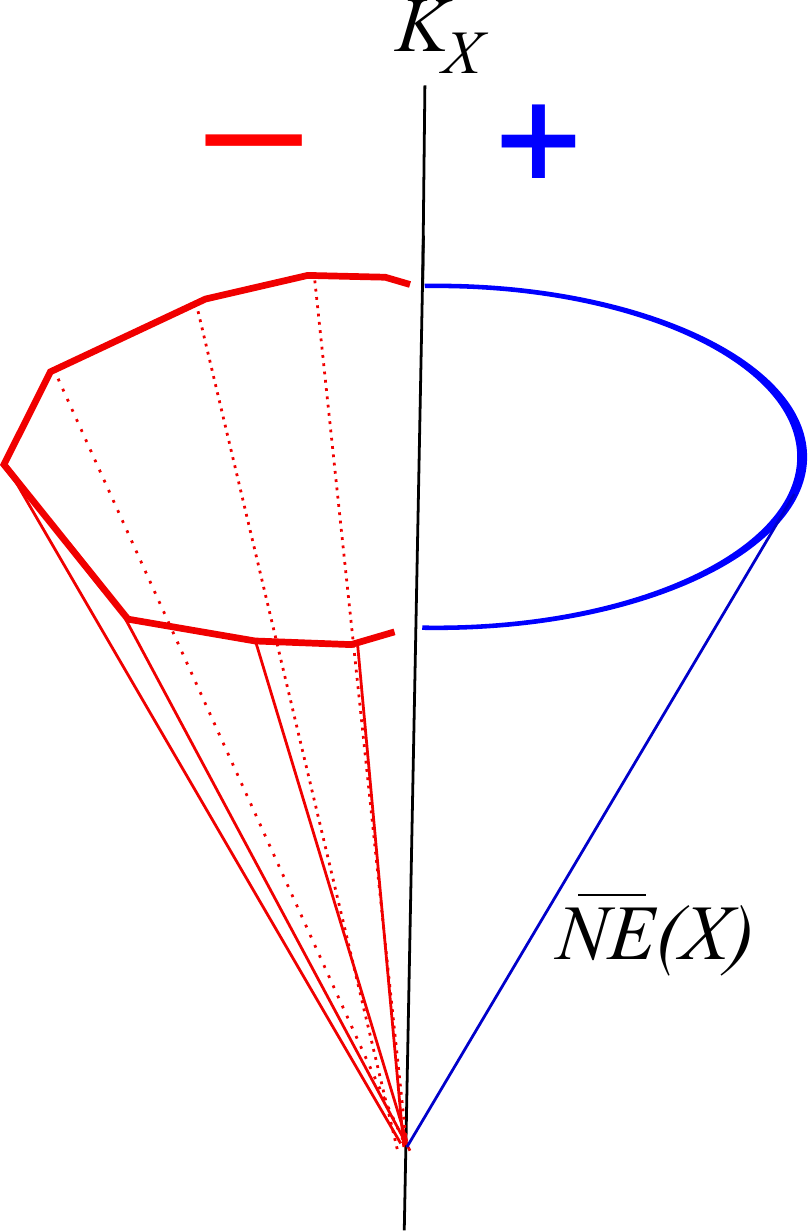}
\caption{The Cone Theorem.}
\end{figure}

The Mori Cone of a normal variety $X$ codifies information about morphisms of $X$: a \emph{contraction} of a face $F \prec \overline{NE}(X)$ is a morphism $\phi_F: X \to Y$ with connected fibers onto a normal variety $Y$, such that the image $\phi_F(C)$ of a curve $C \subset X$ is a point if and only if $[C] \in F$. 

Not every face $F$ of $\overline{NE}(X)$ provides a contraction $\phi_F: X \to Y$. However, if $\phi_F$ exists, it is unique up to isomorphism of $Y$ (as a consequence of the Stein Factorization \cite[III, 11.5]{hartshorne}).

The following theorem gives sufficient conditions to guarantee the existence of such contractions (see for instance \cite[Theorem 3.7.3]{km98}).

\begin{teo}[Contraction Theorem] Let $(X,\Delta)$ be a $\Q$-factorial klt pair and $F \prec \overline{NE}(X)$ a face contained in the halfspace $K_X + \Delta <0$. Then the contraction $\phi_F$ of the face $F$ described above exists.
\label{contraction}
\end{teo}
 
\begin{say}
 Let $(X, \Delta)$ be a $\Q$-factorial klt pair. An important type of contraction, useful specially for the Minimal Model Program, is the contraction of a $(K_X + \Delta)$-\emph{negative extremal ray} $R$, i.e., $R = \R_{\geq 0}[C]$ is an extremal ray of $\overline{NE}(X)$ such that $(K_X + \Delta) \cdot C < 0$. By simplicity, we sometimes write $(K_X + \Delta) \cdot R<0$ to refer to $(K_X+\Delta)$-negative extremal rays. For such a ray $R$, Theorem \ref{contraction} guarantees that there exists a contraction $\phi_R: X \to Y$ of $R$. One can show that there is an exact sequence $0 \longrightarrow \Pic(Y) \stackrel{\phi^*}{\longrightarrow} \Pic(X) \stackrel{g}\longrightarrow \Z$, where $g(D) = D \cdot C$. It follows that $\rho(X) = \rho(Y) +1$. Moreover, by \cite[Prop.2.5]{km98}, $\phi_R$ satisfies one of the properties below, according to the dimension of $Exc(\phi_R)$:
 
\begin{enumerate}
 
\item $Exc(\phi_{R}) = X$. In this case, we have $\dim(Y) < \dim(X)$ and $-(K_X+\Delta)$ is $\phi_R$-ample. We say that $\phi_R:X \to Y$ is a \emph{Mori Fiber space} or a \emph{fiber type contraction};
 
\item $\dim (Exc(\phi_R)) = \dim(X)-1$. The morphism $\phi_R$ is birational, its exceptional locus consists of a prime divisor and $(Y, {\phi_R}_* \Delta)$ is a $\Q$-factorial klt pair. We call such $\phi_R$ a \emph{divisorial contraction};
 
\item $ \dim (Exc(\phi_R)) < \dim(X) - 1$. The morphism $\phi_R:X \to Y$ is birational and $Y$ is not $\Q$-factorial (in fact, $K_Y$ is not $\Q$-Cartier). We say that $\phi_R$ is a \emph{small contraction}.
 
\end{enumerate}

\label{contraction_types}
\end{say}

\noindent \textbf{Length of an extremal ray.} Let $X$ be a smooth projective variety, $R$ a $K_X$-negative extremal ray of the Mori cone $\overline{NE}(X)$ and $\phi_R:X\to Y$ the contraction of $R$. 
The {\it lenght} of  $R$ is
$$
\mathfrak{l}(R) := min \big\{-K_X \cdot C \ \big| \ C \subset X \text{ rational curve 
contracted by }\phi_R \big\}.
$$ 
It satisfies $\mathfrak{l}(R) \leq \dim(X)+1$ (see \cite[Theorem 3.3]{km98}). We denote by $E_R\subset X$ the exceptional locus of $\phi_R$, i.e., 
the locus of points at which $\phi_R$ is not an isomorphism. 
The following result is due to Wi\'sniewski (see  \cite[Theorem 6.36]{beltrametti}).

\begin{prop} Let $R$ be an extremal ray of $\overline{NE}(X)$, $E$ any component of the exceptional locus of $E_R$, and 
$F$ any component of a fiber of the restriction $\phi_R |_E$. Then the inequality below holds:
\begin{equation}
\dim(E) + \dim(F) \geq \dim(X) + \mathfrak{l}(R) -1.
\label{inequality}
\end{equation}

\label{key_inequality}
\end{prop}

\begin{say}\textbf{Specializing to Toric Case.}
\label{toric_contraction}
Let $X=X_{\Sigma}$ be a projective toric variety associated to an $n$-dimensional fan $\Sigma$ in $N_{\R}$. If $\Delta = \displaystyle \sum_{\rho \in \Sigma(1)} b_{\rho}D_{\rho}$ with $b_{\rho} \in [0,1)\cap \Q, \ \forall \rho$ and $K_X+\Delta$ is $\Q$-Cartier, then $(X,\Delta)$ is klt (see for instance \cite[11.4.24]{cox}).
Suppose now that $X$ is $\Q$-factorial. Reid proved that every extremal ray $R$ of $\overline{NE}(X)$ (not necessarily with $(K_X + \Delta)\cdot R <0$) determines a contraction $\phi_R: X \to Y$. The variety $Y$ is toric again and its fan is determined by a simple change in $\Sigma$ that we describe now (see \cite{reid} for the proofs): let $\omega \in \Sigma(n-1)$ be a cone such that $[V(\omega)] \in R$. Since $X$ is complete, there are cones $\sigma, \sigma' \subset \Sigma(n)$ such that $\omega = \sigma \cap \sigma'$. Let $v_1, ..., v_n, v_{n+1}$ be primitive lattice vectors such that:
\begin{align*}
\sigma &= \Cone(v_1, ..., v_{n-1}, v_n) \notag \\ 
\sigma' &= \Cone(v_1, ..., v_{n-1}, v_{n+1}).
\end{align*}
 
There is a linear relation $a_1v_1 + ... + a_nv_n + a_{n+1}v_{n+1} = 0$ such that $a_i = V(v_i) \cdot V(\omega) \in \Q$. Because $v_n$ and $v_{n+1}$ are in opposite sides of the hyperplane determined by $\omega$, we have $a_n, a_{n+1}>0$. Reordering the coefficients if necessary, we may suppose that there are integers $0 \leq \alpha \leq \beta \leq n-1$ such that $a_i <0$ if $i<\alpha$, $a_i = 0$ if $\alpha + 1 \leq i \leq \beta$ and $a_i > 0$ if $\beta + 1 \leq i \leq n$. Set $\sigma(\omega) = \Cone(\sigma \cup \sigma')$. One can show that the integers $\alpha$, $\beta$ and the cone:
$$
U = U(\omega) = \Cone(v_1, ..., v_{\alpha}, v_{\beta + 1}, ..., v_n,  v_{n+1})
$$
do not depend on the wall $\omega$ with $[V(\omega)] \in R$. The cone $U$ is a face of $\sigma(\omega)$ and is a vector space if and only if $\alpha = 0$. Moreover, $\sigma(\omega)$ has a decomposition:
$$
\sigma(\omega) = \displaystyle \bigcup_{j=\beta+1}^{j=n+1} \sigma_j
$$
where each $\sigma_j = \Cone(v_1, ..., \hat{v}_j, ..., v_{n+1}) \in \Sigma, \ \forall j = \beta + 1,... , n+1$.  

Let $\Sigma^*$ be the collection of cones in $N_{\R}$ given by the following maximal cones and their faces:

\begin{itemize}
\item $\sigma(\omega)$ with $[V(\omega)] \in R;$

\item $\tau \in \Sigma(n)$ such that there is no wall $\omega \prec \tau$ with $[V(\omega)] \in R$.
\end{itemize}

We have three cases to consider:

\begin{enumerate}
\item[(I)] $(\alpha = 0)$: The collection $\Sigma^*$ is a degenerated fan with vertex $U$. The fan of $Y$ is the fan in $N_{\R}/U$ obtained by the projection of $\Sigma^*$ on this vector space (see \ref{fan}). The map $\phi_R: X \to Y$ is a Mori fiber space and its general fiber is a toric variety with Picard number 1. Furthermore, $\phi_R$ is a flat morphism. The contracted curve $V(\omega)$ is movable on $X$ and its corresponding effective linear relation in $\Sigma(1)$ is: 
\begin{equation} 
a_{\beta + 1}v_{\beta + 1}+ ... + a_nv_n + a_{n+1}v_{n+1} = 0,
\label{eff_comb}
\end{equation}
where Span$(v_{\beta+1}, ..., v_{n+1}) \simeq {\R}^{n-\beta}$;
\vspace{0,2 cm}

\item[(II)] $(\alpha = 1)$: The collection $\Sigma^*$ is a simplicial fan and $Y = X_{\Sigma^*}$. The map $\phi_R$ is a divisorial contraction, $Exc(\phi_R)= V(\R_{\geq 0} \cdot v_1)$ and $\Sigma(1)-\Sigma^*(1) = {v_1}$; 

\item[(III)] $(\alpha > 1)$: The collection $\Sigma^*$ is a fan, nonsimplicial, and $Y = X_{\Sigma^*}$. 
The map $\phi_R$ is a small contraction and $Exc(\phi_R)= V \big(\Cone(v_1,...,v_{\alpha})\big)$. 

\end{enumerate}
\end{say}

\subsection{The Minimal Model Program}
\label{mmp}

Let $(X,\Delta)$ be a $\Q$-factorial klt pair. The aim of the Minimal Model Program (MMP for short) consists essentially in obtaining a birational map $\varphi: X \dashrightarrow Y$ such that $(Y, \varphi_* \Delta)$ is a $\Q$-factorial klt pair and one of the following situations occur:

\begin{enumerate}

\item $K_Y + \varphi_* \Delta$ is nef; or

\item There exists a Mori fiber space $f: Y \to Z$.

\end{enumerate}
 
To describe the steps of the MMP, we begin by supposing that $K_X + \Delta$ is not nef. There exists an extremal ray such that $(K_X + \Delta)\cdot R < 0$. Theorem \ref{contraction} provides a contraction $\phi_R: X \to Y$. If $\phi_R$ is a Mori fiber space, then we are done. If $\phi_R$ is a divisorial contraction, then by Section \ref{log}, $Y$ is birational to $X$ and $(Y, {\phi_R}_* \Delta)$ is a $\Q$-factorial klt pair. If $K_Y + {\phi_R}_* \Delta$ is nef then we arrive to the situation (1) and we are done. Otherwise, there is an extremal ray $R'$ of $\overline{NE}(Y)$ such that $(K_Y + {\phi_R}_* \Delta) \cdot R' < 0$ and we can repeat the process above to $Y$ contracting this ray. Note that there is not an infinite chain of divisorial contractions $X \rightarrow X_1 \rightarrow ... \rightarrow X_k \rightarrow ...$ because the Picard number of these varieties drop by 1 at each such contraction.

Suppose now that $\phi_R: X \to Y$ is a small contraction. While $Y$ is birational to $X$, $K_Y$ always fails to be $\Q$-Cartier. To get around this deficiency, we define a special type of small modification on $X$ whose existence is proved in \cite[Cor. 1.4.1]{bchm}.

\begin{de-teo}[Existence of flips] Let $(X,\Delta)$ be a $\Q$-factorial klt pair and $R$ be a $(K_X+\Delta)$-negative extremal ray such that $\phi_R:X \to Y$ is a small contraction. There exists a projective $\Q$-factorial variety $X^+$, an extremal ray $R^+$ of $\overline{NE}(X^+)$ and small modifications  $\phi:X \dashrightarrow X^+$ and $\phi_{R^+}:X^+ \to Y$, where $(K_{X^+} + \phi_*(\Delta))\cdot R^+ > 0$, such that the following diagram commutes:
$$\xymatrix{
X \ar[rd]_{\phi_R} \ar@{-->}[rr]^{\phi} & & X^+ \ar[ld]^{\phi_{R+}} \\ 
& Y &  
}$$

Moreover, $(X^+,{\phi}_*\Delta)$ is a $\Q$-factorial klt pair and $(K_{X^+}, {\phi}_*\Delta)\cdot R^+ >0$. Such small modification is called a \emph{flip} of the ray $R$.

\label{flip}
\end{de-teo}

Essentially, flipping an extremal ray $R$ consists in making a birational modification in the variety $X$ changing the sign of this ray relatively to $K_X + \Delta$. 

\begin{rem}[See {\cite[Sec.3]{reid}}] Let $X=X_{\Sigma}$ be a $\Q$-factorial projective toric variety corresponding to an $n$-dimensional fan $\Sigma$ in $N_{\R}$. Let $\Delta$ be a $\Q$-divisor such that the pair $(X,\Delta)$ has klt singularities. We can find a toric variety $X^+ = X_{\Sigma^+}$ that satisfies all the conditions of Definition \ref{flip}. The fan $\Sigma^+$ is obtained from $\Sigma$ by a suitable change in the maximal cones that preserves the collection of their extremal rays, i.e., $\Sigma(1) = \Sigma^+(1)$.  

\label{flip_toric}
\end{rem}

Now we can describe the procedure of MMP:

\vspace{0,2 cm}
\emph{Step 1}: Let $(X, \Delta)$ be a $\Q$-factorial klt pair. If $K_X + \Delta$ is nef, then stop. Otherwise, go to Step 2:

\vspace{0,2 cm}
\emph{Step 2}: Pick an extremal ray $R$ of $\overline{NE}(X)$ with $(K_X + \Delta) \cdot R < 0$ and consider the corresponding contraction $\phi_R:X \to Y$. If $\phi_R$ is a Mori fiber space, then we stop with the pair $(X, \Delta)$. Otherwise, go to Step 3.

\vspace{0,2 cm}
\emph{Step3}:

(3.1) If $\phi_R:X \to Y$ is a divisorial contraction, then $(Y, {\phi_R}_*(\Delta))$ is a $\Q$-factorial klt pair. We return to Step 1 with the pair $(Y,{\phi_R}_* \Delta)$ instead of $(X, \Delta)$. 

(3.2) If $\phi_R:X \to Y$ is a small contraction then we pick the flip $\phi: X \dashrightarrow X^+$ of the ray $R$ and return to Step 1 with the pair $(X^+, {\phi}_* \Delta)$ instead of $(X, \Delta)$. 

\vspace{0,3 cm}

\begin{rem} In \cite{mori} Mori proved that the MMP stops in the case of terminal 3-folds. Recently in \cite{bchm} it was proved that if $(X, \Delta)$ is a $\Q$-factorial klt pair such that $K_X + \Delta \notin \overline{Eff}(X)$ then there exists a special type of MMP (where an appropriate choice of the contracted rays is made), named \emph{MMP with scaling}, that terminates with a Mori fiber space. In general, it is not known the MMP always stops. In fact, for the program to terminate it would be necessary to whether ascertain that there is not an infinte sequence of flips. In \cite{reid}, Reid showed that the MMP (not necessarily with scaling) always stops for toric varieties of arbitrary dimension. 
\label{mori_fiberspace}
\end{rem}

\section{Fano manifolds with large index}
\label{ind}

\textbf{Notation:} Given a vector bundle $\mathcal{E}$ on a variety $Y$, we denote by $\mathbb{P}_Y(\mathcal{E})$ the Grothendieck projectivization $\Proj(\Sym(\mathcal{S}_{\mathcal{E}}))$, where $\mathcal{S}_{\mathcal{E}}$ is the sheaf of sections of $\mathcal{E}$.\\

Let $X$ be a smooth projective variety. 
We say that $X$ is a \emph{Fano manifold} if the anticanonical divisor $-K_X$ is ample.
In this case we define the \emph{index} of $X$ as the largest integer $r$ 
dividing $-K_X$ in $\Pic(X)$.
Fano manifolds with large index are very special.
In \cite{wis},  Wi\'sniewski classified $n$-dimensional Fano manifolds
with $n \leq 2r -1$. They satisfy one of the following conditions:

\begin{prop} Let $X$ be a smooth Fano manifold with $\dim(X) = n$ and index $r \geq \frac{n+1}{2}$. Then one of the following holds:

\begin{enumerate}
\item $X$ has Picard number one;
\item $X \simeq \mathbb{P}^{r-1} \times \mathbb{P}^{r-1}$;
\item $X \simeq \mathbb{P}^{r-1} \times Q^r$, where $Q^r$ is an $r$-dimensional smooth hyperquadric;
\item $X \simeq \mathbb{P}_{\mathbb{P}^r}(T_{\mathbb{P}^r})$; or
\item $X \simeq \mathbb{P}_{\mathbb{P}^r}(\mathcal{O}(2) \oplus \mathcal{O}(1)^{r-1})$.
\end{enumerate}

\label{fano_index}
\end{prop}

The following result is classical in the Theory of Toric Varieties. For the reader's convenience, we state and prove it. 

\begin{lemma} Let $X_{\Sigma}$ be an $n$-dimensional smooth projective toric variety with Picard number 1. Then $X_{\Sigma} \simeq \mathbb{P}^n$. In particular, an $n$-dimensional smooth hyperquadric is toric if and only if $n = 2$.

\label{Pic_one}
\end{lemma}

\begin{proof} Let $D$ be any ample divisor on $X$ and consider the polytope $P_D$ as in Section \ref{polytope_toric}. Since $P_D$ has $\# \Sigma(1)$ vertices and $\rho(X_{\Sigma}) = \# \Sigma(1) - n = 1$, we conclude that $P_D$ is an $n$-dimensional simplex in $M_{\R}$. Since $X_{\Sigma} \simeq X_{\Sigma_{P_D}}$, $P_D$ is smooth. Let $v_0, ..., v_n$ be the primitive lattice vectors that defines the facets of $P$. The maximal cones of $\Sigma_{P_D}$ are $\sigma_j = \Cone(v_0, ..., \hat{v}_j, ..., v_n)$, $j=0, ..., n$ and these cones are smooth. Write:
$$
v_0 = a_1v_1 + ... +a_nv_n.
$$ 

Considering the base $\mathcal{B} = \{v_1, ..., v_n\}$ in Definition \ref{multiplicity}, we have $\forall j$:
$$
1 = \mult(\sigma_j) = \big|\det[v_0, ..., \hat{v}_j, ..., v_n]\big| = |a_j|.
$$
By making a linear change of coordinates, we can suppose that $a_j=-1$. It follows that $\Sigma_{P_D}$ is the fan of $\mathbb{P}^n$.

The second part follows from the first: let $Q$ be an $n$-dimensional smooth projective hyperquadric. Since $\rho(Q) =1$ if $n>2$ (see for instance \cite[Ex.6.5]{hartshorne}), $Q$ is not toric in this case. When $n=2$, $Q$ is isomorphic to the toric variety $\mathbb{P}^1 \times \mathbb{P}^1$. 

\end{proof}

\begin{rem}If $E$ is a vector bundle of rank $k$ on a toric variety $Z$, then $\mathbb{P}_Z(E)$ is toric if and only if $E$ is \emph{decomposable}, i.e., $E$ is isomorphic to a direct sum of $k$ line bundles on $Z$ (see for instance \cite[1.1]{drso}). On the other hand, it is well known that the tangent bundle $T_{\mathbb{P}^r}$ is decomposable if and only if $r=1$.

Using this and Lemma \ref{Pic_one}, we can refine Proposition \ref{fano_index} in the toric case: if $X$ is an $n$-dimensional smooth toric variety with index $r \geq \frac{n+1}{2}$, then $X$ is isomorphic to one of the following varieties:

\begin{itemize}
\item $\mathbb{P}^n;$

\item $\mathbb{P}^{\frac{n}{2}} \times \mathbb{P}^{\frac{n}{2}}$  ($n$ even);

\item $\mathbb{P}^1 \times \mathbb{P}^1 \times \mathbb{P}^1$ ($n=3$);

\item $\mathbb{P}_{\mathbb{P}^r}(\mathcal{O}(2) \oplus \mathcal{O}(1)^{r-1})$ ($n = 2r-1$).
\end{itemize}

\label{toric_fano}
\end{rem}

\section{Nef and spectral values}

Adjunction Theory is an important topic of Algebraic Geometry that studies the interplay between an embedding of a projective variety into a projective space and its canonical bundle. The basic objects of study in Adjunction Theory are pairs $(X, L)$ where $X$ is a normal projective variety and $L$ is an ample divisor on $X$. Each such pair is called a \emph{polarized variety}.

\begin{de} Two polarized varieties $(X, L)$ and $(X',L')$ are said to be \emph{isomorphic} if there exists an isomorphism $f: X \to X'$ such that $f^*(L') = L$.
\end{de}

In what follows we define two invariants of a polarized variety and state certain properties of theirs.
    
\begin{de} Let $(X, L)$ be a polarized variety. Assume that $X$ has terminal singularities and $K_X$ is not pseudo-effective. For $t>>0$, $K_X + tL$ is an ample $\R$-divisor. The \emph{nef value} of $L$ and the \emph{(unnormalized) spectral value} of $L$, written respectively by $\tau(L)$ and $\mu(L)$ are the the real numbers:
$$
\tau(L) = \inf \{t \geq 0 \ / \ K_X + tL \in \overline{Nef}(X) \}
$$
$$
\mu(L) = \inf \{t \geq 0 \ / \ K_X + tL \in \overline{Eff}(X) \}.
$$

\label{nef_spectral}
\end{de}

\begin{figure}[h]
\centering
\includegraphics[scale=0.5]{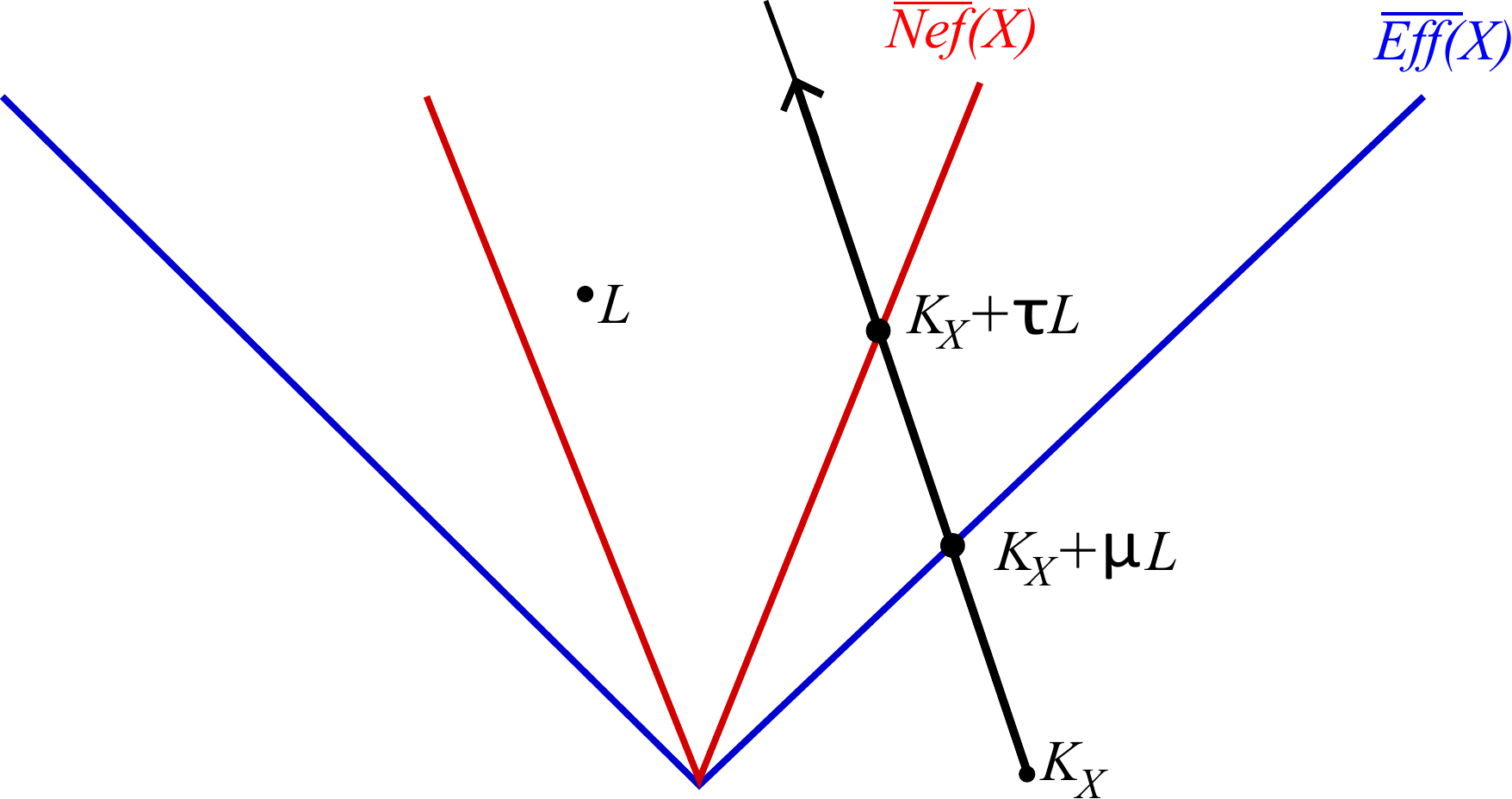}
\caption{The nef value and the spectral value of $L$.}
\end{figure}

Since $\overline{Nef}(X) \subset \overline{Eff}(X)$ and $K_X$ is not pseudo-effective, we see that $\mu(L) \leq \tau(L)$ and $\mu(L) > 0$. It follows from Kawamata Rationality Theorem (see for instance \cite[1.5.2]{beltrametti}) that $\tau(L)$ is a rational number. If in addition $X$ is smooth, the \emph{effective threshold} of $L$ defined by:
$$
\sigma(L) = \sup \{t \geq 0 \ / \ L + tK_X \in \overline{Eff}(X) \},
$$
is a rational number (see \cite[5.2]{araujo} ). Hence, $\mu(L) = 1/\sigma(L)$ is also rational. Note that by definition, the $\Q$-divisor $K_X + \tau(L) \cdot L$ is nef, but not ample on $X$.\\

\noindent \textbf{The Nef Value Morphism.} 
Let $(X, L)$ be a polarized variety. Assume that $X$ has terminal singularities and $K_X$ is not pseudo-effective. We point out the existence of a special morphism of $X$ associated to $\tau(L)$. For details see for instance \cite[4.2]{beltrametti}.

Writing $\tau(L) = a/b$ for positive integers $a$ and $b$, one can show that the complete linear system $|m(bK_X + aL)|$ is basepoint free for every integer $m>>0$ and hence defines a morphism $f:X \to \mathbb{P}^r$. By Remmert-Stein factorization, there exist a morphism $\phi_L:X \to Y$ (unique up to isomorphism of $Y$) with connected fibers onto a normal projective variety and a finite morphism $s:Y \to \mathbb{P}^r$ such that $f = s \circ \phi_L$.

\begin{de} The morphism $\phi_L: X \to Y$ defined above is called the \emph{nef value morphism} of the pair $(X, L)$.  
\end{de}

Since $bK_X + aL$ is nef, this divisor determines a face $F$ of $\overline{NE}(X)$, given by the classes $c \in \overline{NE}(X)$ such that $(bK_X + aL) \cdot c = 0$. Thus, for every such $c$, $bK_X \cdot c = -aL \cdot c < 0$ because $L$ is ample. It follows that $F$ is contained in he halfspace $K_X < 0$. By the Contraction Theorem, the contraction of $F$ exists and by construction, coincides with the nef value morphism $\phi_L$.

\begin{de} Let $X$ be a smooth variety and $L$ an ample divisor. We say that $L$ is $\Q$-normal, if $\mu(L) = \tau(L)$.
\end{de}

The $\Q$-normality assumption for an ample divisor implies an interesting consequence to the nef value morphism $\phi_L$:

\begin{prop}[{\cite[7.1.6]{beltrametti}}] Let $(X,L)$ be a polarized variety, with $X$ smooth. Then $L$ is $\Q$-normal if and only if $\phi_L$ is not birational. 
\label{nonbirational}
\end{prop}

The following conjecture was stated in 1994 by M. C. Beltrametti and A.J. Sommese in \cite[7.1.8]{beltrametti}.

\begin{conj} Let $(X, L)$ be a polarized variety with $X$ smooth. If $\dim(X) = n$ and $\mu(L) > \frac{n+1}{2}$ then $L$ is $\Q$-normal.
\label{conj_codegree}
\end{conj}

We finish this section with two structural results related to the nef value of a divisor. The first result is an immediate consequence of \cite[2.5]{bsm}:

\begin{teo}[{\cite[3.1.1]{bsm}}] Let $(X, L)$ be a polarized variety, where $X$ is smooth. Assume that $\phi_L: X \to Y$ is not birational. If $\tau(L)> \frac{n+1-\dim(Y)}{2}$, then $X$ is covered by lines i.e., there exists an algebraic family of irreducible curves $\{C_t\}_{t \in S}$ with $\displaystyle \bigcup_{t \in S} C_t = X$ and $L \cdot C_t = 1, \ \forall t$. Moreover, the nef value morphism $\phi_L$ contracts all these lines.
\label{key_result_lines}
\end{teo}

\begin{teo} Let $(X, L)$ be a smooth polarized variety whose nef value morphism $\phi_L$ contracts a proper face $F \subsetneq \overline{NE}(X)$. Suppose that $X$ is covered by a family of lines $\{C_t\}_{t \in S}$. Set $\nu:= -K_X \cdot C_t - 2$. If $\nu \geq \frac{n-3}{2}$ then there exists an extremal ray $R$ of $\overline{NE}(X)$ such that the associated contraction $\phi_R: X \to Z$ is a Mori fiber space and $\dim(Z) \leq n - \tau +1$. 
\label{key_result_contraction}
\end{teo}

%%%%%%%%%%%%%%%%%%%%%%%%%%%%%%%%%%%%%%
%%%%%%%%%%%%%%%%%%%%%%%%%%%%%%%%%%%%%%
%%%%%%%%%%%%%%%%%%%%%%%%%%%%%%%%%%%%%%
%
%
%            Capítulo 3
%
%
%%%%%%%%%%%%%%%%%%%%%%%%%%%%%%%%%%%%%%
%%%%%%%%%%%%%%%%%%%%%%%%%%%%%%%%%%%%%%
%%%%%%%%%%%%%%%%%%%%%%%%%%%%%%%%%%%%%%

\chapter{Invariants of Polytopes}
\label{inariants_of_polytopes}

In this chapter we define certain numbers associated to a full dimensional lattice polytope $P$ that are invariant by affine isomorphisms of $P$ (see Definition \ref{aff_isom}). Using the correspondence between lattice polytopes and polarized toric varieties, one can reinterpret these invariants as algebro-geometric invariants defined in Section \ref{nef_spectral}. This will allow us to prove a classification theorem for smooth lattice polytopes.\\

Throughout this chapter, a lattice polytope $P \subset \R^n$ will be a polytope with vertices in $\Z^n$.

\section{Degree and Codegree}
\label{deg_and_codeg}

\begin{de} Let $P$ be a full dimensional lattice polytope in $\R^n$. The \emph{degree} of $P$, denoted by $\deg(P)$, is the smallest nonnegative integer $d$ such that $kP$ contains no interior lattice points for $1 \leq k \leq n-d$. 

The \emph{codegree} of $P$, $\codeg(P)$, is defined by:
$$
\codeg(P) = n + 1 - \deg(P).
$$
\end{de}

It follows immediately from these definitions that:
$$
\codeg(P) = \displaystyle \min_{k \in \Z} \Big\{k \geq 0 \ \Big| \ \interior(kP) \cap \Z^n \neq \varnothing \Big\}.
$$

Let $\varphi_{\R}: \R^n \to \R^n$ be a linear transformation induced by a lattice isomorphism $\varphi: \Z^n \to \Z^n$. Since $\varphi_{\R}$ preserves the number of lattice points in a full dimensional lattice polytope $P$, we conclude that $\deg(P)$ and $\codeg(P)$ are invariants by affine isomorphisms of $P$.

The degree $d$ of $P$ is related to the \emph{Ehrhart series} of $P$ as follows.
For each positive integer $m$, let $f_P(m)$ denote the number of lattice points in $mP$, 
and consider the  Ehrhart series 
$$
F_P(t) \ := \ \sum_{m\geq 1}f_P(m)t^m.
$$
It turns out that $h_P^*(t) := (1-t)^{n+1} F_P(t)$ is a polynomial of degree $d$ in $t$. We have $\deg(P) = d$. Moreover, writing:
$$
h_P^*(t) = h_0^* + h_1^*t + ... + h_d^*t^d,
$$
one can show that the coefficients $h_i^*$ are non-negative, $h_0^* = 1$ and $h_0^* + ... + h_d^* = \Vol(P)$. Here, $\Vol(P)$ denotes the \emph{normalized volume of $P$}, i.e., $\Vol(P)$ is equal to $n!$ times the Euclidean volume of $P$
(see \cite{beck_robins} for more details on Ehrhart series and  $h^*$-polynomials).\\

A full dimensional simplex $P$ in $\R^n$ is \emph{unimodular} if $P$ is affinely isomorphic to the standard simplex:
$$
\Delta_n = \Conv(0, e_1, ..., e_n),
$$ 
where $\{e_1, ..., e_n\}$ is the standard basis of $\R^n$. By a simple computation, we see that $\Vol(\Delta_n) = 1$. Moreover, if $P$ is a lattice simplex with vertices $v_0, ..., v_n$, $\Vol(P)$ is equal to the determinant of the linear operator in $\R^n$ that associates each canonical vector $e_j$ to $v_j-v_0$. Hence $\Vol(P) \geq 1$, and equality holds if and only if the simplex $P$ is unimodular.

The following proposition characterizes full dimensional polytopes with degree zero:

\begin{prop}[{see \cite[1.4]{baty}}] Let $P$ be a full dimensional lattice polytope in $\R^n$. The following statements are equivalent:
\begin{enumerate}
\item[(a)]$\deg(P) = 0;$

\item[(b)]$\Vol(P) = 1;$

\item[(c)]$P$ is a unimodular simplex.
\end{enumerate}

\label{deg_zero}
\end{prop}

\begin{proof}
Suppose that $\deg(P) =0$. We have $\deg(h_P^*(t)) = 0$, then $\Vol(P) = h_0^* + ... +h_d^* = h_0^* = 1$. This proves $(a) \Rightarrow (b)$. Next, suppose that $\Vol(P) =1$. There is a lattice simplex $\Delta \subset P$. The normalized volume is an increasing function in the following sense: if $Q$ and $Q'$ are full dimensional lattice polytopes with $Q \subsetneq Q'$ then $\Vol(Q) < \Vol(Q')$. Since $\Vol(\Delta) \geq 1$, it follows from this observation that $P = \Delta$ and therefore $P$ is a unimodular simplex. This proves $(b) \Rightarrow (c)$. To prove $(c) \Rightarrow (a)$ it is sufficient to show that $\deg(\Delta_n) = 0$. For any $k \in \Z_{>0}$, we have $\interior(k\Delta_n) \cap \Z^n \neq \varnothing$ if and only if there exists a lattice solution for the inequality system:
$$
x_1+...+x_n < k, \ \ \ \ \ \ \ \ \ \ x_i > 0 \ \ \forall i=1, ..., n
$$
This occurs if and only if $k \geq n+1$. Therefore $\codeg(\Delta_n) = n+1$ and $\deg(\Delta_n) =0$.

\end{proof}

In some sense, Proposition \ref{deg_zero} shows that the unimodular simplices are the ``simplest'' full dimensional lattice polytopes. Polytopes with degree 1 were classified by Batyrev and Nill in \cite[2.5]{baty}. They belong to a special class of lattice polytopes, called \emph{Cayley Polytopes}. 

\begin{de} A \emph{Cayley polytope} is a lattice polytope $P$ affinely isomorphic to
$$
P_0 *...* P_k \ := \ \Conv \big(P_0 \times \{\bar{0}\}, P_1 \times \{e_1\},...,P_k \times \{e_k\}\big) \subset \mathbb{R}^{m}\times \mathbb{R}^{k},
$$
where the $P_i$'s are lattice polytopes in $ \mathbb{R}^m$ and $\{e_1,...,e_k\}$ is a basis for $\mathbb{Z}^k$. When the $P_i's$ are $m$-dimensional and have the same normal fan, we say that $P$ is \emph{strict} and denote it by:
$$
P = \Cayley(P_0, ..., P_k).
$$

\label{def_Cayley}
\end{de}

\begin{figure}[h]
\centering
\includegraphics[scale=0.9]{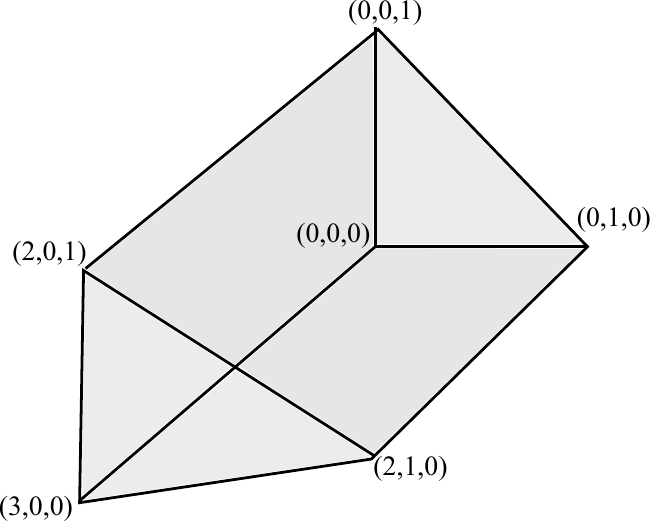}
\caption{The Cayley polytope $\Cayley(3\Delta_1, 2\Delta_1, 2\Delta_1)$.}
\end{figure}

\section{Nef value and $\Q$-codegree}
\label{nef_and_Q_degree}

In this section we will define two invariants of a full dimensional lattice polytope that are, in some sense, "more refined" than the notion of codegree. Next we relate these new invariants to invariants of the polarized toric variety associated to this polytope.\\

Let $P \subset \R^n$ be a full dimensional lattice polytope. 
For each $t\in \R_{\geq 0}$, we let $P^{(t)}$ be the (possibly empty) polytope obtained 
by moving each facet of $P$ toward its inner normal direction by a ``lattice distance'' of $t$ units.  
More precisely, if $\Sigma_P(1)=\{\eta_i\}_{i\in \{1, \dots, r\}}$, and $P$ is given by facet presentation:
$$
P\ =\ \Big\{
x\in \R^n \ \Big| \ \langle x, \eta_i\rangle \geq -a_i, \ 1\leq i\leq r\ 
\Big\},
$$
then $P^{(t)}$ is given by
$$
P^{(t)}\ =\ \Big\{
x\in \R^n \ \Big| \  \langle x, \eta_i\rangle \geq -a_i+t, \ 1\leq i\leq r\ 
\Big\}.
$$
These are called  \emph{adjoint polytopes} in \cite{drhnp}.
Set 
$$
\sigma(P) := sup\ \Big\{ t\geq 0 \ \Big| \  P^{(t)}\neq  \varnothing \Big\}.
$$
One can check that $\dim P^{(t)}=n$ for $0\leq t<\sigma(P)$, and 
$\dim P^{(\sigma(P))}<n$ (see \cite[1.6]{drhnp}). 
When we increase $t$ from $0$ to $\sigma(P)$,  $P^{(t)}$ will change its
combinatorial type at some critical values, the first one being 
$$
\lambda(P):= sup\ \Big\{
t\geq 0 \ \Big| \ P^{(t)} \neq \varnothing \ \mbox{and} \ \Sigma_t:=\Sigma_{P^{(t)}}= \Sigma_P \Big\}\leq \sigma(P).
$$
By  \cite[Lemma 1.13]{drhnp}, $\lambda(P)>0$ if and only if 
the normal fan of $P$ is $\Q$-Gorenstein. This happens for instance when $P$ is 
a smooth lattice polytope.

\begin{de}(\cite{alicia} \label{defn_qnormal}
Let $P \subset \mathbb{R}^n$ be an $n$-dimensional lattice polytope. 
 
The $\mathbb{Q}$-\emph{codegree} of $P$ is 
$$\codeg_{\mathbb{Q}}(P)\ :=\ \sigma(P)^{-1}.$$

Suppose that $\lambda(P)>0$. Then the  \emph{nef value} of $P$ is 
$$
\tau(P)\ :=\ \lambda(P)^{-1}.$$

In particular, $\codeg_{\Q}(P) \leq \tau(P)$. When equality holds, $P$ called is $\Q$\emph{-normal}.
\end{de}

\begin{lemma} Let $P$ be a full dimensional lattice polytope in $\R^n$ and $T: \R^n \to \R^n$ be an affine isomorphism of $\R^n$. Then, for every $t>0$:
$$
T\big(P^{(t)}\big) = \big(T(P)\big)^{(t)}.
$$
In particular, the nef value and the $\Q$-codegree of a lattice polytope are invariant by affine isomorphisms of $\R^n$. 

\label{aff_adjoint}
\end{lemma}

\begin{proof}
Suppose that the facet presentation of $P$ is  
$$
P\ =\ \Big\{
x\in \R^n \ \Big| \ \langle x, \eta_i\rangle \geq -a_i, \ 1\leq i\leq r\ 
\Big\}.
$$
Let $(T^{-1})^*$ be the adjoint operator of the inverse linear map $T^{-1}$ and set $\zeta_i = (T^{-1})^*(\eta_i)$. Since $T$ is an affine isomorphism and the $\eta_i's$ are primitive, the $\zeta_i$ are primitive again. For every $x \in \R^n$ we have:
$$
\langle T(x), \zeta_i \rangle = \big \langle T(x), (T^{-1})^*(\eta_i) \big \rangle = \langle x, \eta_i \rangle . 
$$
Therefore, the facet presentation of $T(P)$ is:
$$   
T(P)\ =\ \Big\{
y\in \R^n \ \Big| \ \langle y, \zeta_i\rangle \geq -a_i, \ 1\leq i\leq r\ 
\Big\}.
$$
It follows immediately from this computation that $T\big(P^{(t)}\big) = \big(T(P)\big)^{(t)}$ for every $t>0$.
\end{proof}

The following proposition reveals the geometric meaning of these polytope invariants.

\begin{prop} Let $P$ be a smooth $n$-dimensional lattice polytope in $\R^n$. Let $(X,L)$ be the polarized toric variety associated to $P$. Then the nef value and the $\Q$-codegree of $P$ are equal, respectively, to the nef value and the spectral value of $L$, defined in Section \ref{nef_spectral}. In symbols:
$$
\tau(P) = \tau(L) \ \ \mbox{and} \ \ \codeg_{\Q}(P) = \mu(L).
$$

\label{inv_polytopes_varieties}
\end{prop}

\begin{proof}
The polytope $P$ is the polytope $P_L$ associated to $L$. Write $L= \displaystyle \sum_{\eta_i \in \Sigma_P(1)} a_iV(\eta_i)$. For all $t>0$ the divisor $L_t = L + tK_X = \sum_i (a_i - t)V(\eta_i)$ corresponds to the adjoint polytope:
$$
{P_L}^{(t)} = P^{(t)} = \Big\{
x\in \R^n \ \Big| \  \langle x, \eta_i\rangle \geq -a_i+t, \ 1\leq i\leq r\ 
\Big\}.
$$
By Proposition \ref{polytope_div_conections}, $L_t$ is ample if and only if $\Sigma_{P^{(t)}} = \Sigma_P$. Then:
\begin{align*}
\tau(L)  & = \inf\{t>0 | K_X + tL \ \mbox{is ample} \ \} \\
      & = (\sup\{t>0 | L + tK_X \ \mbox{is ample}\ \})^{-1}  \notag \\
      & = (\sup\{t>0 | \Sigma_{P^{(t)}} = \Sigma_P\})^{-1} \notag \\
      & = \lambda(P)^{-1} = \tau(P). \notag   
\end{align*}
For the $\Q$-codegree we have:
\begin{align*}
\mu(L) & = \inf\{t>0 | K_X + tL \ \mbox{ is pseudo-effective} \} \notag \\ 
       & = (\sup\{t>0 | L + tK_X \ \mbox{is pseudo-effective}\})^{-1} \notag \\
       & = (\sup\{t>0 | P^{(t)} \neq \varnothing \})^{-1} \notag \\
       & =  \sigma(P)^{-1} = \codeg_{\Q}(P).
\end{align*}
\end{proof}

Next, we study the relation between the codegree and the $\Q$-codegree of a full dimensional lattice polytope $P \subset \R^n$. First, we give another characterization of $\codeg(P)$ in terms of adjoint polytopes. The lattice points contained in the interior of $kP$ are given by the intersection:
$$
(kP)^{(1)} \cap \Z^n.
$$
In fact, let $\Sigma$ the normal fan of $P$ and $\Sigma(1)=\{\eta_i\}_{i=1,...,r}$. If
\begin{equation}
P\ =\ \Big\{
u\in \R^n \ \Big| \ \langle u, \eta_i\rangle \geq -a_i, \ 1\leq i\leq r\ 
\Big\}
\label{fac_pres}
\end{equation}
is the facet presentation of $P$, and $u \in \R^n$ is a lattice point, we have:
\begin{align*}
u \in int(kP) & \iff \langle u, \eta_i\rangle > -ka_i, \forall i=1,...,r \notag \\
             & \stackrel{(*)}{\iff} \langle u, \eta_i\rangle \geq -ka_i +1 \notag \\
             & \iff u \in (kP)^{(1)},
\end{align*}
where the equivalence in $(*)$ holds since $u$ is a lattice point. We conclude that:
\begin{equation}
\codeg(P) = \displaystyle \min_{k \in \N} \Big\{ k \ \Big| \ (kP)^{(1)} \cap \Z^n \neq \varnothing \Big\}.
\label{codeg_reformulado}
\end{equation}

\begin{prop} Let $P$ be a smooth full dimensional lattice polytope in $\R^n$. Then $\codeg(P) \geq \lceil \codeg_{\Q}(P) \rceil$. Equality holds if $P$ is $\Q$-normal and hence, in this case, $\tau(P) > \codeg(P) - 1$ {\rm(see \cite[Lemma 2.4]{alicia})}.
\label{teto}
\end{prop}

\begin{proof} Let $k:= \codeg(P)$. By (\ref{codeg_reformulado}), $(kP)^{(1)} \neq \varnothing$. Since $(kP)^{(1)} = k \cdot (P^{(\frac{1}{k})})$, we have that $P^{(\frac{1}{k})} \neq \varnothing$. By definition of $\Q$-codegree, we conclude that $k \geq \codeg_{\Q}(P)$. Since $k$ is an integer, $k \geq \lceil \codeg_{\Q}(P) \rceil$.

Next, suppose that $P$ is $\Q$-normal. Let $(X,L)$ be the polarized toric variety corresponding to $P$. If $m>0$ is an integer such that $(mP)^{(1)} \neq \varnothing$ then $(mP)^{(1)} \cap \Z^n \neq \varnothing$. In fact, the divisor on $X$ corresponding to $(mP)^{(1)}$ is $K_X + mL$. Since $X$ is smooth, $K_X+mL$ is Cartier. By Proposition \ref{inv_polytopes_varieties}, $m \geq \codeg_{\Q}(L) = \tau(L)$. Hence, $K_X+mL$ is nef. Since nef divisors on toric varieties are basepoint free (see \ref{mov_toric}), we conclude that $K_X + mL$ is basepoint free. It follows from Proposition \ref{polytope_div_conections} that $(mP)^{(1)}$ is a lattice polytope and hence $(mP)^{(1)} \cap \Z^n \neq \varnothing$.

We proved that for smooth $\Q$-normal lattice polytopes we have:
$$
\codeg(P) = \displaystyle \min_{k \in \N} \Big\{ k \ \Big| \ (kP)^{(1)} \neq \varnothing \Big\}. 
$$
Therefore:
\begin{align*}
\codeg(P) & = \displaystyle \min_{k \in \N} \Big\{ k \ \Big| \ (kP)^{(1)} \neq \varnothing \Big\} \notag \\
          & = \displaystyle \min_{k \in \N} \Big\{ k \ \Big| \ P^{(\frac{1}{k})} \neq \varnothing \Big\} \notag \\
          & = \lceil \codeg_{\Q}(P) \rceil. \notag 
\end{align*}

\end{proof}

\begin{ex} Let $P$ be a full dimensional lattice polytope in $\R^n$ and $k \in \N$ such that $(kP)^{(1)}$ is non-empty. If $P$ is not smooth, $(kP)^{(1)}$ can fail to be a lattice polytope even when $k = \codeg(P)$. In fact, let $P = \Conv \Big((0,0), (1,1), (0,3) \Big)$. Then $P$ is a non-smooth triangle in $\R^2$ since the vectors $(-2,-1)$ and $(1, -1)$ defining the facets incidents to $(1,1)$ do not form a basis for $\Z^2$. We have $\codeg(P) = 2$ but:
$$
(2P)^{(1)} = \Conv \Big((1,2), (1,3), (4/3, 7/3) \Big)
$$
is not a lattice polytope (see Figure \ref{non_lattice_pol}).  
\end{ex}

\begin{figure}[h]
\centering
\includegraphics[scale=0.6]{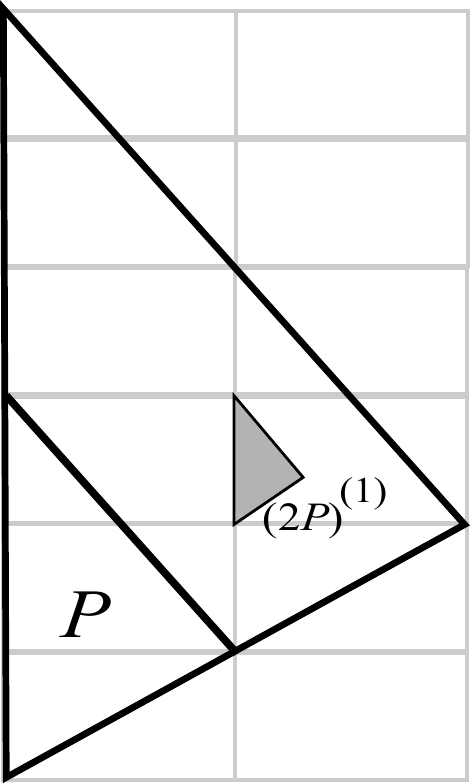}
\caption{The polytope $(2P)^{(1)}$ is not a lattice polytope.}
\label{non_lattice_pol}
\end{figure}
 
\begin{rem}
Let $P$ be a lattice polytope, and
$(X,L)$ the corresponding polarized toric variety.
When $X$ is $\Q$-Gorenstein (i.e., some multiple of $K_X$ is Cartier),
the family of adjoint polytopes $\{P^{(t)}\}_{0\leq t\leq \sigma(P)}$ is
the polytope counterpart of the \emph{Minimal Model Program with scaling}, 
established in \cite{bchm}.
The projective varieties $X_t=X_{\Sigma_t}$ that appear when
we increase $t$ from $0$ to $\sigma(P)$ are precisely the varieties that appear in the 
Minimal Model Program for $X$ with scaling of $L$.
A precise statement and proof can be found in \cite{tese_edilaine}.
\end{rem}

\section{Toric $\mathbb{P}^k$-bundles}
\label{toric_proj_bundles}

In this section we describe the fan of toric projective bundles. Details and proofs can be found in \cite[Sec. 3]{dirocco} or \cite[7.3]{cox}.\\

\begin{say} Let $\Sigma'$ be a complete fan in $\R^m$ and $Y:=X_{\Sigma'}$ the corresponding toric variety. Suppose that $Y$ is projective. Let $D_0, ..., D_k$ be $T$-invariant Cartier divisors on $Y$. The $\mathbb{P}^k$-bundle:
$$
X = \mathbb{P}_Y(\mathcal{O}(D_0) \oplus ... \oplus \mathcal{O}(D_k))
$$
is a toric variety and its fan $\Sigma$ is constructed as follows: write $\Sigma'(1)=\{\eta_1, ...,\eta_r\}$ and let $L_i = V(\eta_i)$ be the $T$-invariant divisor corresponding to $\eta_i$. Suppose that each divisor $D_j$ is written in the form: 
$$
D_j = \displaystyle \sum_i a_{ij}L_i.
$$
Let $\{e_1, ..., e_k\}$ be the canonical basis of $\R^k$ and $e_0=-e_1-...-e_k$. By simplicity, we also denote by $e_j$ the vector $(\bar{0},e_j) \in \R^m \times \R^k$. Similarly, we use the same symbol $\eta_i$ to denote the vector $(\eta_i, \bar{0}) \in \R^m \times \R^k$. For each $\eta_i$ set:
$$
\tilde{\eta}_i = \eta_i +  \displaystyle \sum_{j=0}^k (a_{ij} - a_{i0})e_j.
$$
For each cone $\sigma = \Cone(\eta_{i_1}, ..., \eta_{i_r}) \in \Sigma'(n)$, set $\tilde{\sigma} = \Cone(\tilde{\eta}_{i_1}, ..., \tilde{\eta}_{i_r})$. The maximal cones of $\Sigma$ are of the form:
$$
\tilde{\sigma} + \Cone(e_0, ..., \hat{e}_j, ..., e_k) \subset \R^m \times \R^k,
$$
for $\sigma \in \Sigma'(m)$ and $j \in \{0, ..., k\}$.\\

The natural projection $\pi: X \to Y$ is a toric morphism induced by the  projection $\bar{\pi}: \Z^m \times \Z^k \to \Z^m$. For every $\eta_i \in \Sigma'(1)$, the pullback $\pi^*(V(\eta_i))$ of the divisor $V(\eta_i)$ on $Y$ is the divisor $V(\tilde{\eta}_i)$. 

\label{notation_proj_bundle}
\end{say}

The following lemma describes the class group of $X$ and its tautological line bundle in terms of $T$-invariant divisors: 

\begin{lemma}[{\cite[Sec.3, Lemma 3]{dirocco}}] Let
$$
X=X_{\Sigma} = \mathbb{P}_Y(\mathcal{O}(D_0) \oplus ... \oplus \mathcal{O}(D_k)) \stackrel{\pi}{\longrightarrow} Y=X_{\Sigma'}
$$
be a toric projective bundle. 
\begin{enumerate}

\item[(a)] For each $j=1,...,k$ we have:
$$
V(e_j) \sim V(e_0) + \pi^*(D_0) - \pi^*(D_j).
$$ 
In particular, the classes of the divisors $V(e_0)$ and $V(\tilde{\eta}_i), \  \tilde{\eta}_i \in \Sigma(1)$, generate the class group $Cl(X)$.

\item[(b)] The divisor:  
 $$
\xi := V(e_0) + \displaystyle \sum_i a_{i0}V(\tilde{\eta}_i) = V(e_0) + \pi^*(D_0)
$$  
corresponds to the tautological line bundle on $X$, i.e., $\mathcal{O}_X(\xi) \simeq \mathcal{O}_X(1)$.
\end{enumerate}

\label{div_proj_bundle}
\end{lemma}

\begin{rem} In the discussion above, suppose in addition that each $D_j$ is an ample divisor on $Y$. By \cite[III, Thm. 1.1]{hartshorne2}, the tautological line bundle $\mathcal{O}_X(1)$ on $X$ is also ample. Therefore, in this case, the divisor $\xi$ will be ample on $X$.
\label{tautological_ample}
\end{rem}

\begin{rem} The following result will be useful later. Let $X = \mathbb{P}_Y(\mathcal{E})$ be a projective bundle of rank $k$ over a projective variety $Y$ and $\pi: X \to Y$ the $\mathbb{P}^k$-bundle map. If $K_Y$ is Cartier then the equality below holds:
$$
K_X = \pi^*(K_Y+ E) -(k+1)\xi,
$$
where the divisors $E$ and $\xi$ satisfy $\mathcal{O}_Y(E) \simeq \det(\mathcal{E})$ and $\mathcal{O}_X(\xi) \simeq \mathcal{O}_X(1)$. 
\label{equality_canonical}
\end{rem}

\begin{say} Next we will see that toric varieties corresponding to the strict Cayley Polytopes defined in Section \ref{deg_and_codeg} are toric $\mathbb{P}^k$-bundles.

Let $P_0, \dots, P_k \subset \mathbb{R}^m$ be $m$-dimensional lattice polytopes
with the same normal fan $\Sigma'$.
Let $Y=X_{\Sigma'}$ be the corresponding projective $m$-dimensional toric variety, 
and $D_j$ the ample $T$-invariant divisor on $Y$ associated to $P_j$.
More precisely, write $\Sigma'(1)=\big\{\eta_i\big\}_{i\in \{1, \dots, r\}}$, and let $P_j$ 
be given by facet presentation:
$$
P_j\ =\ \Big\{
x\in \R^m \ \Big| \ \langle \eta_i,x\rangle \geq -a_{ij}, \ 1\leq i\leq r\ 
\Big\}.
$$
Let $L_i=V(\eta_i)$ be the invariant Weil divisor on $Y$ associated to $\eta_i$.
Then $D_j = \displaystyle \sum_{i=1}^r a_{ij}L_i$.

\label{notation_P_i}
\end{say}
The next result was proved in \cite[Sec. 3]{ccd} using a version of Farkas Lemma (see \cite[Ex. 1.6]{z} ). We give an alternative proof using the classical geometric property stated in Remark \ref{tautological_ample}.
 
\begin{lemma} Let the notation be as in \ref{notation_proj_bundle} and \ref{notation_P_i}. Then the polarized toric variety $(X_{\Sigma}, L)$ corresponding to $\Cayley(P_0, ..., P_k)$ is isomorphic to
$$ \Big(\mathbb{P}_Y\big(\mathcal{O}(D_0)\oplus \cdots \oplus \mathcal{O}(D_k) \big), 
\ \xi \Big),
$$
where $\xi$ is a divisor corresponding to the tautological line bundle on $\mathbb{P}_Y\big(\mathcal{O}(D_0)\oplus \cdots \oplus \mathcal{O}(D_k) \big)$.

\label{cayley_1}
\end{lemma}

\begin{proof} Let $\xi = V(e_0) + \pi^*(D_0)$ be as in Lemma \ref{div_proj_bundle}. The divisor $\xi$ corresponds to the tautological line bundle on $X_{\Sigma}$. Since every $D_j$ is ample, by Remark \ref{tautological_ample}, $\xi$ is ample. From Proposition \ref{polytope_div_conections}, $\xi$ defines a full dimensional lattice polytope $P_{\xi}$ in $\R^m \times \R^k$ whose corresponding polarized toric variety is $(X_{\Sigma}, L)$. 
Thus, to conclude the proof we will show that:
$$
P_{\xi} = \Cayley(P_0, ..., P_k).
$$ 
The facet presentation of $P_{\xi}$ is:
\begin{equation}
\langle \tilde{u}, \tilde{\eta}_i \rangle \geq -a_{i0}, \ \ \ \langle \tilde{u}, e_0 \rangle \geq -1, \ \ \ \langle \tilde{u}, e_j \rangle \geq 0, \ j=1, ..., k.
\label{facets_cayley}
\end{equation}
If $\tilde{u}= u+u'$, $u \in \R^m \times \{\bar{0}\}, \ u' \in \{\bar{0}\} \times \R^k$, is a lattice point of $P_{\xi}$, the two last inequalities above implies that $u'=0$ or $u'=e_j$ for some $j=1, ..., k$. Replacing the value of $u'$ in the first inequality, we obtain:
$$
\left. \begin{array}{ll}
u' = 0 \ \Longrightarrow \langle u, \eta_i \rangle \geq -a_{i0} \Longrightarrow u \in P_0\\
u'= e_j \Longrightarrow \langle u, \eta_i \rangle \geq -a_{ij} \Longrightarrow u \in P_j.\\
\end{array} \right.
$$
It follows that $P_{\xi} \subset \Cayley(P_0, ..., P_k)$. Conversely, any point of $\Cayley(P_0, ..., P_k)$ is of form:
$$
\tilde{u} = \lambda_0(u_0, \bar{0}) + \lambda_1(u_1, e_1) + ...+ \lambda_k(u_k, e_k),
$$
where for all $j= 0, ..., k$, $u_j \in P_j$, $\lambda_j$ is a non-negative integer and $\sum \lambda_j = 1$. We have:
\begin{itemize}
\item $\langle \tilde{u}, e_0 \rangle = \displaystyle \sum_{j=1}^k -\lambda_j \geq -1,$

\item $\langle \tilde{u}, e_j \rangle = \lambda_j \geq 0, \ \forall j=1, ..., k,$

\item $\langle \tilde{u}, \tilde{\eta}_j \rangle = \displaystyle \sum_{j=0}^k \lambda_j(\langle u_j, \eta_i \rangle + a_{ij}-a_{i0}) \geq \displaystyle \sum_{j=0}^k \lambda_j(-a_{i0}) = -a_{i0},$
\end{itemize}
where the last inequality holds because $u_j \in P_j$ implies that $\langle u_j, \eta_i \rangle \geq -a_{ij}$. Thus the points of $\Cayley(P_0, ..., P_k)$ satisfy inequalities \ref{facets_cayley}. Therefore:
$$
\Cayley(P_0, ..., P_k) \subset P_{\xi}.
$$
\end{proof}

\begin{ex} In this example we describe the nef cone and the effective cone of the toric projective bundle: 
$$
X = \mathbb{P}_{\mathbb{P}^m}(\mathcal{O}(a_0) \oplus ... \oplus \mathcal{O}(a_k)) \stackrel{\pi}{\longrightarrow} \mathbb{P}^m,
$$
where $0< a_0 \leq ... \leq a_k$ are positive integers. We have $\rho(X) = 2$.
Let $\{f_1, ..., f_m\}$ and $\{e_1, ..., e_k\}$ be respectively the canonical basis of $\R^m$ and $\R^k$, $f_0: -\sum f_i$ and $e_0 = - \sum e_j$. Recall that the maximal cones of the fan of $\mathbb{P}^m$ are of the form:
$$
\sigma_i = \Cone(f_0, ..., \hat{f}_i, ..., f_m)\ \ \ \ i=0, ..., m.
$$
The divisor $H = V(f_0)$ represents the hyperplane class on $\mathbb{P}^m$. For each $j=0, ..., k$, $\mathcal{O}(a_j) \simeq \mathcal{O}(a_jH)$. Using the notation of Section \ref{toric_proj_bundles}, if $\Sigma$ is the fan of $X$, its maximal cones are of the form:
$$
\sigma_{ij} := \Cone(\tilde{f}_0, ..., \widehat{\tilde{f}}_i, ..., \tilde{f}_m, e_0, ..., \hat{e}_j, ..., e_k), \ \ \ i=0, ..., m, \ \ j=0, ..., k,
$$
where $\tilde{f_i} = f_i$ , for $i = 1, ..., m$ and $\tilde{f_0} = f_0 + \sum_j(a_j-a_0)e_j$. Using Lemma \ref{div_proj_bundle}, we conclude that:
$$
\overline{Eff}(X) = \Cone([V(e_k)], [\pi^*(H)]).
$$ 
Next, we compute the nef cone. Since $H$ is ample on $\mathbb{P}^m$, by the projection formula $\pi^*(H)$ is nef but not ample. We claim that $V(e_0)$ is also nef. In fact, by the description of intersection products on toric varieties given in Proposition \ref{toric_intersections} and by \ref{mov_toric}, it is sufficient to see that $V(e_0) \cdot V(\omega) \geq 0$, for every wall $\omega$ of $\sigma_{i0}$ containing $e_0$. Note that $V(\tilde{f}_0) \sim V(\tilde{f}_i) \ \forall i=1,...,m$ because $V(f_0)$ and $V(f_i)$ are hyperplanes in $\mathbb{P}^m$ and $V(\tilde{f_i}) = \pi^*(V(f_i))$, $i=0,...,m$. If
$$
\omega = \Cone(\tilde{f}_0, ..., \widehat{\tilde{f}}_i, ..., \tilde{f}_m, e_0, ..., \hat{e}_l, ..., \hat{e}_j, ..., e_k), \ \ l,k>1,
$$
since $V(e_0) \sim (a_l - a_0)V(\tilde{f_0}) + V(e_l)$ (see Lemma \ref{div_proj_bundle}), we have:
$$
V(e_0) \cdot V(\omega) = [(a_l - a_0)V(\tilde{f_i}) + V(e_l)] \cdot V(\omega) = V(e_l) \cdot V(\omega) >0
$$
because $\Cone(\tilde{f}_i, \omega) \notin \Sigma$ and $\Cone(e_l, \omega) \in \Sigma$. 

Now, if:
$$
\omega = \Cone(\tilde{f}_0, ..., \widehat{\tilde{f}}_s, ...\widehat{\tilde{f}}_i, ..., \tilde{f}_m, e_0, ..., \hat{e}_j, ..., e_k), \ \ s, l = 0, ...,m, \ j>0,
$$
then:
$$
V(e_0) \cdot V(\omega) = [(a_j-a_0)V(\tilde{f}_s) + V(e_j)] \cdot V(\omega) =  (a_j-a_0)V(\tilde{f}_s)\cdot V(\omega) \geq 0.
$$
This proves that $V(e_0)$ is nef. Since:
$$
V(e_0) \cdot V \Big(\Cone((\tilde{f}_1, ..., \tilde{f}_m, e_1, ..., e_k) \Big) = 0,
$$
we conclude that $V(e_0)$ is not ample. The classes $V(e_0)$ and $\pi^*(H)$ are not multiple of each other because $\xi = V(e_0) + a_0 \pi^*(H)$ is ample. Therefore:
$$
\overline{Nef}(X) = \Cone([V(e_0)], [\pi^*(H)]).
$$
\label{proj_bundle_ex}
\end{ex}

\begin{prop} Let $X= X_{\Sigma}$ be an $n$-dimensional smooth projective toric variety and $R$ an extremal ray of $\overline{NE}(X)$. Suppose that the contraction $\phi_R: X \to Y$ associated to $R$ is a Mori fiber space and set $m = \dim(Y)$. Then there exists a decomposable vector bundle $\mathcal{E}$ on $Y$ with rank $n-m+1$ such that $X \simeq \mathbb{P}_Y(\mathcal{E})$. 

\label{Mori_fiber_pkbundle}
\end{prop}

\begin{proof} Recall from \ref{toric_contraction} that each maximal cone $\sigma$ of the fan $\Sigma'$ of $Y$ is constructed as follows. There exist a wall $\omega \in \Sigma(n-1)$ separating two maximal cones:
$$
\tilde{\sigma} = \Cone(\tilde{\eta}_1, ..., \tilde{\eta}_m, u_0, u_1, ..., u_{k-1}),
$$
$$\tilde{\sigma'} = \Cone(\tilde{\eta}_1, ..., \tilde{\eta}_m, u_1, ..., u_{k-1}, u_k).
$$
of $\Sigma$ and positive integers $a_0,...,a_k$ such that $a_0u_0 + ... + a_ku_k = 0$. The cone:
$$
U = \Cone(u_0, ..., u_k)
$$
is a vector space of dimension $k$. Write $\R^n = U \oplus U^{\perp}$. Setting $\sigma(\omega) = \tilde{\sigma} + \tilde{\sigma'}$, we have that $\sigma$ is the image of $\sigma(\omega)$ by the projection $\pi: \R^n  \to U^{\perp}$. If $\eta_i = \pi(\tilde{\eta}_i)$, we have $\sigma = \Cone(\eta_i)_{i=1,...,m}$. There exist integers $b_{ij}$ such that:
$$
\tilde{\eta_i} = \eta_i + \sum_{j=1}^k b_{ij}u_j.
$$
The cones of $\Sigma(n)$ whose images by $\pi$ are equal to $\sigma$ are of the form:
$$
\sigma_j = \Cone(\tilde{\eta_i})_{i=1,...,m} + \Cone(u_0, ..., \hat{u}_j, ..., u_k).
$$
Since $X$ is smooth, we have $a_j = 1, \ \forall j$. By making a linear isomorphism in $\Z^n = \Z^m \times \Z^k$, we can suppose that $u_j$ is the canonical vector $e_j$ in $\Z^k$, $\forall j>0$ and $u_0 = e_0 = - \sum_{j=1}^k e_j$. 
Comparing with \ref{notation_proj_bundle}, we see that 
$$ 
X \simeq \mathbb{P}_Y(\mathcal{O} \oplus \mathcal{O}(D_1) \oplus ... \oplus \mathcal{O}(D_k)),
$$
where $D_j = \displaystyle \sum_{\eta_i \in \Sigma'(1)} b_{ij}V(\eta_i)$.

\end{proof}

\vspace{30 cm}
\ 

\vspace{30 cm}

%%%%%%%%%%%%%%%%%%%%%%%%%%%%%%%%%%%%%%%%%%%%
%%%%%%%%%%%%%%%%%%%%%%%%%%%%%%%%%%%%%%%%%%%%
%%%%%%%%%%%%%%%%%%%%%%%%%%%%%%%%%%%%%%%%%%%%
%
%
%           CHAPTER 4
%
%
%%%%%%%%%%%%%%%%%%%%%%%%%%%%%%%%%%%%%%%%%%%%
%%%%%%%%%%%%%%%%%%%%%%%%%%%%%%%%%%%%%%%%%%%%
%%%%%%%%%%%%%%%%%%%%%%%%%%%%%%%%%%%%%%%%%%%%

\chapter{Classifying Smooth Lattice Polytopes}
\label{main_chapter}

In this chapter we classify smooth lattice polytopes with small degree. This class of polytopes is very special. For instance, we saw in Section \ref{deg_and_codeg} that lattice polytopes with degree $\leq 1$ are Cayley Polytopes. In \cite{baty}, Batyrev and Nill posed the following question:
\vspace{0,3 cm}

\noindent \textbf{Question:} Is there a function $N(d)$ of $d$ such that every lattice polytope of degree $d$ and dimension $n>N(d)$ is a Cayley polytope?

\vspace{0,5 cm}

In \cite[Theorem 1.2]{hnp}, Hasse, Nill and Payne solved this problem with the quadratic polynomial 
$N(d) = (d^2+19d-4)/2$. 
It was conjectured in \cite[Conjecture 1.2]{alicia2} that one can take $N(d)=2d$. 
This would be a sharp bound.
Indeed, let $ \Delta_n$ denote the standard  $n$-dimensional unimodular  simplex.
If $n$ is even, then %the polytope
$2 \Delta_n$ has degree $d=\frac{n}{2}$, but is not a Cayley polytope.\\

While the methods of \cite{baty} and \cite{hnp} are purely combinatorial,  
Hasse, Nill and Payne pointed out that these results can be interpreted 
in terms of Adjunction Theory on toric varieties.
This point of view was then explored by Dickenstein, Di Rocco and Piene in \cite{alicia} 
to study smooth lattice polytopes with small degree. \\

One has the following classification of smooth $n$-dimensional  
lattice polytopes $P$ with degree $d<\frac{n}{2}$ (or, equivalently, $\codeg(P) \geq \frac{n+3}{2}$).

\begin{teo}[{\cite[Theorem 1.12]{alicia}  and \cite[Theorem 1.6]{alicia2}}] Let $P \subset \mathbb{R}^n$ be a smooth $n$-dimensional lattice polytope. Then $\codeg(P) \geq \frac{n+3}{2}$ if and only if $P$ is 
affinely isomorphic to a Cayley polytope $P_0 *...* P_k$, where all the $P_i$'s 
have the same normal fan, and $k = \codeg(P)-1 > \frac{n}{2}$. 
\label{DDRP}
\end{teo}

Theorem~\ref{DDRP} was first proved in \cite{alicia} under the additional assumption that 
$P$ is a \emph{$\mathbb{Q}$-normal} polytope (see Definition~\ref{defn_qnormal}). 
Then, using combinatorial methods, Dickenstein and Nill showed in  \cite{alicia2}
that  the inequality $\codeg(P) \geq \frac{n+3}{2}$  implies that $P$ is $\mathbb{Q}$-normal. 

In our work we go one step further, addressing smooth $n$-dimensional lattice polytopes $P$ of degree 
$d< \frac{n}{2}+1$ (or, equivalently, $\codeg(P) \geq \frac{n+1}{2}$). Not all such polytopes are Cayley polytopes as we observed for $2\Delta_n$. We need to generalize the concept of Cayley polytope in order to state our classification result.

\section{Generalized Cayley Polytopes} The following definition was introduced in \cite{alicia}.

\begin{de} \label{cay}
Let $P_0, \dots, P_k \subset \mathbb{R}^m$  be $m$-dimensional lattice polytopes, 
and $s$ a positive integer.   Set
$$
[P_0 *...* P_k]^s \ := \ \Conv \big(P_0 \times \{\bar{0}\}, P_1 \times \{se_1\},...,P_k \times \{se_k\}\big) \subset \mathbb{R}^{m}\times \mathbb{R}^{k},
$$
where $\{e_1,...,e_k\}$ is a basis for $\mathbb{Z}^k$.
We say that a lattice polytope $P$ is an $s^{th}$  \emph{order generalized Cayley polytope} 
if $P$ is affinely isomorphic to a polytope $[P_0 *...* P_k]^s$ as above.
If all the $P_i$'s  have the same normal fan, we write 
$$
P = \Cayley^s(P_0, \ldots ,P_k),
$$
and say that $P$ is \emph{strict}.
\end{de}

Generalized Cayley polytopes with $s=1$ are standard Cayley polytopes as defined in Chapter \ref{inariants_of_polytopes}. 

\begin{ex} The polytope $2 \Delta_n$ defined in the beginning of this Chapter is the $2^{\mbox{nd}}$ order generalized strict Cayley polytope $\Cayley^2(\underbrace{\Delta_0, ..., \Delta_0}_{n+1 \mbox{\ times}})$, where $\Delta_0 = \R^0$ is a point. 
\end{ex}

Next we describe the geometry of polarized toric varieties associated to 
generalized Cayley polytopes. We start by fixing the notation to be used throughout this section.

\begin{notation} \label{notation:P_i}
Let $P_0, \dots, P_k \subset \mathbb{R}^m$  be full dimensional lattice polytopes
having the same normal fan $\Sigma'$.
Let $Y=X_{\Sigma'}$ be the corresponding projective toric variety, 
and $D_j$ the ample $T$-invariant divisor on $Y$ associated to $P_j$. Write $\Sigma'(1)=\big\{\eta_i\big\}_{i\in \{1, \dots, r\}}$, and suppose that each $P_j$ is given by the facet presentation:
$$
P_j\ =\ \Big\{
x\in \R^m \ \Big| \ \langle x, \eta_i\rangle \geq -a_{ij}, \ 1\leq i\leq r\ 
\Big\}.
$$

Then $D_j = \displaystyle \sum_{i=1}^r a_{ij}L_i$, where $L_i = V(\eta_i)$. We denote by $T$-$\Div(Y)$ the group of $T$-invariant Weil divisors on $Y$.
\end{notation}

Let $X$ be the projective toric variety corresponding to $\Cayley^s(P_0, ..., P_k)$.
In \cite{alicia}, it was shown that there exists a toric
fibration $\pi:X\to Y$ whose set theoretical fibers are all isomorphic to $\mathbb{P}^k$.
We note that $\pi$ may have multiple fibers, and  $X$ may be singular, even when $Y$ is smooth.
The following lemma gives a necessary and sufficient condition for $\Cayley^s(P_0, \ldots ,P_k)$
to be smooth.

\begin{lemma} Let the notation be as in \ref{notation:P_i}.
Then $\Cayley^s(P_0, \ldots ,P_k)$ is smooth if and only if $Y$ is smooth and
$s$ divides $a_{ij}-a_{i0}$ for every $i\in\{1,\cdots,r\}$ and $j\in\{1,\cdots,k\}$. 
In this case, $s$ divides $D_j - D_0$ in $T$-$\Div(Y)$
for every $j\in\{1,\cdots,k\}$, and the corresponding polarized toric variety $(X,L)$  satisfies:
\begin{align}
 X \ &\simeq \ \mathbb{P}_Y\Big(\mathcal{O}(D_0)\oplus \mathcal{O}\Big(\frac{D_1-D_0}{s}+D_0\Big) 
\oplus \cdots \oplus \mathcal{O}\Big(\frac{D_k-D_0}{s}+D_0\Big) \Big) \ , \notag
\\
 L \ &\sim \ s\xi + \pi^*\big((1-s)D_0\big) \ , \notag
\end{align}
where $\pi: X \to Y$ is the $\mathbb{P}^k$-bundle map, 
and $\xi$ is a divisor corresponding to the tautological line bundle.
\label{cayley_s}
\end{lemma}

\begin{proof} 
Set $P^1:=P_0 *...* P_k$, and $P^s := \Cayley^s(P_0, \ldots ,P_k)$. 
Notice that a point $x =(y,z)\in\mathbb{R}^m\times\mathbb{R}^k$ is in $P^s$ if and only if 
$(y,\frac{z}{s})\in P^1$. Hence, from the facet description of $P^1$ given in the proof of Lemma \ref{cayley_1},
we deduce that $P^s$ has the following facet presentation:
$$
\langle x,\hat{\eta_i}\rangle \geq-a_{i0}, \ \ \ \ \langle x,e_0\rangle\geq -s, \ \ \ \ 
\langle x,e_j\rangle\geq 0 \ \ \ \ j=1,...,k,
$$
where  $\displaystyle \hat{\eta_i}=\eta_i+\sum_{j=1}^k\frac{(a_{ij}-a_{i0})}{s} e_j$.

Suppose that $Y$ is smooth and
$s$ divides $a_{ij}-a_{i0}$ for every $i\in\{1,\cdots,r\}$ and $j\in\{1,\cdots,k\}$.
Then $\hat{\eta_i}$ is a primitive lattice vector, 
$s$ divides $D_j-D_0$ in $T-\Div(Y)$, and one can check that $P^s$ and 
$P_{D_0}*P_{\left(\frac{D_1 - D_0}{s}+D_0\right)}*\cdots *P_{\left(\frac{D_k - D_0}{s}+D_0 \right)}$
have the same normal fan. Thus
$$
X \ \cong \ \mathbb{P}_Y\Big(\mathcal{O}(D_0)\oplus \mathcal{O}\Big(\frac{D_1-D_0}{s}+D_0\Big) 
\oplus \cdots \oplus \mathcal{O}\Big(\frac{D_k-D_0}{s}+D_0\Big) \Big),
$$
and $P^s$ is smooth. 
The facet presentations of $P^s$ and 
$P_{D_0}*P_{\left(\frac{D_1 - D_0}{s}+D_0\right)}*\cdots *P_{\left(\frac{D_k - D_0}{s}+D_0 \right)}$
also show that 
$$
L \ = \ sV(e_0) + \displaystyle \sum a_{i0}V(\hat{\eta_i}) \
                  = \ sV(e_0) + \pi^*(D_0) \
                  \sim s\xi + \pi^*((1-s)D_0).
$$

Conversely, suppose that $P^s$ is smooth, and denote by $\Sigma$ its normal fan.
For each $i\in\{1,\cdots,r\}$, let $r_i$ be the least 
positive (integer) number  such that $r_i\hat{\eta}_i$ is a lattice vector.
We deduce from Section \ref{toric_proj_bundles} that the maximal cones of $\Sigma$ are of the form:
$$
\tilde{\tau} :=\Cone \big( r_{i_1}\hat{\eta}_{i_1}, \ldots, r_{i_t}\hat{\eta}_{i_t}, e_0, \ldots \hat{e}_j, \ldots e_k\big),  
$$
where $\tau =\Cone \big( \eta_{i_1}, \ldots, \eta_{i_t} \big) \in \Sigma'(m)$
and $j\in\{0,\cdots,k\}$. 
Since $X$ is smooth, $\Sigma$ must be simplicial (i.e., $t=m$), and
\begin{align*}
1\ & = \ \mult(\tilde{\tau}) \\ 
   & = \ \big|\det[r_{i_1}\hat{\eta}_{i_1}, \ldots , r_{i_m}\hat{\eta}_{i_m}, e_0, \ldots \hat{e}_j, \ldots , e_k]\big| \\
   & = \ \big|r_{i_1} \ldots r_{i_m}\big|\cdot \big|\det[\eta_{i_1}, \ldots , \eta_{i_m}]\big| \\
   & = \ \big|r_{i_1} \ldots r_{im}\big| \cdot \mult(\tau).
\end{align*} 

It follows that $\mult(\tau)=1$ for every 
$\tau \in \Sigma'(m)$, 
and $r_i =1$ for every $i\in\{1,\cdots,r\}$.
Thus $Y$ is smooth, and
$s$ divides $a_{ij}-a_{i0}$ for every $i\in\{1,\cdots,r\}$ and $j\in\{1,\cdots,k\}$.
\end{proof}

It will be useful to compute the nef value and the $\Q$-codegree of the following type of generalized Cayley polytope.

\begin{ex} Let $0< d_0 \leq d_1 \leq ... \leq d_k$ be integers and suppose that $s$ is an integer that divides $d_j-d_0$ for every $j \in \{1, ..., k\}$. We compute the nef value and the $\Q$-codegree of the smooth polytope:
$$
P^s = \Cayley^s(d_0 \Delta_m, ..., d_k \Delta_m).
$$
Set $a_j := \frac{d_j-d_0}{s} + d_0$, $j=0, ..., k$. By Lemma \ref{cayley_s}, $a_j$ is a positive integer for every $j$, and the polarized toric variety corresponding to $P^s$ is:
$$
(X, L) = \Big( \mathbb{P}_{\mathbb{P}^m}(\mathcal{O}(a_0) \oplus ... \oplus \mathcal{O}(a_k)), s\xi + (1-s)a_0 \pi^*(H) \Big)
$$ 
where $\pi: X \to Y$ is the natural projection, $H$ is an invariant hyperplane in $\mathbb{P}^m$ and $\xi = V(e_0)+a_0 \pi^*(H)$. It follows from Remark \ref{equality_canonical} that:
$$
K_X = \Big(\sum a_j - (m+1)\Big)\pi^*(H) -(k+1)\xi.
$$
For every $t>0$ we have:
\begin{align*}
K_X + tL & = \Big(\sum a_j -(m+1)\Big)\pi^*(H)-(k+1)\big(V(e_0)+a_0\pi^*(H)\big) + \\ 
& \ \ \ \ \ tsV(e_0)+tsa_0\pi^*(H)+t(1-s)a_0\pi^*(H) \\
& = \Big[\sum a_j -(m+1)+(t-(k+1))a_0\Big]\pi^*(H)+(ts-(k+1))V(e_0).        
\end{align*}

Since $\overline{Nef}(X) = \Cone\big([V(e_0)], [\pi^*(H)]\big)$ (see Example \ref{proj_bundle_ex}), $K_X+tL$ is nef if and only if:
$$
t \geq \frac{k+1}{s} \ \ \mbox{and} \ \ \sum a_j -(m+1)+(t-(k+1))a_0 \geq 0.
$$
Replacing $a_j = \frac{d_j-d_0}{s} + d_0$, the second inequality becomes the following:
$$t \geq \frac{k+1}{s} + \frac{m+1 - \sum \frac{d_j}{s}}{d_0}.
$$
Therefore:
$$\tau(P^s) = \tau(L) = \max \Big\{\frac{k+1}{s}, \  \frac{k+1}{s} + \frac{m+1 - \sum \frac{d_j}{s}}{d_0} \Big\}.
$$

Next we compute the spectral value of $L$. From Lemma \ref{div_proj_bundle}, $V(e_k) \sim V(e_0) + (a_k-a_0)\pi^*(H)$. Thus:
$$ \xi \sim V(e_k) + a_k\pi^*(H).
$$
For every $t \geq 0$ we have:
\begin{align*}
K_X + tL & \sim \Big(\sum a_j -(m+1)\Big)\pi^*(H)-(k+1)\Big(V(e_k)+a_0\pi^*(H)\Big) + \\ 
& \ \ \ \ \ tsV(e_k)+tsa_k\pi^*(H)+t(1-s)a_0\pi^*(H) \\
& =\Big[\sum a_j -(m+1)+(st-(k+1))a_k+t(1-s)a_0\Big]\pi^*(H)+ \\ 
& \ \ \ \ \ (ts-(k+1))V(e_0).        
\end{align*}

By Example \ref{proj_bundle_ex} we have:
$$
\overline{Eff}(X) = \Cone \big([V(e_k)], [\pi^*(H)]\big).
$$
Hence $K_X+tL$ is effective if and only if:
$$
t \geq \frac{k+1}{s} \ \ \mbox{and} \ \ \sum a_j -(m+1)+(st-(k+1))a_k+t(1-s)a_0 \geq 0.
$$
Replacing $a_j = \frac{d_j-d_0}{s} + d_0$, the second inequality becomes:
$$ -(m+1) + d_kt  \geq (k+1)\Big(\frac{d_k-d_0}{s}+d_0)- \sum \Big(\frac{d_j-d_0}{s}+d_0 \Big) 
$$ 
$$  \therefore \ \ t \geq \frac{k+1}{s} + \frac{m+1-\sum \frac{d_j}{s}}{d_k}.
$$
Therefore:
$$
\codeg_{\Q}(P^s) = \mu(L) = \max \Big\{\frac{k+1}{s}, \  \frac{k+1}{s} + \frac{m+1 - \sum \frac{d_j}{s}}{d_k} \Big\}.
$$
Note that $P^s$ is $\Q$-normal exactly in two situations:
\begin{enumerate}
\item $s(m+1) < \sum d_j$. In this case, $\codeg_{\Q}(P^s) = \frac{k+1}{s}$.

\item $d_0 = ... = d_k$ and $s(m+1)>(k+1)d_0$. In this case, $P^s = (d_0\Delta_m) \times (s\Delta_k)$ and $\codeg_{\Q}(P) = \frac{m+1}{d_0}$.
\end{enumerate}

\label{key_ex_1} 
\end{ex}

\begin{ex} Consider again the toric projective bundle:
$$
X = \mathbb{P}(\mathcal{O}(a_0) \oplus ... \oplus \mathcal{O}(a_k)) \to \mathbb{P}^m,
$$
where $0<a_0 \leq ... \leq a_k$. Let $\xi$ be a divisor corresponding the tautological line bundle and $s$ a positive integer. If $H$ is an $T$-invariant hyperplane on $\mathbb{P}^m$, consider the divisor:
$$
L = s\xi + a\pi^*(H), \ \ \ a \in \Z.
$$
We conclude from the description of the tautological line bundle and of the nef cone of $X$ given in Lemma \ref{div_proj_bundle} and Example \ref{proj_bundle_ex} that $L$ is ample on $X$ if and only if $a>-sa_0$. In this case the polytope $P_L$ corresponding to $L$ is affinelly isomorphic to:
$$
P^s = \Cayley^s((sa_0 + a)\Delta_m, ..., (sa_k + a)\Delta_m).
$$
In fact, by Lemma \ref{cayley_s}, the polarized toric variety  associated to $P^s$ is:
$$
(X', L') = (\mathbb{P}_{\mathbb{P}^m}(\mathcal{E}), s\xi'+(1-s)(sa_0+a)\pi'^*(H)),
$$ 
where $\mathcal{O}_{X'}(\xi') \simeq \mathcal{O}_{X'}(1)$, $\pi':X' \to \mathbb{P}^m$ is the $\mathbb{P}^k$-bundle map and: 
$$
\mathcal{E} = \mathcal{O}(sa_0+a) \oplus \mathcal{O}(a_1-a_0+sa_0+a) \oplus ... \oplus \mathcal{O}(a_k-a_0+sa_0+a).
$$ 
Tensorizing $\mathcal{E}$ by the line bundle $\mathcal{O}(a_0(1-s)-a)$ induces a natural toric isomorphism $X' \stackrel{\simeq}{\to} X$ that takes $\xi'$ to $\xi -a_0((1-s)+a)\pi^*(H)$ (see for instance \cite[II 7.9]{hartshorne}). It follows that this isomorphism takes $L'$ to:
$$s\xi + \big[-s(a_0(1-s)+a)+(1-s)(sa_0+a)\big] \pi^*(H) = s\xi +a \pi^*(H) = L.
$$
Therefore, the pairs $(X, L)$ and $(X',L')$ are isomorphic. By Remark \ref{aff_isom_toric_pol}, $P_L \simeq P^s$.

\label{key_ex_2}
\end{ex}

\vspace{0,3 cm}

In Example \ref{key_ex_1} we determined sufficient conditions for the generalized Cayley polytope 
$$
\Cayley^s(d_0 \Delta_m, ..., d_k \Delta_m), \ \ 0<d_0 \leq ... \leq d_k
$$
to be $\Q$-normal. 
Our next goal is to do the same for more general polytopes. We begin proving a toric geometric result. 

\begin{lemma} Let $Y$ be an $n$-dimensional smooth projective toric variety and $L$ a $\Q$-divisor on $Y$. If $L \cdot C \geq n$ for every $T$-invariant curve $C \subset Y$, then either $K_Y + L$ is nef or $Y \simeq \mathbb{P}^n$ and $L \sim tH$ where $H$ is a hyperplane and $t$ is a rational number with $n \leq t < n+1$.

\label{fujita_conj}
\end{lemma}

\begin{proof}
Let $R$ be a $K_Y$-negative extremal ray of $\overline{NE}(Y)$, $\phi_R$ the contraction associated to $R$, $E$ any component of the excepcional locus of $R$, and $F$ any component of the fiber of the restriction ${\phi_R}_{|E}$. Since $X$ is smooth, there exists a $T$-invariant curve $C$ whose class generates $R$, and satisfies:
$$
\mathfrak{l}(R) = -K_Y \cdot C,
$$
where $\mathfrak{l}(R)$ is the lenght of $R$, as defined in Section \ref{log}.We have seen that $-K_Y \cdot C \leq n+1$.

If $K_Y + L$ is not nef, there exist such a curve $C$ such that $-K_Y \cdot C =n+1$. Proposition \ref{key_inequality} implies that $E = F = Y$. Thus $Y \simeq \mathbb{P}^n$ and since $K_Y+L$ is not nef, we have $L \sim tH$ where $H$ is a hyperplane on $Y$ and $t$ is a rational number with $n \leq t < n+1$. 
\end{proof}

The following is the main result of this section. It improves the $\Q$-normality criterion given in \cite[Proposition 3.9]{alicia}.

\begin{prop}
Let the notation be as in \ref{notation:P_i}, and suppose that:
$$
P^s:=\Cayley^s(P_0, \ldots ,P_k)
$$ 
is smooth. If $\frac{k+1}{s} \geq m$, then   $\codeg_{\mathbb{Q}}(P^s) = \frac{k+1}{s}$ and $P^s$ is $\mathbb{Q}$-normal, unless one of the following holds:
\begin{enumerate}
\item[(1)] $P^s \simeq \Cayley^s(d_0 \Delta_m, ..., d_k \Delta_m)$, where $0<d_0\leq ... \leq d_k$, $d_0 \neq d_k$ and $sm\leq \sum_{j=0}^k d_j < s(m+1)$. In this case, $P^s$ is not $\Q$-normal and:
$$ \codeg_{\Q}(P) = \frac{k+1}{s} + \frac{m+1 - \sum \frac{d_j}{s}}{d_k},
$$
$$\tau(P) = \frac{k+1}{s} + \frac{m+1 - \sum \frac{d_j}{s}}{d_0}.
$$

\item[(2)]$P^s \simeq (r\Delta_m) \times (s\Delta_k)$, where $r$ is a positive integer with $s(m+1)>(k+1)r \geq sm$. In this case $P^s$ is $\Q$-normal and $\codeg_{\Q}(P) = \frac{m+1}{r}$.
\end{enumerate}
\label{prop1}
\end{prop}

\begin{proof}
The second projection $f:\mathbb{R}^m \times \mathbb{R}^k \to \mathbb{R}^k$ maps $P^s$ 
onto the $k$-dimensional simplex $s\Delta_k$. 
Therefore, $ \tau(P^s)\geq \codeg_{\mathbb{Q}}(P^s) \geq \codeg_{\mathbb{Q}}(s\Delta_k)=\frac{k+1}{s}$.

Recall from Lemma \ref{cayley_s} that the polarized toric variety associated to $P^s$ is 
$$
(X,L) \ \cong \ \Big(\mathbb{P}_Y\big(\mathcal{O}(E_0)\oplus \cdots \oplus \mathcal{O}(E_k) \big), 
\ s\xi +(1-s)\pi^*(D_0) \Big),
$$
where $Y$ is smooth, $\pi:X \to Y$ is the $\mathbb{P}^k$-bundle map, $\xi$ is a divisor corresponding to the tautological line bundle on $X$ and $E_i := \frac{D_i + (s-1)D_0}{s}, \forall i=0,...,k$. 
The $\Q$-divisor
$$
M := \displaystyle \sum_{i=0}^{k} E_i - \frac{(k+1)(s-1)}{s}D_0 = \frac{1}{s} \sum_{i=0}^{k} D_i 
$$
satisfies $M \cdot C \geq \frac{k+1}{s} \geq m$ for every $T$-invariant integral curve $C \subset Y$. By Lemma \ref{fujita_conj}, we conclude that either $K_Y + \mathcal{L}$ is nef, or $Y \simeq \mathbb{P}^m$ and $sm \leq \sum_{j=0}^k d_j < s(m+1)$, where $d_j$ denotes the degree of the ample divisor $D_j$ under the isomorphism $Y \simeq \mathbb{P}^m$.

Suppose that $K_Y + M$ is nef. Then $\pi^*(K_Y + M)$ is nef but not ample on $X$. Since $L$ is ample, it follows from Remark \ref{equality_canonical} that
$$
K_X + tL = \pi^*(K_Y + M) + \big(t -\frac{k+1}{s} \big)L
$$
is ample if and only if $t-\frac{k+1}{s} >0$. Hence $\tau(P^s)=\frac{k+1}{s}$, as desired.

Suppose now that $Y \simeq \mathbb{P}^m$ and $sm \leq \sum_{j=0}^k d_j < m+1$. By Lemma \ref{cayley_s} we conclude that:
$$P^s \simeq \Cayley^s(d_0 \Delta_m, ..., d_k \Delta_m).
$$
Therefore, we are reduced to the computations made in Example \ref{key_ex_1}. This concludes the proof.  
\end{proof}

\section{Main Result}

We have the following classification theorem for polytopes:

\begin{teo} Let $P \subset \mathbb{R}^n$ be a smooth $n$-dimensional $\mathbb{Q}$-normal lattice polytope. Then $\codeg(P) \geq \frac{n+1}{2}$ if and only if $P$ is affinely isomorphic to one of the following polytopes:
\begin{enumerate}
\item[{\rm (i)}] $2\Delta_n$;
\item[{\rm (ii)}] $3\Delta_3 \ (n=3)$;
\item[{\rm (iii)}] $s\Delta_1$, $s\geq 1 \ (n=1)$;
\item[{\rm (iv)}] $\Cayley^1(P_0, \ldots, P_k)$, where $k=\codeg(P)-1 \geq \frac{n-1}{2}$;
\item[{\rm (v)}] $\Cayley^2(d_0\Delta_1, d_1\Delta_1, \ldots, d_{n-1}\Delta_1)$, where $n \geq 3$ is odd and the $d_i$'s are congruent modulo 2.
\end{enumerate}
\label{mainthm}
\end{teo} 

Our proof of Theorem~\ref{mainthm} follows the same strategy of \cite{alicia}: we interpret the codegree of $P$ as a geometric invariant of the corresponding polarized variety $(X,L)$,
and then apply techniques from Adjunction Theory and Mori Theory to classify this pair. Later we recover the description of the polytope $P$.
We state a classical result about projective bundles that will be useful (see for instance \cite[3.2.1]{beltrametti}):

\begin{prop}[Fujita's Lemma] Let $X$ be a smooth projective variety and $\pi: X \to Y$ a flat morphism with connected fibers onto a normal variety $Y$. Suppose that there exists an ample line bundle $\mathcal{L}$ on $X$ such that for a general fiber $F$ of $\pi$ we have $(F, \mathcal{L}_{|F}) \simeq (\mathbb{P}^k, \mathcal{O}_{\mathbb{P}^k}(1))$, $k \in \Z$. Then $(X, L) \simeq (\mathbb{P}_Y(\mathcal{E}), \mathcal{O}_{\mathbb{P}_Y(\mathcal{E})}(1))$, where $\mathcal{E} = \pi_*\mathcal{L}$. 
\label{fuj_lemma}
\end{prop} 

\begin{proof}[Proof of Theorem \ref{mainthm}]

First note that the five classes of polytopes listed in Theorem \ref{mainthm} are $\mathbb{Q}$-normal and have codegree $\geq \frac{n+1}{2}$. This is straightforward for polytopes (i), (ii) and (iii). 

Let $P$ be a polytope as in (iv). Setting $m = \dim(P_i)$, we have $n = m + k$, and the inequality $k \geq \frac{n-1}{2}$ is equivalent to $k +1 \geq m$. Hence, $P$ satisfies the hypothesis of Proposition \ref{prop1} and then $P$ is $\Q$-normal and $\codeg(P) = k+1 \geq \frac{n+1}{2}$ (as desired), unless $P$ is a polytope as items (1) or (2) of that Proposition. In this case, since $k+1 \geq m$, it is not difficult to see that (1) does not occur and (2) holds if and only if $k=m-1$, $P \simeq \Delta_m \times \Delta_{m-1}$ and $\codeg(P) = m+1 = \frac{n+1}{2}$.

Next, let $P$ be a polytope as in (v). Since $n \geq 3$, by Proposition \ref{prop1}, $\codeg_{\Q}(P) = n/2$ unless $P$ satisfies (1) or (2). As in the later case, (1) does not occur, and $P$ satisfies (2) if and only if $P \simeq \Delta_1 \times (2\Delta_{n-1})$ whose $\Q$-codegree is $n/2$. In this case, since $n$ is odd, it follows from Proposition \ref{teto} that $\codeg(P) = \lceil \frac{n}{2} \rceil = \frac{n+1}{2}$, as required.\\

Conversely, let $P \subset \mathbb{R}^n$ be a smooth $n$-dimensional $\mathbb{Q}$-normal lattice polytope with $\codeg(P) \geq \frac{n+1}{2}$ and denote by $(X,L)$ the corresponding polarized toric variety. We may assume that $n>1$. Set $\tau:=\tau(P)$. By Proposition \ref{teto}, $\tau > \codeg(P)-1$. Thus $\tau > \frac{n-1}{2}$. It follows from the discussion in Section \ref{nef_spectral} that the nef value morphism  $\phi=\phi_L: X \to Y$ is defined by the linear system $|e(K_X + \tau L)|$, for a large and divisible enough integer $e$. Moreover, the assumption that $P$ is $\Q$-normal implies that $\dim(Y) < \dim(X)$ (Proposition \ref{nonbirational}). If $C \subset X$ is an extremal curve contracted by $\phi$ which satisfies $\mathfrak{l}(\R_+[C]) = -K_X \cdot C$, then:
\begin{equation}
n+1 \geq  -K_X \cdot C = \tau \cdot L \cdot C > \frac{n-1}{2} L \cdot C.
\label{eq_1}
\end{equation}
In particular, $L \cdot C \leq 5$. We consider three cases.

\vspace{0.2cm}

{\it Case 1}: Suppose that $L \cdot C = 1$ for {\it every extremal curve} $C \subset X$ contracted by $\phi$. We have $\tau = -K_X \cdot C \in \Z$. Since $\codeg(P)$ is also integer and $\tau > \codeg(P) -1$, it follows that $\tau \geq \frac{n+1}{2}$.  

If $\dim(Y) = 0$ then the divisor $K_X +\tau L$ is zero. Hence $-K_X = \tau L$ and then $X$ is a Fano variety. Denote by $r$ the index of $X$ (see \ref{ind}). Then, $r \geq \tau \geq \frac{n+1}{2}$. By Remark \ref{toric_fano}, $X$ is isomorphic to one of the following: $\mathbb{P}^n$,  $\mathbb{P}^{\frac{n}{2}} \times \mathbb{P}^{\frac{n}{2}}$ ($n$ even), $\mathbb{P}^1 \times \mathbb{P}^1 \times \mathbb{P}^1$ ($n=3$), or $\mathbb{P}_{\mathbb{P}^r}({\mathcal{O}(2) \oplus \mathcal{O}(1)^{r-1}})$ ($n=2r-1$). 

In the first three cases we have $P \simeq \Delta_n$,  $P \simeq \Delta_{n/2} \times \Delta_{n/2}$ and $P \simeq \Delta_1 \times \Delta_1 \times \Delta_1$ respectively. These are strict Cayley polytopes as in (iv). In the last case, let $\pi:X \to \mathbb{P}^r$ be the $\mathbb{P}^{r-1}$-bundle map, $H$ a hyperplane in $\mathbb{P}^r$ and $\xi$ a divisor on $X$ corresponding to the tautological line bundle. Then by Remark \ref{equality_canonical}:
$$
-K_X \sim -\pi^{\ast}(K_{\mathbb{P}^r} + (r+1)H) - r\xi \sim -r\xi.
$$
It follows that $L \sim \xi$. By Lemma \ref{cayley_s}, $P \simeq \Cayley^1(\underbrace{\Delta_r, \ldots, \Delta_r}_{r-1 \ times}, 2\Delta_r).$

\vspace{0,2 cm}
Suppose now that $\dim(Y)>0$. Then $\tau \geq \frac{n+1}{2} > \frac{n+1-\dim(Y)}{2}$. Since $\phi$ is not birational,  we conclude from Theorem \ref{key_result_lines} that $X$ is covered by lines. Using Theorem \ref{key_result_contraction} and Proposition \ref{Mori_fiber_pkbundle}, we see that there exists an extremal ray $R$ of $NE(X)$ whose associated contraction $\phi_R: X \to Z$ realizes $X$ as a projectivization of a vector bundle $\mathcal{E}$ of rank $k+1$ on another smooth toric variety $Z$, with $\dim(Z) \leq n- \tau +1$. Since that $n = \dim(Z) + k$, we have $k \geq \tau -1 \geq \frac{n-1}{2}$. Since $L \cdot C = 1$, we have $\mathcal{O}_X(L)_{|F} \simeq \mathcal{O}_{\mathbb{P}^k}(1)$. By Fujita's Lemma, $X \simeq_Z \mathbb{P}_Z({\phi_R}_*\mathcal{O}_X(L))$, and under this isomorphism $\mathcal{O}_X(L)$ corresponds to the tautological line bundle. Since $X$ is toric, the ample vector bundle ${\phi_R}_*\mathcal{O}_X(L)$ splits as a sum of k+1 ample line bundles on $Z$. By Lemma \ref{cayley_s}, there are lattice polytopes $P_0, ..., P_k$ with same normal fan such that $P \simeq \Cayley^1(P_0, ..., P_k)$.     
\vspace{0.2cm}

{\it Case 2}: Suppose that there is an extremal curve $C$ that is contracted by $\phi$ such that $L \cdot C = 2$. Let $R$ be the extremal ray generated by $C$ and $\phi_R: X \to Z$ the associated contraction. By (\ref{eq_1}), $-K_X \cdot C \in \{n, n+1\}$. Let $E$ be the excepcional locus of $\phi_R$ and $F$ an irreducible component of the fiber of the restriction $\phi_R|_E$. By Proposition \ref{key_inequality}, $\dim(E) = n$ and $n-1 \leq \dim(F) \leq n$.

If $\dim(F) = n$, then $(X, \mathcal{O}_X(L)) \simeq (\mathbb{P}^n, \mathcal{O}(2))$, and $P \simeq 2\Delta_n$.

If $\dim(F) = n-1$, by Proposition \ref{Mori_fiber_pkbundle}, $\phi_R: X \to \mathbb{P}^1$ is a projective bundle of rank $n-1$ with $\mathcal{O}_X(L)_{|F} \simeq \mathcal{O}_{\mathbb{P}^{n-1}}(2)$. So there are integers $0<a_0 \leq ... \leq a_{n-1}$ and $a>-2a_0$ such that:
\begin{align*}
X & \simeq \mathbb{P}_{\mathbb{P}^1}(\mathcal{O}(a_0) \oplus ... \oplus \mathcal{O}(a_{n-1})) \notag \\
L & \simeq 2\xi + aF \notag, 
\end{align*}
where $\xi$ is a divisor corresponding to the tautological line bundle. By Example \ref{key_ex_2},
$$
P \simeq \Cayley^2\big((2a_0+a)\Delta_1, ..., (2a_{n-1}+a)\Delta_1\big).
$$

\vspace{0.2 cm}

{\it Case 3} : Suppose that there is an extremal curve $C$ that is contracted by $\phi$ with $3 \leq L \cdot C \leq 5$. By (\ref{eq_1}), we must have $n \leq 4$. If $3 \leq n \leq 4$, then Proposition \ref{key_inequality} and (\ref{eq_1}) imply that $L \cdot C = 3$ and $X \simeq \mathbb{P}^n$. Thus $P \simeq 3\Delta_n$. For $n \in \{3, 4\}$, $\codeg(3\Delta_n)=2$, which is $\geq \frac{n+1}{2}$ only if $n=3$.

From now, suppose that $n=2$. If $L \cdot C \in \{4, 5\}$, then Proposition \ref{key_inequality} implies that $-K_X \cdot C = 3$ and thus $X \simeq \mathbb{P}^2$. On other hande, $4\Delta_2$ and $5\Delta_2$ do not satisfy the codegree hypothesis. So we must have $L \cdot C = 3$. We conclude from Prop. \ref{key_inequality} and (\ref{eq_1}) that there are integers $0<a_0\leq a_1$ and $a>-3a_0$ such that:
\begin{align*}
X & \simeq \mathbb{P}_{\mathbb{P}^1}\big(\mathcal{O}(a_0) \oplus \mathcal{O}(a_1)\big) \notag \\
L & \simeq 3\xi + aF \notag, 
\end{align*}
where $\xi$ is a divisor corresponding to the tautological line bundle and $F$ is a fiber of $X \to \mathbb{P}^1$. By Example \ref{key_ex_2},
$$
P \simeq \Cayley^3\big((3a_0 + a)\Delta_1, (3a_1+a)\Delta_1)\big).
$$
On the other hand, by Example \ref{key_ex_1}, $P$ is $\Q$-normal and has codegree $\geq \frac{n+1}{2}$ if and only if $P \simeq 3\Delta_1 \times \Delta_1 = \Cayley^1(3\Delta_1, 3\Delta_1)$, which is a polytope as in (iv).
\end{proof} 

The following is the polytope version of Conjecture \ref{conj_codegree}:

\begin{conj} Let $P\subset \mathbb{R}^n$ be a smooth $n$-dimensional lattice polytope. 
If the inequality $\codeg_{\mathbb{Q}}(P) > \frac{n+1}{2}$ holds, then $P$ is $\mathbb{Q}$-normal.
\label{bs}
\end{conj}

If this conjecture holds, then Theorem \ref{mainthm} implies that smooth polytopes $P$ with 
$\codeg_{\mathbb{Q}}(P) > \frac{n+1}{2}$ are those in (iv) with $k \geq \frac{n}{2}$. 
By Proposition \ref{prop1}, these polytopes have $\mathbb{Q}$-codegree $\geq \frac{n+2}{2}$. 
Hence, if Conjecture \ref{bs} holds, then the $\mathbb{Q}$-codegree of smooth lattice polytopes
does not assume values in the interval $\left( \frac{n+1}{2}, \frac{n+2}{2} \right)$.

\section{Is $\Q$-normality necessary?}
An immediate consequence of Theorem \ref{mainthm} is stated below:
 
\begin{cor} Let $P \subset \mathbb{R}^n$ be a smooth $n$-dimensional $\mathbb{Q}$-normal lattice polytope. If $\codeg(P) \geq \frac{n+1}{2}$, then $P$ is a strict generalized Cayley polytope.
\label{bound_tobe_ cayley}
\end{cor}

We have seen in Theorem \ref{DDRP} that if $P$ is a smooth $n$-dimensional lattice polytope with $\codeg(P) \geq \frac{n+3}{2}$ then $P$ is a strict Cayley polytope. This result was first proved with the assumption that $P$ is $\Q$-normal, but later it was showed that this can be dropped. 

The following example shows that, to guarantee the result of Corollary \ref{bound_tobe_ cayley}, the $\Q$-normality assumption cannot be dropped . We exhibit a non $\Q$-normal smooth lattice polytope $P$ that satisfies $\codeg(P) \geq \frac{\dim(P)+1}{2}$ but $P$ is not a generalized strict Cayley polytope.

\begin{ex} Let $m$ be a positive integer and $H \subset \mathbb{P}^m$ a hyperplane. Let $\pi: X \to \mathbb{P}^m \times \mathbb{P}^1$ the blow up of $\mathbb{P}^m \times \mathbb{P}^1$ along $H_o = H \times \{o\}$. Then $X$ is a smooth projective toric variety with Picard number 3. We will see that $X$ is Fano, and the Mori cone $\overline{NE}(X)$ has exactly 3 extremal rays, whose corresponding contractions are all divisorial contractions. Since $X$ does not admit any contraction of fiber type, $P_L$ is not a strict generalized Cayley polytope for any ample divisor $L$ on $X$ by Lemma \ref{cayley_s}. When $m$ is even, we will then exhibit an ample divisor $L$ on $X$ such that $\codeg(P_L) \geq \frac{\dim(P_L) +1}{2}$.

Let $\{e_1, \ldots, e_m, e\}$ be the canonical basis of $\mathbb{R}^m \times \mathbb{R}$ and $e_0 = -e_1 - ... -e_m$. 
The maximal cones of the fan $\Sigma$ of $\mathbb{P}^m \times \mathbb{P}^1$ is $\Cone( e_0, \ldots , \hat{e_i}, \ldots e_m, \pm e)$, $i=0, \ldots n$. Set $f=e_1+e$. The fan $\Sigma_X$ of $X$ is obtained from $\Sigma$ by 
star subdivision centered in $f$. Let $D_i := V(e_i), \ i=0, \ldots m, \ D_e:=V(e)$ and $E:=V(f)$. One check that
$D_i \sim D_0$ for $i>1$, $V(-e) \sim D_e + E$ and $D_0 \sim D_1 + E$. Therefore the classes of $D_1, D_e$ and $E$  form a basis for $N^1(X)$ and generates the cone $\overline{Eff}(X)$. Written in this basis:
$$-K_X = \sum_{i=0}^{m} D_i + D_e + V(-e) + E \sim (m+1)D_1 +2D_e + (m+2)E.
$$
On $N_1(X)$ the Mori Cone $\overline{NE}(X)$ is generated by classes of invariant curves (see Sec. \ref{mov_toric}). One can check that $\overline{NE}(X)$ is generated by the numerical classes of invariant curves $C_1$, $C_2$ and $C_3$ associated to the cones $\Cone (e_1,e_2, \ldots e_m)$, $\Cone(e_2, \ldots e_m, f)$ and $\Cone(e, e_2, \ldots e_m)$ respectively. For each $i \in \{1, 2, 3\}$, denote  by $\pi_i$ the contraction of the extremal ray generated by $[C_i]$. One checks that $\pi_1: X \to \mathbb{P}_{\mathbb{P}^m}(\mathcal{O} \oplus \mathcal{O}(1))$ blows down the divisor $D_1$ onto a $\mathbb{P}^{m-1}$, $\pi_2=\pi$ and $\pi_3:X \to \mathbb{P}_{\mathbb{P}^1}(\mathcal{O}(1) \oplus \mathcal{O}^{\oplus m})$ blows down the divisor $D_e$ onto a point (see Figure \ref{blow_downs}). 

In terms of the basis for $N^1(X)$ and $N_1(X)$ given above, the intersection products between divisors and curves are given by:
\[D_1 \cdot C_1 = -1, D_1 \cdot C_2 = 1 , D_1 \cdot C_3 = 0 \]
\[D_e \cdot C_1 = 0, D_e \cdot C_2 = 1, D_e \cdot C_3 = -1 \]
\[E \cdot C_1 = 1, E \cdot C_2 = -1 , E \cdot C_3 = 1. \]
By Kleiman Ampleness Criterion, a divisor $D = aD_1 + bD_e + cE$ is ample if and only if $-a+c > 0$, $a+b-c >0$ and $-b+c>0$. Thus $L = 2D_1 + 2D_e + 3E$ is ample on $X$, and $K_X+tL$ is ample if and only if $t>m$, hence $\tau(L) = m$. Since $\overline{Eff}(X)=\Cone(D_1,D_e,E)$, $K_X + tl \in \overline{Eff}(X)$ if and only if $t \geq \frac{m+1}{2}$. Thus $\codeg_{\Q}(P_L) = \frac{m+1}{2}$. When $m$ is even, $\codeg(P_L) \geq \lceil \codeg_{\Q}(P_L) \rceil = \frac{m+2}{2} = \frac{\dim(P_L) + 1}{2}$. 
\label{contra}
\end{ex}

\begin{figure}[h]
\centering
\includegraphics[scale=0.54]{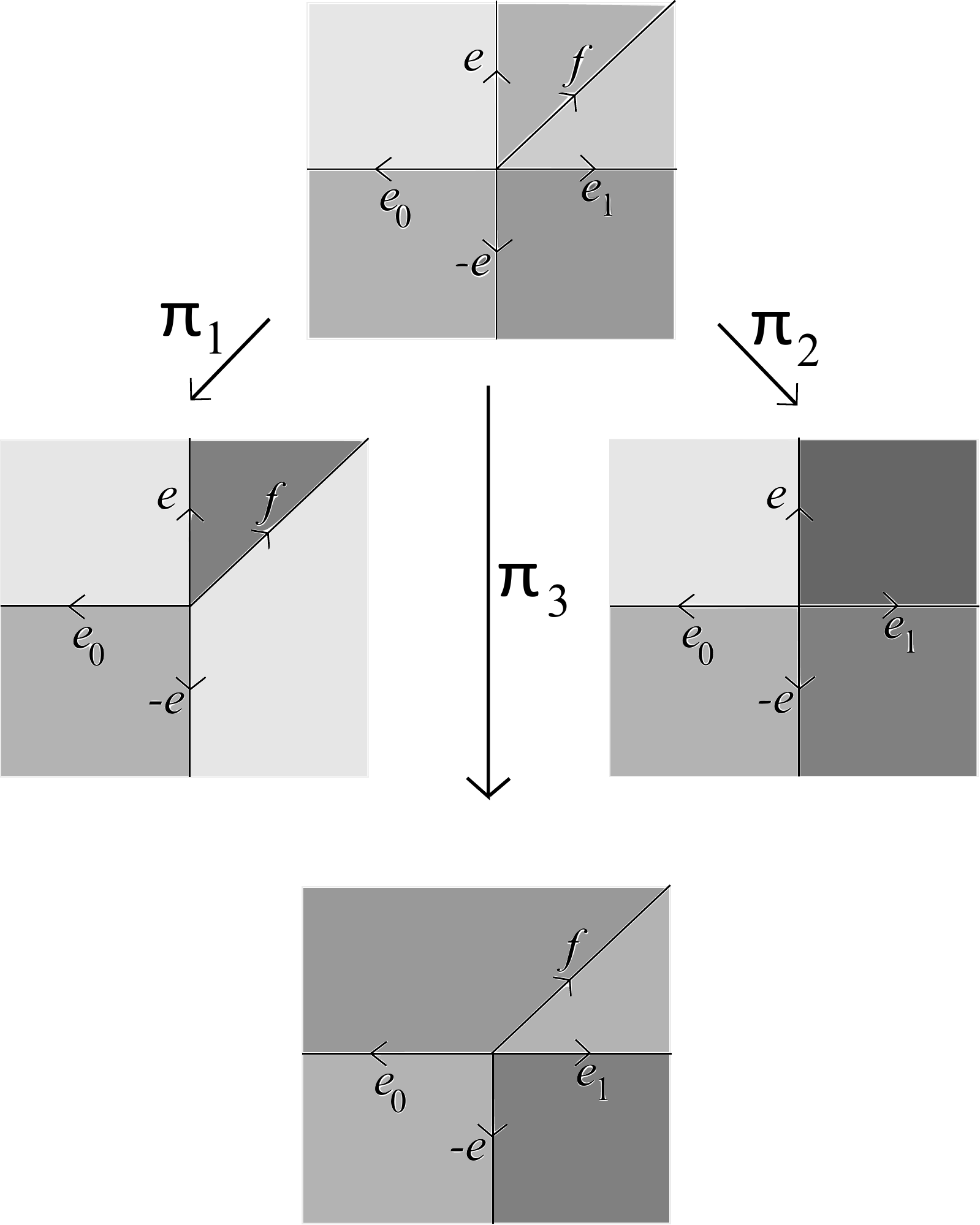}
\caption{The divisorial contractions $\pi_1$, $\pi_2$ and $\pi_3$.}
\label{blow_downs}
\end{figure}

%%%%%%%%%%%%%%%%%%%%%%%%%%%%%%%%%%%%%%%%%%%
%%%%%%%%%%%%%%%%%%%%%%%%%%%%%%%%%%%%%%%%%%%
%%%%%%%%%%%%%%%%%%%%%%%%%%%%%%%%%%%%%%%%%%%
%
%
%        CHAPTER 5
%
%
%%%%%%%%%%%%%%%%%%%%%%%%%%%%%%%%%%%%%%%%%%%
%%%%%%%%%%%%%%%%%%%%%%%%%%%%%%%%%%%%%%%%%%%
%%%%%%%%%%%%%%%%%%%%%%%%%%%%%%%%%%%%%%%%%%%

\chapter{Some results about Cones of Cycles on Toric Varieties}

In this chapter we turn our attention to the cones of cycles  introduced in Section \ref{cones_cycles}. Throughout the following sections, $N$ will be a lattice of rank $n$ and $M$ will be the dual lattice of $N$. We write $N_{\R}$ to denote the vector space $N \otimes_{\Z} \R$. We denote by $M_{\R}$  the dual of $N_{\R}$.

\section{Pseudo-effective Cycles} 

Let $X=X_{\Sigma}$ be a smooth $n$-dimensional complete toric variety. Recall from Section \ref{cones_cycles} that for every non-negative integer $k$ with $0 \leq k \leq n$, the \emph{cone of pseudo-effective $k$-cycles} on $X$ is denoted by $\overline{NE}_k(X)$. This cone is the closure in $N_k(X)$ of the cone generated by the effective $k$-cycles on $X$. In \cite[1.6]{reid} it is showed that the Mori Cone $\overline{NE}_1(X)$ is generated by classes of $T$-invariant curves on $X$. The idea of the proof is very simple: let $C \subset X$ be an irreducible complete curve. By an induction argument, we can suppose that $C$ intersects the torus $T\subset X$. Then, one choose a suitable $1$-dimensional subtorus $T' \subset T$ and, using the action of $T'$, one can move $C$ to a (possibly reducible) curve $C'$ contained in $X-T$. Thus, $C$ is numerically equivalent to $C'$ and the latter is numerically equivalent to a effective $T$-invariant cycle. Our goal in this section is generalize this result to cycles of arbitrary dimension. First, we discuss some ingredients to be used in the proof.

\begin{say}[One-parameter subgroups] Let $v = (v_1, ..., v_n) \in N$, $v \neq 0$. The \emph{one-parameter subgroup} associated to $v$ is the morphism $\lambda_v: \C^* \to T_N$, given by:
$$
\lambda_v(t) = (t^{v_1}, ..., t^{v_n}).
$$
One-parameter subgroups are related to distinguished points of  toric varieties. Let $X=X_{\Sigma}$ be a toric variety with fan $\Sigma$ in $N_{\R}$ and $\sigma \in \Sigma$. We have (see \cite[3.2.2]{cox}):
$$
v \in \sigma \Leftrightarrow \displaystyle \lim_{t \to 0} \lambda_v(t) \ \mbox{exists in} \ \mathcal{U}_{\sigma}.
$$
Moreover, if $u$ belongs to the relative interior of $\sigma$, then $\lim_{t \to 0} \lambda_v(t) = x_{\sigma}$.

The one-parameter subgroup associated to $v \in N$ defines a $\C^*$-action $\C^* \times X \to X$, $(t,x) \mapsto \lambda_v(t) \cdot x$. We claim that we can choose $v \in N$ such that this action has finitely many fixed points. Indeed, let $\tau \in \Sigma(k), \ k<n$. Let $O(\tau) \simeq T_{N(\tau)}$ be the orbit in $X$ corresponding to $\tau$ and $x \in O(\tau)$. Since $\dim(\tau) <n$, $O(\tau)$ is not a point. Pick $v_\tau \in N$ such that $v_{\tau} \notin N_{\tau} = \Span(\tau) \cap N$. Thus, the class $\bar{v}_{\tau}$ of $v_{\tau}$ in $N(\tau) = N/N_{\tau}$ is not zero. Hence for $t \neq 1$, $\lambda_{\bar{v}_{\tau}}(t)$ is not the unity in $T_{N(\tau)}$. Thus, if $t \neq 1$:
$$
\lambda_{v_{\tau}}(t) \cdot x = \lambda_{\bar{v}_{\tau}}(t) \cdot x \neq x.
$$
Therefore, there is no fixed points of the action of $\lambda_{v_{\tau}}$ in the orbit $O(\tau)$. 

Finally, let $v \in N$ be such that $v \notin \Span(\tau)$ for every $\tau \in \Sigma(k)$, $k<n$. Clearly, the distinguished points $x_{\sigma}$, $\sigma \in \Sigma(n)$, are fixed points of $\lambda_v$ since they are fixed by the action of $T_N$. We prove that these are the unique fixed points of $\lambda_v$. In fact, let $x \in X$ be a such fixed point. From Cone-Orbit correspondence, $x$ belong to a $T_N$-orbit $O(\tau)$, $\tau \in \Sigma(k)$. From what has been done above, we have $k = n$ and $x = x_{\tau}$ as required.
\end{say}

\begin{say}[Chow varieties] Let $X$ be a complex projective variety with a fixed embedding $i: X \hookrightarrow \mathbb{P}^N$. Let $k$ and $d$ be non-negative integers with $0 \leq k \leq \dim(X)$. Consider the set of effective $k$-cycles of $X$:
$$
Chow_{k,d}(X) = \Big\{ \sum_i a_i V_i \ \Big| \ a_i>0, \ \sum_i a_i\deg(V_i) = d \Big\}.
$$
One can show that $Chow_{k,d}(X)$ has structure of a (possibly reducible) projective variety (see for instance \cite[Chap. 21]{harris}). The \emph{Chow varieties of $k$ cycles on X} is defined by:
$$
Chow_k(X) = \coprod_{d\geq0} Chow_{k,d}(X).
$$ 

Suppose in addition that $X$ is a toric variety and let $\lambda_v: \C^* \to T$ be a one-parameter subgroup. The action of $\lambda_v$ in $X$ induces an action in the Chow variety of $k$-cycles:
\begin{align*}
\C^* \times Chow_k(X) & \longrightarrow Chow_k(X) \\
\Big(t, \sum_i a_iV_i \Big)  & \longmapsto \Big( \sum_i a_i (\lambda_v(t)\cdot V_i) \Big).
\end{align*}
\end{say}

The following is the main result of this section:

\begin{teo} 
Let $X=X_{\Sigma}$ be a smooth complete toric variety and $V \subset X$ any $k$-dimensional subvariety. Then:
$$
V \equiv \displaystyle \sum_{\omega \in \Sigma(n-k)} a_{\omega} \cdot V(\omega), \ \ \ a_{\omega} \geq 0.
$$
In particular, the cone $\overline{NE}_k(X)$ is polyhedral.
\label{thm_cone}
\end{teo}   

\begin{proof} We begin supposing that $X$ is a projective variety. We can suppose that $k<n$. The proof will be by induction in $\dim(X)$. If $\dim(X) = k+1$, then $V$ is a Weil divisor on $X$ and the theorem follows from Proposition \ref{eff_toric_div}. Suppose now that the theorem holds for projective varieties with dimension $<n$. We can suppose that $V$ intersects the torus $T$ of $X$ (otherwise, $V$ would be contained in an invariant subvariety of $X$ and the result follows by the induction hypothesis). 

Choose $v_1 \in N$ such that the $\C^*$-action $\C^* \times X \to X$ given by $(t, x) \mapsto \lambda_{v_1}(t) \cdot x$ has finitely many fixed points. This action induces a natural rational map $f: \C \times V \dashrightarrow X$. By a sequence of blowups along smooth centers, we obtain a variety $Y$ with a morphism $\psi: Y \to X$ that resolves the map $f$.
$$\xymatrix{
Y \ar@/_0.8cm/[dd]_{\pi}
\ar[d]_p \ar[dr]^{\psi} & \\ 
\mathbb{C}\times V \ar[d]_{p_1} \ar@{-->}[r]^f  & X  \\
\mathbb{C} & & 
}$$
By \cite[III 9.6 and 9.7]{hartshorne}, $\pi = p_1 \circ p$ is a flat morphism and the fibers of $\pi$ have pure dimension $k$. Consider the $\C^*$-action $\C^* \times (\C \times V) \to \C \times V$, $(t', (t,x)) \mapsto (t\cdot t', x)$. Since the indeterminacy locus of $f$ is contained in $\{0\} \times V$, this action lifts to a $\C^*$-action $(t', y) \mapsto t' \cdot y$ in $Y$. Furthermore, for every $t' \in \C^*$ and $y \in Y$, $\psi(t' \cdot y) = \lambda_{v_1}(t') \cdot \psi(y)$.  It follows that the effective cycle $\psi_*(\pi^*(0))$ is invariant by the action of $\lambda_{v_1}$. Note that the strict transform of $\{0\} \times V$ is contracted to a fixed point of this one-parameter subgroup. We have that $V = \psi_*(\pi^*(1)) \equiv \psi_*(\pi^*(0))$.
In terms of Chow varieties, we proved that if $V_t := \lambda_{v_1}(t)\cdot V$, $t \in \C^*$, the limit $V_0:=\lim_{t \to 0} V_t$ exists in $Chow_k(X)$. Moreover, $V_0$ is invariant by the action of $\lambda_{v_1}$ in $Chow_k(X)$ and $V \equiv V_0$. 

If each irreducible $k$-cycle appearing in the expression of  $V_0$ does not intersect $T$, we use the induction hypothesis for these cycles and hence, we are done. Otherwise, we choose another one-parameter subgroup $\lambda_{v_2}$ with finitely many fixed points such that $\{v_1, v_2\}$ is a linearly independent set. The limit $V'_0 = \lim_{t \to 0} \big( \lambda_{v_2}(t) \cdot V_0 \big)$ is invariant by $\lambda_{v_2}$ and $V \equiv V'_0$. We claim that $V'_0$ is also invariant by $\lambda_{v_1}$. In fact, for every $t' \in \C^*$:
\begin{align*}
\lambda_{v_1}(t') \cdot V'_0 & = \lambda_{v_1}(t) \cdot \big(\lim_{t \to 0} \lambda_{v_2}(t) \cdot V_0 \big) \\ 
                               & = \lim_{t \to 0} \big(\lambda_{v_2}(t) \cdot \lambda_{v_1}(t') \cdot V_0 \big) \\
                               & = \lim_{t \to 0} \big( \lambda_{v_2}(t) \cdot V_0 \big) = V'_0, 
\end{align*}
where the second equality follows from the continuity of $\lambda_{v_1}$ and the latter because $V_0$ is invariant by this one-parameter subgroup. If $V'_0$ does not intersect the torus, then we are done. Otherwise, we repeat the process. We claim that this process stops after at most $k+1$ steps. In fact, suppose that there are $v_1, ..., v_{k+1} \in N$ linearly independent elements, and a subvariety $\tilde{V} \subset X$ invariant by $\lambda_{v_i}$, $i = 1, ..., k+1$, such that $V \equiv \tilde{V}$. If there were $t \in \tilde{V} \cap T$, we would have a injective map:
\begin{align*}
(\C^*)^{k+1} & \longrightarrow  \tilde{V}  \\
  (t_1, ..., t_{k+1}) & \mapsto \lambda_{v_1}(t_1) \cdot ... \cdot \lambda_{v_{k+1}}(t_{k+1}) \cdot t.
\end{align*}
But this is absurd because $\dim(\tilde{V}) = k$. This concludes the proof of the theorem for projective varieties.

Suppose now that $X_{\Sigma}$ is a (non-necessarily projective) smooth complete variety. By Proposition \ref{chow_lemma} there exists a refinement $\Sigma'$ of $\Sigma$ such that $X_{\Sigma'}$ is a smooth projective toric variety. This refinement also induces a proper birational toric morphism $f: X_{\Sigma'} \to X_{\Sigma}$. Given a $k$-dimensional subvariety $V \subset X_{\Sigma}$, let $V' \subset X_{\Sigma'}$ be a $k$-dimensional subvariety such that $f(V') = V$. Since we still proved the theorem for projective varieties, we have: 
$$
V' \equiv \displaystyle \sum_{\omega' \in \Sigma'(n-k)} a_{\omega'}V(\omega'), \ \ \ \ a_{\omega'} \geq 0.
$$
For every $\sigma' \in \Sigma'(n-k)$, let $\sigma \in \Sigma$ be the smallest cone containing $\sigma'$. Of course the pushforward $f_*(V(\sigma'))$ is not zero if and only if $\dim(\sigma) = n-k$. In this case, since $f$ is birational, it follows from Lemma \ref{toric_morphism} that $f_*(V(\sigma')) = V(\sigma')$. Therefore, $V = f_*(V')$ is numerically equivalent to an effective linear combination of $T$-invariant $k$-cycles. The theorem is proved.

\end{proof}

\section{Log Fano Varieties}
\label{log_fano}

\noindent \textbf{Notation:} Given a $\Q$-factorial klt pair $(X, \Delta)$, throughout the next sections we denote by $(X, \Delta, X_i, \phi_i, f)$ an MMP for the pair $(X, \Delta)$:
$$
X \stackrel{\phi_1}{\dashrightarrow} X_1 \stackrel{\phi_2}{\dashrightarrow} X_2 \stackrel{\phi_3}{\dashrightarrow} ... \stackrel{\phi_r}{\dashrightarrow} X_r 
$$
that terminates with a Mori fiber space $f:X_r \to Y$. The $\phi_i's$ are the birational transformations occurring in the MMP. 

\vspace{0,4 cm}

Let $X$ be a $\Q$-factorial projective variety. We say that $X$ is \emph{log Fano} if there is a pseudo-effective $\Q$-divisor $\Delta$ on $X$ such that $-(K_X + \Delta)$ is ample and the pair $(X,\Delta)$ has klt singularities. Sometimes we write $(X,\Delta)$ to specify this $\Q$-divisor. The cones of divisors and curves of log Fano varieties are very special: they are polyhedral. In fact, it is immediate from Theorem \ref{cone_theorem} that $\overline{NE}(X)$ is polyhedral (and consequently its dual $\overline{Nef}(X)$). By \cite[Cor. 1.2]{araujo}, $\overline{Mov}(X)$ is polyhedral (and consequently its dual $\overline{Eff}(X)$). 

The next definition will be useful to exhibit a set of generators of $\overline{Mov}(X)$ when $X$ is a log Fano variety. 

\begin{de} Let $\phi: X \dashrightarrow Y$ be a birational map between $\Q$-factorial projective varieties that is surjective in codimension 1 (see Section \ref{log}). The \emph{numerical pullback of curves}, $\phi^*_{num}: N_1(Y) \to N_1(X)$, is defined to be the dual linear map of $\phi_*:N^1(X) \to N^1(Y)$.
\end{de}

Differently than the standard pullback of divisors, the numerical pull-back of curves does not have a clear geometric interpretation. In fact, $\phi^*_{num}$ is defined for classes of curves on $Y$ that are not image of any curve on $X$. There are examples of birational maps $\phi: X \dashrightarrow Y$ and classes $[C]$ of effective curves on $Y$ where $\phi^*_{num}([C]) \notin \overline{NE}(X)$ (see \cite[4.3]{araujo}). 

Next we show how to use the numerical pullback to find classes in the cone of moving curves of a projective variety.

\begin{lemma} Let $(X, \Delta)$ be a $\Q$-factorial klt pair with $K_X + \Delta \notin \overline{Eff}(X)$. By Remark \ref{mori_fiberspace}, we can run an MMP $(X, \Delta, X_i, \phi, f)$ terminating with a Mori fiber space $f:X_r \to Y$.  Let $\phi = \phi_r \circ ... \circ \phi_1$. If $C$ is a curve contained in a fiber of $f$, then $\phi^*_{num}([C]) \in \overline{Mov}(X)$. Furthermore, there exists an irreducible curve $\tilde{C} \subset X$ whose numerical class is a multiple of $\phi^*_{num}([C])$.
\label{nummerical_pull}
\end{lemma}

\begin{proof} Let $D$ be any effective Cartier divisor on $X$. Since $\overline{Eff}(X)^{\vee} = \overline{Mov}(X)$, we need to show that $[D] \cdot \phi^*_{num}([C]) \geq 0$. It follows from the definition of numerical pullback of curves that:
$$
[D] \cdot \phi^*_{num}([C]) = \phi_*([D]) \cdot [C].
$$
If $D$ is contracted by $\phi$, i.e. if $\phi_*([D]) = 0$, then $[D] \cdot \phi^*_{num}([C]) = 0$. Suppose now that $\phi_*([D]) \neq 0$. We have that $\phi_*([D])$ is the class of an effective divisor $D'$ on $X_r$. Choose a fiber $F$ of $f$ and a curve $C' \subset F$ such that $C' \nsubseteq D'$. Thus $D' \cdot C' \geq 0$. Since any two curves contracted by $f$ are numerically proportional, there exists $\lambda >0$ such that $[C] = \lambda [C']$. Thus:
$$
[D] \cdot \phi^*_{num}([C]) = \phi_*([D]) \cdot [C] = \lambda \cdot [D'] \cdot [C'] \geq 0.
$$
This proves that $\phi^*_{num}([C]) \in \overline{Mov}(X)$. 

For the second part, since the indeterminacy locus $Z$ of $\phi^{-1}$ has codimension $\geq 2$ on $X$, there exists a fiber $F$ of $f$ such that $Z \cap F$ has codimension $\geq 2$ on $F$. By Bertini's Theorem, we can obtain an irreducible curve $\hat{C} \subset F$, given as intersection of $\dim(F)-1$ very ample divisors on $F$, such that $\hat{C}$ does not intersect $Z$. Let $\tilde{C} \subset X$ be the birational transform of $\hat{C}$. We have that $\tilde{C}$ is irreducible and $[\tilde{C}] = \phi^*_{num}([\hat{C}])$. Since $C$ and $\hat{C}$ are contracted by $f$, there exists $\mu >0$ such that $[\hat{C}] = \mu [C]$. Therefore:
$$
[\tilde{C}] = \phi^*_{num}([\hat{C}]) = \mu \cdot \phi^*_{num}([C]).
$$ 
\end{proof}

\begin{rem} In the statement of the Lemma \ref{nummerical_pull}, the curve $\tilde{C}$ can be chosen satisfying $\tilde{C} \nsubseteq D$, for a fixed effective divisor $D$ on $X$. In fact, let $Z$ be the indeterminacy locus of $\phi^{-1}$. We choose a fiber $F$ of $f$ such that $F \nsubseteq \phi(D)$ and such that $Z\cap F$ has codimension 2 in $F$. By Bertini's Theorem, there exists a curve $C' \subset F$ not contained in $F \cap \phi(D)$ and with $C' \cap Z = \varnothing$. Hence, the birational transform $\tilde{C} \subset X$ of $C'$ has the required properties.

\label{add_lem_num_pull}
\end{rem}

Araujo proved in \cite[Cor. 1.2]{araujo} that if $(X, \Delta)$ is log Fano, there exist finitely many classes $\xi_1,...,\xi_s$ generating $\overline{Mov}(X)$ such that the $\xi_i$ are a numerical pullback of curves lying on a general fiber of a Mori fiber space obtained by an MMP of $(X, \Delta)$.  

Next, we observe that these classes are exactly the generators of extremal rays of $\overline{Mov}(X)$.
 
\begin{prop} Let $(X,\Delta)$ be a log Fano variety. Let $(X, \Delta, X_i, \phi_i, f)$ be an MMP terminating with a Mori Fiber Space $f:X_r \to Y$. If $\phi$ denotes the composition $\phi_r \circ ... \circ \phi_1$ and $C$ is a curve lying on a general fiber of $f$, then the numerical class $\phi^*_{num}([C])$ generates an extremal ray of the cone of movable curves $\overline{Mov}(X)$. Moreover, every extremal ray of $\overline{Mov}(X)$ is generated by a numerical pullback of a such curve.

\label{mov_rays}   
\end{prop}

\begin{proof} Of course, we can suppose that $\rho(X) \geq 2$. Two any curves contracted by $f$ are numerically proportional, hence the same holds for their numerical pullbacks. Thus, we can suppose that $C$ is a curve in a fiber of $f$ such that $\phi^*_{num}([C])$ is the class of a curve $\tilde{C}$ given by the birational transform of $C$ (see the proof of Lemma \ref{nummerical_pull}). In particular, $\phi^*_{num}([C])$ is not zero. We need to show that the half line $R = \R_{\geq 0}[\tilde{C}]$ is an extremal ray
of $\overline{Mov}(X)$. We consider two cases:

\vspace{0,2 cm}

{\it Case 1} : $\dim(Y)=0$. 

Since $f$ is a Mori fiber space, by \ref{contraction_types} we have $\rho(X_r) = 1$. Let $\phi_{i_1},...,\phi_{i_l}$ be the divisorial contractions that appear in the MMP $(X, \Delta, X_i, \phi_i, f)$. It follows from \ref{contraction_types} that $\rho(X) = l+1$ and there exist effective divisors $E_1, ..., E_l$ on $X$ contracted by $\phi$, each one corresponding to each divisorial contraction $\phi_{i_j}$. The classes $[E_1],...,[E_l]$ are linearly independent. Indeed, it follows from the fact that a divisorial contraction $\psi: Z \to Z'$ satisfies $N_1(Z) = \psi^*(N_1(Z')) \oplus \R \cdot [E]$, where $E$ is the prime divisor contracted for $\psi$. For each $i=1,...,l$, by definition of numerical pullback, we have $[E_i] \cdot [\tilde{C}] = \phi_*([E_i]) \cdot [C] = 0$. We conclude from Remark \ref{lin_equiv} that $R$ is an extremal ray of $\overline{Mov}(X)$.

\vspace{0,2 cm}
{\it Case 2} : $\dim(Y)>0$.

Let $A_Y$ be an ample divisor on $Y$ and $A := (f \circ \phi)^*[A_Y]$. Then $[A] \in \overline{Eff}(X)$ and the intersection product with $[A]$ determines a face $F$ in $\overline{Mov}(X)$. Since $\tilde{C}$ is the birational transform of $C$ and $C$ lies on a general fiber of $f$, we have that $[A] \cdot [\tilde{C}] = [A_Y] \cdot (f \circ \phi)_*[\tilde{C}] = [A_Y]\cdot f_*[C] = 0$. It follows that $[\tilde{C}] \in F$. 

To conclude the proof, it is enough to show that each extremal ray of $F$, generated by a class $\xi$, satisfies $\xi = s \cdot [\tilde{C}]$ for some positive real number $s$. Since we already have seen that the extremal rays of $\overline{Mov}(X)$ are generated by classes of curves on $X$, we can suppose that $\xi$ is a class of a curve $C' \subset X$. Moreover, $C'$ can be chosen such that it is not contained in $Exc(\phi)$ (see Remark \ref{add_lem_num_pull}). Then $\phi_*[C']$ is the class of a curve on $X_r$. Because $\xi \in F$ we have $[A] \cdot \xi = 0$. Thus:
$$
0 = [A]\cdot \xi = (f \circ \phi)^*[A_Y] \cdot [C'] = [A_Y] \cdot f_*(\phi_*[C'])
$$
It follows that $\phi_*[C']$ is the class of a curve lying on some fiber of $f$. Then $\phi_*[C'] = s[C]$ for some $s>0$ and therefore, $\xi = \phi^*_{num}(\phi_*[C']) = s \cdot \phi^*_{num}([C]) = s \cdot [\tilde{C}]$. This concludes the proof.
   
\end{proof}

The following proposition gives a lot of examples of log Fano varieties.

\begin{prop} Every $\Q$-factorial projective toric variety is log Fano.

\label{log_toric}
\end{prop}

\begin{proof} There exists an ample $\Q$-divisor $\mathcal{L} = \sum_{\rho \in \Sigma(1)}a_{\rho}D_{\rho}$ such that $a_{\rho} \in \Q \cap (0,1) \ \forall \rho$. In fact, pick any ample divisor $L$ on $X$. Since $L$ is effective and $\overline{Eff}(X)$ is generated by $T$-invariant prime divisors (see \ref{mov_toric}), we can suppose that the divisor $L$ is given by a sum $\sum_{\rho \in \Sigma(1)}c_{\rho}D_{\rho}$, $c_\rho \geq 0 \ \forall \rho$. The required divisor $\mathcal{L}$ is obtained by adding to $L$ small positive rational multiples of each $D_{\rho}$ and then dividing all the sum by a sufficiently large rational number.

We define $\Delta = \sum_{\rho \in \Sigma(1)}b_{\rho} D_{\rho}$, where $b_{\rho} = 1-a_{\rho} \in (0,1), \ \forall \rho$. It follows that   $-(K_X+\Delta) = \mathcal{L}$ is ample and by (\ref{toric_contraction}), $(X,\Delta)$ has klt singularities. Therefore $X$ is log Fano.
\end{proof}

\section{Extremal Rays of the Cone of Moving Curves}

Throughout this section, $X=X_{\Sigma}$ will be a $\Q$-factorial projective toric variety associated to an $n$-dimensional simplicial fan in $N_{\R}$.

We saw in \ref{mov_toric} that the classes $c \in N_1(X)$ can be identified with linear relations $\sum_{\rho \in \Sigma(1)}a_{\rho}v_{\rho}=0$, where $v_{\rho} \in N$ is the minimal generator of $\rho$, $a_{\rho} = D_{\rho}\cdot c$. One such relation is \emph{effective} when $a_{\rho} \geq 0, \ \forall \rho \in \Sigma(1)$ and this corresponds to a class in $\overline{Mov}(X)$.  

We begin this section observing that this interpretation of $N_1(X)$ behaves well under numerical pullback of curves when we run the MMP of $X$. We recall that if $\phi: X_{\Sigma} \dashrightarrow X_{\Sigma'}$ is a divisorial contraction or a flip of an extremal ray of $\overline{NE}(X_{\Sigma})$, we have $\Sigma'(1) \subset \Sigma(1)$. 

\begin{lemma} Let $\phi: X_{\Sigma} \dashrightarrow X_{\Sigma'}$ be a divisorial contraction or a flip associated to an extremal ray of $\overline{NE}(X_{\Sigma})$. Then given any $c \in N_1(X_{\Sigma'})$, the classes $c$ and $\phi^*_{num}(c)$ correspond to the same linear relation $\displaystyle \sum_{\rho \in \Sigma'(1)}a_{\rho}v_{\rho}=0$.

\label{le1}
\end{lemma}

\begin{proof} Given $\rho \in \Sigma'(1)$, we denote by $D'_{\rho}$ the prime divisor on $X_{\Sigma'}$ associated to $\rho$. Let $\displaystyle \sum_{\rho \in \Sigma'(1)}a_{\rho}v_{\rho}=0$ and $\displaystyle \sum_{\rho \in \Sigma(1)}b_{\rho}v_{\rho}=0$ be the linear relations corresponding to $c$ and $\phi^*_{num}(c)$ respectively. 

If $\phi: X_{\Sigma} \dashrightarrow X_{\Sigma'}$ is a flip of an extremal ray, we have $\Sigma(1) = \Sigma'(1)$ and $\phi$ is an isomorphism in codimension 1. It follows that $\phi_*([D_{\rho}])=[D'_{\rho}]$ and hence, from definition of numerical pull-back of curves:
$$
b_{\rho} = D_{\rho} \cdot \phi^*_{num}(c) = \phi_*([D_{\rho}]) \cdot c = D'_{\rho} \cdot c = a_{\rho}.
$$
Thus, the linear relation $\sum_{\rho \in \Sigma'(1)} a_{\rho}v_{\rho} = 0$ represents simultaneously the classes $c$ and $\phi^*_{num}(c)$.

Now we suppose that $\phi: X_{\Sigma} \to X_{\Sigma'}$ is a divisorial contraction. There exists a unique $\tau \in \Sigma(1)-\Sigma'(1)$ and it satisfies $\phi_*([D_{\tau}]) = 0$. Thus:
$$
b_{\tau} = D_{\tau} \cdot \phi^*_{num}(c) = \phi_*([D_{\tau}]) \cdot c = 0.
$$ 
Since $\phi$ is an isomorphism outside $D_{\tau}$, for every  $\rho \in \Sigma'(1)$ we have $\phi_*([D_{\rho}]) = D'_{\rho}$ and hence $a_{\rho}=b_{\rho}$. This concludes the proof.

\end{proof}

\begin{de} Let $v_0, ..., v_k \in N$ be primitive generators of distinct cones in $\Sigma(1)$. An effective linear relation $\displaystyle \sum_{i=0}^{i=k} a_iv_i=0$, with $a_i>0 \ \forall i$, is called \emph{minimal} when Span$(v_0, ..., v_k) \simeq \R^k$. 

\label{minimal}
\end{de}

The following is the main result of this chapter. It gives a combinatorial characterization of the extremal rays of $\overline{Mov}(X_{\Sigma})$.

\begin{teo} A non-zero effective relation among vectors in $\Sigma(1)$ belongs to an extremal ray of $\overline{Mov}(X)$ if and only if it is minimal. Moreover, each minimal effective linear relation of primitive lattice vectors in $\Sigma(1)$ determines the fan of a general fiber of a Mori fiber space obtained by an MMP of $X$.

\label{main_theorem}
\end{teo}  

\begin{proof} By Proposition \ref{log_toric}, there exists a $\Q$-divisor $\Delta$ such that $(X, \Delta)$ is a $\Q$-factorial klt pair and $-(K_X + \Delta)$ is ample. Let $c \in N_1(X)$ be a generator of an extremal ray of $\overline{Mov}(X)$. By Proposition \ref{mov_rays}, there is an MMP $(X, \Delta, X_i, \phi, f)$ that terminates with a Mori fiber space $f: X_r \to Y$ such that $c=s\phi^*_{num}([V(\omega])$, where $V(\omega)\subset X_r$ is a $T$-invariant curve contracted by $f$, $\phi = \phi_r \circ ... \circ \phi_1$ and $s>0$. It follows from (\ref{eff_comb}), Section \ref{toric_contraction}, that there exists a minimal effective linear relation: 
\begin{equation}
a_0v_0+ ... + a_{k-1}v_{k-1} + a_kv_k = 0,
\label{fiber}
\end{equation}
corresponding to $[V(\omega)]$. Since $\phi^*_{num} = {(\phi_1)}^*_{num} \circ ... \circ {(\phi_r)}^*_{num}$, by successive applications of Lemma \ref{le1}, we conclude that the numerical pull-back $\phi^*_{num}([V(\omega])$ is given by linear relation (\ref{fiber}). Therefore, $c$ corresponds to the minimal linear relation of vectors in $\Sigma(1)$:
$$
(sa_0)v_0+ ... + (sa_{k-1})v_{k-1} + (sa_k)v_k = 0
$$
Conversely, consider $c \in N_1(X)$ corresponding to the minimal linear relation of vectors in $\Sigma(1)$ given in (\ref{fiber}). Let $c_1, ..., c_l$ be non-zero classes lying in extremal rays of $\overline{Mov}(X)$ such that $c = c_1 + ... + c_l$. The vectors $v_0, ..., v_k$ are the only ones that can appear in the minimal linear relations of each $c_i$. Since Span$(v_0,...,v_k) \simeq {\R}^k$, for every $j=0, ..., k$, the set $\{v_0, ..., \hat{v}_j, ..., v_k \}$ is linearly independent. Hence, all vectors $v_0, ..., v_k$ appear in the minimal linear relation corresponding to $c_1$. Let $\displaystyle \sum_{j=0}^{k} b_jv_j=0$ be the linear relation corresponding to $c_1$. Dividing this equality by $b_k$, we have $v_k =  \displaystyle -\sum_{j=0}^{k-1} \frac{b_j}{b_k}v_j$. Using (\ref{fiber}) we have that $b_j/b_k = a_j \ \forall j$, because the set $\{v_0, ..., v_{k-1}\}$ is linearly independent. Therefore, $c = \frac{1}{b_k} \cdot c_1$ belongs to an extremal ray of $\overline{Mov}(X)$.

For the last assertion of the Theorem, let $a_0v_0+ ... + a_{k-1}v_{k-1} + a_kv_k = 0$ be a minimal effective linear relation among vectors in $\Sigma(1)$. By the first part and Lemma \ref{le1}, there exists a MMP $(X, \Delta, X_i, \phi_i, f)$ that terminates with a Mori fiber space $f: X_k = X_{\Sigma'} \to Y$ such that, up to a positive multiple, this linear relation corresponds to an invariant curve $V(\omega)$ contracted by $f$. It follows from \ref{toric_contraction} that every maximal cone of $\Sigma'$ contains a face of the form $\tau_j = \Cone(v_0, ..., \hat{v}_j, ..., v_k )$ and these are the unique cones of $\Sigma$ contained in the vector space $U= \Cone(v_0, ..., v_k)$. Let $\Sigma_0$ be the fan in $U$ whose maximal cones are the cones $\tau_j$. Note that $\Sigma_0$ depends uniquely on the vectors $v_0, ..., v_k$. This fan determines a $k$-dimensional toric variety $X_{\Sigma_0}$. The toric morphism $f$ is induced by the projection $N \to N/(U \cap N)$. The maps of the exact sequence:
$$
0 \to U \cap N \to N \to N/(U \cap N) \to 0,
$$
are compatible with the fans of the varieties $X_k$, $Y$ and $X_{\Sigma_0}$ in the sense of Lemma \ref{toric_morphism}. 
We conclude from \ref{fiber_bundle} that $f$ is a fiber bundle over the torus $T_Y$ of $Y$ and $f^{-1}(T_Y) \simeq X_{\Sigma_0} \times T_Y$. Therefore, the general fiber of $f$ is isomorphic to $X_{\Sigma_0}$. This concludes the proof. 
\end{proof}

\begin{cor} The number of extremal rays of $\overline{Mov}(X)$ is less than or equal to the number of MMP for $X$ terminating with a Mori fiber space .
\label{rays_mov_mmp}
\end{cor}

\begin{ex} Equality can be achieved in Corollary \ref{rays_mov_mmp}, for instance when $X = \mathbb{P}^n$. However, we give an example where the inequality is strict: let $H \subset \mathbb{P}^m$ be a hyperplane and $\pi: X \to \mathbb{P}^m \times \mathbb{P}^1$ be the blowup of $\mathbb{P}^m \times \mathbb{P}^1$ along $H_o = H \times \{o\}$. The pseudo-effective cone and the extremal contractions of $X$ were described in Example \ref{contra}. If $\Sigma$ denotes the fan of $X$, we have:
$$
\Sigma(1) = \{e_0, e_1, ..., e_m, e, -e, f\},
$$
where $\{e_1, ..., e_m, e\}$ is the canonical basis of $\R^m \times \R$, $e_0 = -e_1, ..., -e_m$ and $f=e_1+e$. Since $\overline{Eff}(X)$ has 3 extremal rays, the same holds for its dual $\overline{Mov}(X)$. The minimal relations generating the cone of moving curves of $X$ are:
\begin{align*}
(1) \ \ \ \ \ \ & e_0 + e_1 + ... +e_m = 0, \\
(2) \ \ \ \ \ \ & e + (-e) = 0, \\
(3) \ \ \ \ \ \ & f+ e_0 + e_2 + ... + e_m + (-e) = 0.
\end{align*}  
In fact, these relations are linearly independent because the second is not multiple of the first and the third contains the vector $f$ that does not appear in the other two relations. The fans determined by these relations are respectively those of $\mathbb{P}^m$, $\mathbb{P}^1$ and $\mathbb{P}^{m+1}$. The diagram below illustrates the possible MMP's for $X$ terminating with a Mori fiber space  (in the last line, $\mathcal{H}'$ is a linear subspace of $\mathbb{P}^{m+1}$ of codimension 2):
$$
\xymatrix{
                                &  Bl_{pt} \mathbb{P}^{m+1} \ar @{=}[d]  \ar[r] & \mathbb{P}^{m+1} \ar[r]^-{f_1} & \big\{pt \big\}  \\
                                & \mathbb{P}_{\mathbb{P}^m}(\mathcal{O} \oplus \mathcal{O}(1)) \ar[r]^-{f_2}  & \mathbb{P}^m   \\
X  \ar[ru]^-{\pi_1} \ar[r]^-{\pi_2} \ar[rd]^-{\pi_3} & \mathbb{P}^m \times \mathbb{P}^1 \ar[ru]^-{f_3} \ar[rd]^-{f_4} &  \\               
                                & \mathbb{P}_{\mathbb{P}^1}(\mathcal{O}(1) \oplus \mathcal{O}^m) \ar[r]^-{f_5} & \mathbb{P}^1   &  \\
    & Bl_{\mathcal{H}'} \mathbb{P}^{m+1} \ar @{=}[u] \ar[r] & \mathbb{P}^{m+1} \ar[r]^-{f_6} & \big\{pt \big\}.  }
$$     
Note that the minimal relation (1) determines a general fiber of the Mori fiber spaces $f_4$ and $f_5$. The relation (2) determines fibers of $f_2$ and $f_3$ and the relation (3) determines fibers of $f_1$ and $f_6$.
\end{ex}

\begin{cor} A cone $\rho \in \Sigma(1)$ is such that $[D_{\rho}]$ generates an extremal ray of $\overline{Eff}(X)$ if and only if there exist $\rho(X)-1$ minimal linear relations, linearly independent (i.e., their corresponding classes in $N_1(X)$ are linearly independent) such that $v_{\rho}$ does not appear in any of these relations.
\label{criterion_eff}
\end{cor}

\begin{proof}  
A class $c \in N_1(X)$ satisfies $[D_{\rho}] \cdot c = 0$ if and only if $v_{\rho}$ does not appear in the linear relation between the vectors in $N_1(X)$ corresponding to $c$. Suppose that there are $\rho(X)-1$ minimal linear relations linearly independent in $\overline{Mov}(X)$ such that $v_{\rho}$ does not appear in any of these relations. It follows from Remark \ref{lin_equiv} that $[D_\rho]$ generates an extremal ray of $\overline{Eff}(X)$.

Conversely, if $[D_{\rho}]$ generates an extremal ray of $\overline{Eff}(X)$, there exists a facet $F$ of $\overline{Mov}(X)$ generated by $\rho(X)-1$ linear classes $c_i$ such that $[D_\rho] \cdot c_i = 0$, $i=1, ..., \rho(X)-1$. By Theorem \ref{main_theorem}, each $c_i$ corresponds to a minimal linear relation of vectors in $\Sigma(1)$. Therefore, $v_{\rho}$ does not appear in any of these relations.  
\end{proof}

\section{An Application}
\label{losev_manin}

In this section we use Corollary \ref{criterion_eff} to describe the cone of pseudo-effective divisors of Losev-Manin moduli spaces. Part of the content of this section comes from a joint work with Edilaine Nobili.

\begin{say} \noindent \textbf{Losev-Manin moduli spaces} The study of moduli spaces is an active research area of Algebraic Geometry. Roughly speaking, a moduli space is an algebraic variety whose points correspond to algebro-geometric objects of some fixed kind, or isomorphism classes of such objects. We are interested in the so called \emph{Losev-Manin moduli spaces} $\bar{L}_{n+1}$. A point of $\bar{L}_{n+1}$ corresponds to an isomorphism class of a chain of projective lines $C = C_1 \cup ... \cup C_r$ with (not necessarily distinct) $(n+1)$ marked points $s_1, ..., s_{n+1}$ on $C$ distributed as follows: each curve $C_j$ has two distinguished points $p^-_j$ and $p^+_j$ named \emph{poles}. For $p_1^- \in C_1$ and $p_r^+ \in C_r$ we write $s_0$ and $s_{\infty}$. Two different lines $C_i$ and $C_j$ intersect only if $|i-j|=1$. The line $C_j$ intersects $C_{j+1}$ at $p^+_j = p^-_{j+1}$. The points $s_i$ are smooth points of $C$ that are different from the poles. Every line $C_j$ must contain at least one point $s_i$. 
We have that $\bar{L}_{n+1}$ provides a compactification of the moduli space $\mathcal{M}_{0,n+3}$ of smooth rational curves with $n+3$ marked points. For more details and for the proof of the existence of $\bar{L}_{n+1}$, see \cite{lm00}. \\

\begin{figure}[h]
\centering
\includegraphics[scale=0.5]{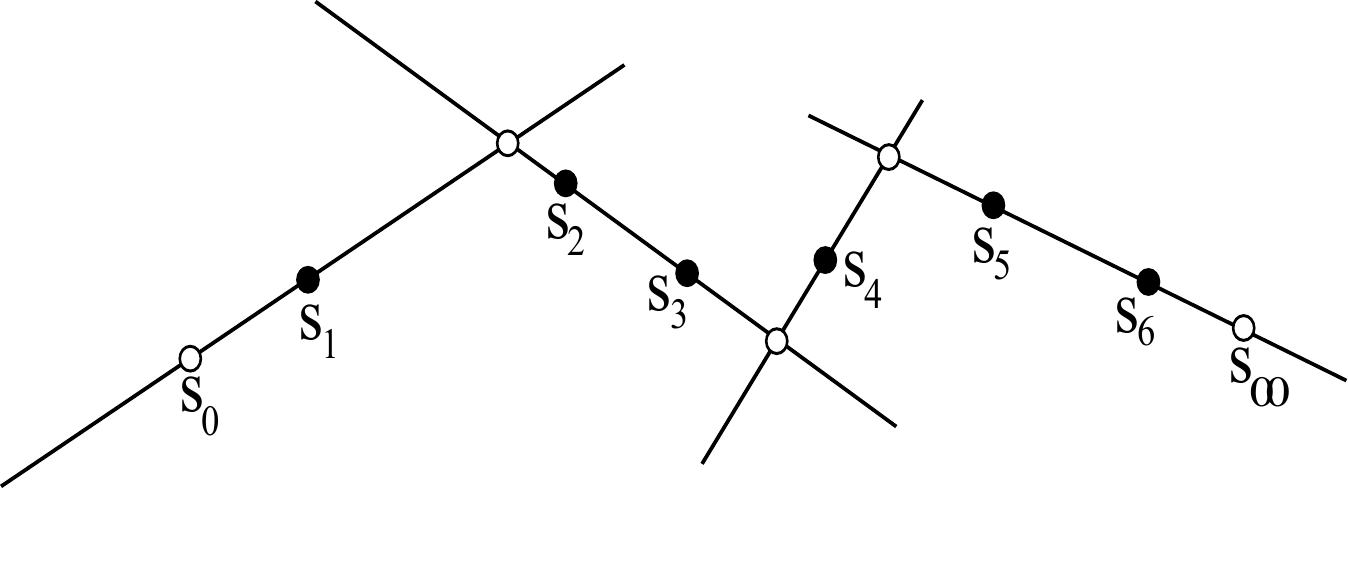}
\caption{A point of $\bar{L}_6$.}
\end{figure}

Kapranov provided in \cite[4.3]{kapranov} the following concrete realisation of $\bar{L}_{n+1}$ (and of others moduli spaces) as a successive chain of toric blowups of linear subspaces of the projective space:

\begin{enumerate}
\item[(*)] Let $q_0$, ..., $q_n$ be $n+1$ points in general position in $\mathbb{P}^n$. We blowup the points $q_i$. Then, we blowup the strict transforms of the lines $l_{ij}$ passing through $q_i$ and $q_j$. Then we blow up the strict tranforms of the $2$-planes $H_{ijk}$ passing through 3 points $q_i$, $q_j$ and $q_k$. We continue this process until all proper transform of  codimension 2 linear spaces have been blown up. 
\end{enumerate}

The smooth projective variety obtained in the process (*) is isomorphic to $\bar{L}_{n+1}$. Moreover, $\bar{L}_{n+1}$ is a toric variety.
\end{say}

Our main goal is to use Corollary \ref{criterion_eff} to prove the following:

\begin{prop} The pseudo-effective cone $\overline{Eff}(\bar{L}_{n+1})$ is polyhedral and has $2^{n+1} -2$ extremal rays. This number is exactly the quantity of $T$-invariant prime divisors on $\bar{L}_{n+1}$.
\label{losev_cone}
\end{prop}

Before the proof, we make a crucial observation:

\begin{rem} Let $X_{\Sigma}$ be a smooth projective toric variety and $\tau \in \Sigma$. Consider $\pi: \tilde{X} \to X$ the blowup of $X$ along $V(\tau)$. Recall that the fan of $\tilde{X}$ is given by the star subdivision $\Sigma(\tau)$ centered in $\tau$ (see \ref{blowup}). Since the star subdivison $\Sigma(\tau)$ changes only the cones of $\Sigma$ containing $\tau$, every proper face $\sigma \prec \tau$ is also a cone of $\Sigma(\tau)$. We denote by $V(\sigma)$ and $\tilde{V}(\sigma)$ the $T$-invariant subvarieties associated to $\sigma$ in $X$ and $\tilde{X}$ respectively. The blowup $\pi$ is the toric morphism induced by the identity $id: N \to N$ on the lattice $N$. Thus, the subvariety $\tilde{V}(\sigma)$ is mapped onto $V(\sigma)$. Since $\tilde{V}(\sigma)$ and $V(\sigma)$ are irreducible and have same dimension, we conclude that $\tilde{V}(\sigma)$ is the strict transform of $V(\sigma)$.
\label{strict_transform}
\end{rem}

To fix ideas, first we describe the cone of pseudo-effective divisors of $\bar{L}_4$. Let $\{e_1, e_2, e_3\}$ be the canonical basis of $\R^3$ and $e_0 := -e_1 -e_2 -e_3$. Let $q_0, ..., q_3$ be the four invariant points in $\mathbb{P}^3$. If $q_i$ corresponds to the cone $\Cone(e_0, ..., \hat{e_i}, .., e_3)$, by blowing up this point we introduce in the fan the new vector $u_i = \Big(\sum_{j=0}^4 e_j \Big) - e_i$. Next, let $l_{ij}$ be the invariant line passing through $q_i$ and $q_j$. The cone corresponding to $l_{ij}$ is $\Cone(e_k \big| k \neq i,j)$ and by Remark \ref{strict_transform}, this is also the cone of the strict transform of this line in the blowup of $\mathbb{P}^3$ at the points $q_0, ..., q_3$. Blowing up this strict transform, we add to the fan the vector $u_{ij} = \Big(\sum_{j=0}^4 e_j \Big) - e_i - e_j$. It follows that:
$$
\Sigma_{\bar{L}_4}(1) = \big\{e_0, e_1, e_2, e_3, u_0, u_1, u_2, u_3, u_{01}, u_{02}, u_{03}, u_{12}, u_{13}, u_{23} \big\}.
$$
The Picard number of $\bar{L}_4$ is $\rho(\bar{L}_4) = \# \Sigma_{\bar{L}_4}(1) - 3 = 11$. Consider the following effective linear relations:
\begin{align}
 & e_0+e_1+e_2+e_3 = 0, \notag \\ 
 & {u_i + e_i = 0, \ \ i=0,...,3,  \label{11_relations}} \\ 
 & u_{ij} + e_i + e_j = 0, \ \ i,j=0,...,3, \ \ i\neq j. \notag \\ \notag
\end{align}
Observe that $T$-invariant prime divisors corresponding to any two vectors of $\Sigma_{\bar{L}_4}(1)$ are not numerically equivalent. It is immediate that these relations are minimal. Furthermore, they are linearly independent because, except for the first, each relation contains a vector that does not appear in the others relations. Moreover, they form a basis for $N_1(\bar{L}_4)$ because there are 11 relations in (\ref{11_relations}). By Corollary \ref{criterion_eff}, the classes of the $T$-invariant prime divisors $V(u_i)$ and $V(u_{ij})$ generate extremal rays of $\overline{Eff}(\bar{L}_4)$. There are 10 such divisors. Since $\overline{Eff}(\bar{L}_4)$ is a maximal cone in $N^1(\bar{L}_4)$ and the dimension of this vector space is 11, some $T$-invariant divisor $V(e_j)$ must determine an extremal ray. By Corollary \ref{criterion_eff}, there are 10 minimal linear relations among the vectors of $\Sigma_{\bar{L}_4}(1)$ in which the vector $e_j$ does not appear. These 10 relations are linear combinations of the relations in (\ref{11_relations}). Notice the symmetry of (\ref{11_relations}) on the indices 0, 1, 2, 3. It follows that all the $T$-invariant divisors corresponding to the vectors $e_0$, ..., $e_3$ generate extremal rays of $\overline{Eff}(\bar{L}_4)$ and this cone has $\# \Sigma_{\bar{L}_4}(1) = 2^4 - 2 = 14$ extremal rays.  

\vspace{0,5 cm}

Next, we treat the general case. Let $\{e_1, ..., e_n\}$ be the canonical basis of $\R^n$ and $e_0 := -e_1 -...-e_n$. Let $q_0, ..., q_n$ be the $T$-invariant points of $\mathbb{P}^n$, each $q_j$ corresponding to the cone $\Cone(e_0, ..., \hat{e}_j, ..., e_n)$. For every positive integer $r \leq n-1$ the number of $(r-1)$-planes passing through $r$ of the points $q_0, ..., q_n$ is equal to ${n+1 \choose r}$. For every $\alpha = \{i_1<i_2<...<i_r\} \subset \{0,1,...,n\}$, $1 \leq r \leq n-1$ the $(r-1)$-plane passing through the points $\{q_j \ | \ j \in \alpha \}$ is the $T$-invariant subvariety of $\mathbb{P}^n$ corresponding to the cone $\Cone \big(e_j \big| j \notin \alpha \big)$. We define:
$$
u_{\alpha} :=  \Big(\sum_{j=0}^n e_j \Big) - \Big(\sum_{j \in \alpha} e_j \Big).
$$ 
Thus, the vectors appearing in $\Sigma_{\bar{L}_{n+1}}(1)$ are the canonical vectors $e_j$ and the vectors $u_{\alpha}$ with $\alpha \subset \{0,1,...,n\}$, $1 \leq \# \alpha \leq n-1$. We have:
$$
\rho(\bar{L}_{n+1}) = \# \Sigma_{\bar{L}_{n+1}}(1) - n =  \displaystyle \sum_{r = 1}^n {n+1 \choose r} - n = 2^{n+1} - (n+2).
$$
Consider the following linear relations among the vectors of $\Sigma_{\bar{L}_{n+1}}(1)$:

\begin{align}
& e_0 + e_1 + ... + e_n = 0, \notag \\
& {\label{minimal_example}} \\
& u_{\alpha} + \Big(\sum_{j \in \alpha} e_j \Big) = 0, \ \ \ \alpha \subset \{0,...,n\},  \  1 \leq \# \alpha \leq n-1. \notag
\end{align}
There are $\rho(\bar{L}_{n+1})$ relations in (\ref{minimal_example}). They are minimal and linearly independent, since each such relation, except for the first, has a vector that does not appear in the others. Moreover, two vectors in $\Sigma_{\bar{L}_{n+1}}(1)$ determine $T$-invariant prime divisors whose numerical classes are distinct. By Corollary \ref{criterion_eff}, for every $\alpha \subset \{0,...,n\}$, $1 \leq \# \alpha \leq n-1$,  $[V(u_{\alpha})]$ generates an extremal ray of $\overline{Eff}(\bar{L}_{n+1})$. Since there are $\rho(\bar{L}_{n+1}) - 1$ vectors $u_{\alpha}$ and $\overline{Eff}(\bar{L}_{n+1})$ is a cone of dimension $\rho(\bar{L}_{n+1})$, there exists a vector $e_j$ such that $[V(e_j)]$ generates an extremal ray of this cone. By Corollary \ref{criterion_eff}, there are $\rho(X)-1$ minimal linear relations among vectors of $\Sigma_{\bar{L}_{n+1}}(1)$ in which the vector $e_j$ does not appear. These relations are linear combinations of the basis (\ref{minimal_example}). Notice the symmetry on the indices of the vectors of (\ref{minimal_example}). It follows that all the classes of $T$-invariant divisors $V(e_0), ..., V(e_n)$ generate extremal rays of $\overline{Eff}(\bar{L}_{n+1})$ and this cone has $\# \Sigma_{\bar{L}_{n+1}}(1) = 2^{n+1} - 2$ extremal rays. 

\vspace{30 cm}
\


\vspace{30 cm}

\addcontentsline{toc}{chapter}{Bibliography}
\begin{thebibliography}{99}

\bibitem[Ara10]{araujo} - C. Araujo {The cone of pseudo-effective divisors of log varieties after Batyrev} -  Mathematische Zeitschrift 264 (2010) pp. 179-193.

\bibitem[Bar07]{bar} S. Barkowski - {\it The cone of moving curves of a smooth Fano-threefold} -  (2007), arXiv:math/0703025.

\bibitem[BCHM10]{bchm} C. Birkar, P. Cascini, C. Hacon, J. McKernan - {\it Existence of minimal models for varieties of log general type} - J. Amer. Math. Soc. 23 (2010), pp. 405-468

\bibitem[BDPP04]{bdpp} S. Boucksom, J-P. Demailly, M. Paun, T. Peternell - {\it The pseudo-effective cone of a compact K\"ahler manifold and varieties of negative Kodaira dimension} - math.AG/0405285 (2004, preprint).

\bibitem[BN07]{baty} V. V. Batyrev, B. Nill - {\it Multiples of lattice polytopes without interior lattice points} - Mosc. Math. J. 7(2) (2007) pp. 195-207.

\bibitem[BR07]{beck_robins} - M. Beck, S. Robins - {\it Computing the continuous discretely: Integer point enumeration in polyhedra} - Undergraduate Texts in Mathematics, Springer-Verlag (2007).

\bibitem[BS95]{beltrametti} M. C. Beltrametti, A. J. Sommese - {\it The Adjunction Theory of Complex Projective Varieties} - De Gruyter Expostitions in Mathematics, vol. 16,  Walter de Gruyter - Berlin - New York (1995).

\bibitem[BSW92]{bsm} M. C. Beltrametti, A. J. Sommese, J. Wi\'sniewski - {\it Results on varieties with many lines and their applications to adjunction theory} - Lecture Notes in Math., vol. 1507, Springer-Verlag, Berlin (1992) pp. 16-38.

\bibitem[CCD97]{ccd} E. Cattani, D. Cox, A. Dickenstein - {\it Residues in toric varieties} - Compositio Mathematica Volume 108, Number 1, (1997) pp. 35-76.

\bibitem[CLS11]{cox} D. A. Cox, J. B. Little, H. N. Schenck {\it Toric Varieties} - Graduate Studies in Mathematics Vol. 124 AMS (2011)

\bibitem[DDRP09]{alicia} A. Dickenstein, S. Di Rocco, R. Piene - {\it Classifying smooth lattice polytopes via toric fibrations} - Advances in Math. 222, Science Direct, (2009) pp.240-254.

\bibitem[DELV11]{voisin} O. Debarre, L. Ein, R. Lazarsfeld, C. Voisin - {\it Pseudoeffective and Nef Classes on Abelian Varieties} - Compositio Mathematica (2011) 147, pp 1793-1818.

\bibitem[DN10]{alicia2} A. Dickenstein, B. Nill - {\it A simple combinatorial criterion for projective toric manifolds with dual defect} - Math. Res. Lett. 17(3) (2010) pp. 435-448.

\bibitem[DR06]{dirocco} S. Di Rocco - {\it Projective duality of toric manifolds and defect polytopes} - Proc. London Math. Soc. (3) 93 (1) (2006) pp. 85-104.

\bibitem[DRHNP10]{drhnp}  S. Di Rocco, C. Hasse, B. Nill, A. Paffenholz - {\it Polyhedral Adjunction Theory} - (2011) preprint: \texttt{arXiv:1105.2415}.

\bibitem[DRS04]{drso} S. Di Rocco, A. J. Sommese - {\it Chern Numbers of Ample Vector Bundles on Toric Surfaces} - Trans. Amer. Math. Soc. 356 (2004), 587-598.

\bibitem[Ful93]{fulton} W. Fulton - {\it Introduction to Toric Varieties} - Princetown University Press (1993).

\bibitem[Ful98]{fulton_int} W. Fulton - {\it Intersection Theory} - Springer-Verlag (1998)

\bibitem[Ha92]{harris} J. Harris - {\it Algebraic Geometry. A First Course} - Springer-Verlag (1992)

\bibitem[Har70]{hartshorne2} R. Hartshorne - {\it Ample Subvarieties of Algebraic Varieties} - Springer-Verlag (1970).

\bibitem[Har77]{hartshorne} R. Hartshorne -{\it Algebraic Geometry} - Springer-Verlag (1977)

\bibitem[HNP09]{hnp} C. Haase, B. Nill, and S. Payne - {\it Cayley decompositions of lattice polytopes and upper bounds for {$h^*$}-polynomials} - J. Reine Angew. Math., 637 (2009) pp. 207-216.

\bibitem[Kap93]{kapranov} M. M. Kapranov - {\it Chow quotients of Grassmannians I} - I. M. Gelfand Seminar, Adv. Soviet Math., 16, Part 2, Amer. Math. Soc., Providence, RI, pp.29-110 (1993).

\bibitem[Kl66]{kleiman} S. Kleiman - {\it Toward a numerical theory of ampleness} - Ann. of Math. (2) 84 (1966) pp. 293-344.

\bibitem[KM98]{km98} J. Koll\'ar, S. Mori - {\it Birational Geometry of algebraic varieties} - Cambridge University Press (1998).

\bibitem[La04]{lazarsfeld} R. Lazarsfeld - {\it Positivity in Algebraic Geometry I, II} - Ergebnisse der Mathematik 48,
49, Springer-Verlag (2004).

\bibitem[LM00]{lm00} A. Losev, Y. Manin - {\it New moduli spaces of pointed curves and pencils of flat connections} -  Michigan Math. J. Volume 48, Issue 1 (2000), 443-472.

\bibitem[Mo88]{mori} S. Mori - {\it Flip theorem and the existence of minimal models for 3-folds} - J.
Amer. Math. Soc. 1 (1988), no. 1, pp. 117-253

\bibitem[No11]{eide} E. Nobili - {\it Classification of Toric 2-Fano 4-Folds} - Bulletin of Brazilian Mathematical Society, Volume 42, Number 3 (2011) pp. 399-414.

\bibitem[No12]{tese_edilaine} E. Nobili - {\it Birational Geometry of Toric Varieties} (2012) - arXiv:1204.3883.

\bibitem[Re83]{reid} M. Reid - {\it Decompostition of toric morphisms} - Arithmetic and Geometry, vol. II, in Progr. Math., vol 36, Birk\"auser Boston, MA, (1983) pp. 395-418.

\bibitem[Sha96]{shafa} I. Shafarevich - {\it Basic Algebraic Geometry} - Volumes 1 and 2, Springer, New York, 1994 and 1996.

\bibitem[Wis91]{wis} J. A. Wi\'sniewski - {\it On Fano Manifolds of Large Index} - Manuscripta Math. 70, Springer-Verlag, pp. 145-152 (1991).

\bibitem[Z95]{z} - G. Ziegler - {\it Lectures on Polytopes} - Springer-Verlag, Berlin Heidelberg New York,
(1995).

 


   



























 





 



\end{thebibliography}
\end{document}